\documentclass[11pt]{article}
\usepackage{setspace}
\doublespace
\usepackage{authblk}
\usepackage{amsmath,amsthm,amssymb,color,array,multirow}
\usepackage{epsfig,amsfonts,latexsym, hyperref,  mathrsfs}
\usepackage{natbib}
\usepackage{centernot}
\usepackage{stmaryrd}
\usepackage{geometry}
\usepackage{tikz}
\usetikzlibrary{shadows,trees}

\hypersetup{
    colorlinks=true,
    linkcolor=red,
    citecolor=blue,
    urlcolor=blue,
    pdfborder={0 0 0}
}

\newtheorem{remark}{Remark}[section]
\newtheorem{theorem}{Theorem}[section]

\newtheorem{prop}[theorem]{Proposition}
\newtheorem{lemma}{Lemma}[section]

\newtheorem{cor}{Corollary}[section]
\theoremstyle{definition}

\newtheorem{Example}{Example}[section]

\numberwithin{equation}{section}
\numberwithin{table}{section}

\allowdisplaybreaks

\title{Change Point Estimation in Panel Data with Temporal and Cross-sectional Dependence}

\date{\today}

\author[1]{Monika Bhattacharjee}
\author[2]{Moulinath Banerjee}
\author[3]{George Michailidis}
\affil[1]{Informatics Institute, University of Florida, USA}
\affil[2]{Department of Statistics, University of Michigan, USA}
\affil[3]{Department of Statistics \& Informatics Institute, University of Florida, USA}

\begin{document}

\maketitle

\begin{abstract}   
We study the problem of detecting a {\em common} change point in large panel data based on a mean shift model, wherein the errors exhibit both temporal and cross-sectional dependence. A least squares based procedure is used to estimate the location of the change point. Further, we establish the convergence rate and obtain the asymptotic distribution of the least squares estimator. The form of the distribution is determined by the behavior of the norm difference of the means before and after the change point. Since the behavior of this norm difference is, a priori, unknown to the practitioner, we also develop a novel data driven adaptive procedure that provides valid confidence intervals for the common change point, without requiring any such knowledge. Numerical work based on synthetic data illustrates the performance of the estimator in finite samples under different settings of temporal and cross-sectional dependence, sample size and number of panels. Finally, we examine an application to financial stock data and discuss the identified change points.
\end{abstract}
\medskip

   \medskip
 
 \bigskip

\section{Introduction} \label{sec: intro}

The change point problem for univariate data has a long history in the econometrics and statistics literatures. A broad overview of the technical aspects of the problem is provided in \citet{Basseville1993detection, csorgo1997limit}. The problem has a wide range of applications in economics \citep{baltagi2016estimation, RePEc:siu:wpaper:07-2015, doi:10.1080/01621459.2015.1119696} and finance \citep{frisen2008financial}, while other standard areas include quality monitoring and control \citep{qiu2013introduction}, as well as newer ones, like genetics and medicine \citep{chen2011parametric} and neuroscience \citep{koepcke2016single}. 

In many of these applications, the multivariate (panel) data streams exhibit both temporal, as well as cross-sectional dependence, since they reflect different facets of coordinated activity -e.g. stock price co-movements, cross-talk amongst brain regions and co-expression of members of genetic regions.

The rather limited technical literature on change point analysis for panel data focuses on the common break signal plus noise model given by
\begin{eqnarray}
\label{model.defn}
X_{it} & = & \mu_{i1}+\epsilon_{it}, \ \ \ t=1,2,\cdots, [n\tau_n] \\ \nonumber
X_{it} & = & \mu_{i2}+\epsilon_{it}, \ \ \ t=[n\tau_n]+1,\cdots,n \\
       &    & i=1,\cdots,p, \nonumber
\end{eqnarray}
where $\tau_n$ represents a common break fraction for all $p$ series (streams), the difference $|\mu_{i1}-\mu_{i2}|$ represents the magnitude of the shift for each series, and $\epsilon_{it}$ are random error processes. The primary objective is to estimate the location of change point $\tau_n$, as well as the levels of the series before and after it.

\noindent
{\bf Literature review:}
Different aspects of change point analysis have been studied in the literature for the aforementioned model, wherein the
random process $\{\epsilon_{it}\}$ exhibits {\em only} temporal dependence (assumed cross-sectionally independent). 
\citet{bai2010common} employed a least squares criterion to estimate the common change point $\tau_n$ and established its asymptotic distribution, while \citet{horvath2012change} developed tests for the presence of a
change point during the observation period. \citet{pevstova2017change} and \citet{bardwell2018most} provided a method to detect a common break point even when the change happens immediately after the first time point or just before the last epoch for panels with limited number of time points. \citet{cho2015multiple} segmented the second-order structure of a high-dimensional time series and used the CUSUM statistic to detect multiple change points. Further,
\citet{kim2014common, baltagi2016estimation, barigozzi2018simultaneous,  westerlund2018common} investigated estimation of the change point in panel data, wherein the cross-sectional dependence is modeled by a common factor model, which effectively
makes the cross-sectional dependence low-dimensional. 

However, as previously argued, often both temporal and cross-sectional dependence is present in panel data under the signal plus noise model. To the best of our knowledge,  \citet{cho2016change} represents the only work on change point analysis in such a setting and investigates both single and multiple change-point detection. The nature of the cross-sectional dependence is general, while geometrically decaying $\alpha$-mixing is assumed across time, and the number of series $p$ can grow at a polynomial rate in the number of time points $n$. Another work that considers both temporal and cross-sectional dependence is by \citet{SS2017}, that examines change point analysis for sparse high-dimensional vector autoregressive models.

\noindent
{\bf Key contributions:} In this paper, we consider the problem of single change point detection in high-dimensional panel data using a least squares criterion, where the temporal and cross-sectional dependence are captured through an infinite order moving average process (MA($\infty$)). We establish the convergence rate (Theorem \ref{thm: conrate}) and derive the asymptotic distribution (Theorem \ref{thm: asympdist}) for the least squares estimator of the change point. Further, since there are {\em multiple regimes} for the asymptotic distribution of the change point estimate determined by the underlying 
{\em unknown} signal-to-noise ratio, we also provide a {\em self-adaptive, data driven method} for computing confidence intervals of the change point that does not require us to know the specific regime.

Note that this work extends the analysis in \citep{bai2010common} to the case where the multivariate MA($\infty$) error process is \emph{correlated} across its coordinates. Such correlations introduce a number of technical challenges that are successfully resolved in Theorems \ref{thm: conrate} and \ref{thm: asympdist}. Another broadly related work is that of
\citet{cho2016change}. For the single change point analysis that is the focus of this paper, we note that our rate result
is obtained under weaker detectability conditions for a much larger class of error processes and allowing faster growth of the time series $p$ as a function of the number of time points $n$. Further, the obtained rate of the change point is sharper and in addition we derive its asymptotic distribution - a more detailed discussion of these points are provided in Remark \ref{rem: cho} and Section \ref{sec: simulation}. On the other hand, the latter paper develops methodology for detection of multiple change points, which is outside the scope of the current work.  

The remainder of the paper is organized as follows. In Section \ref{sec: model}, we describe the signal-plus-noise model exhibiting a single change point in its mean structure, with temporal and cross-sectional dependence introduced through a vector moving average process.   
In Section \ref{sec: results}, we define the least squares estimator for the change point and establish its convergence rate and asymptotic distribution in Theorems \ref{thm: conrate} and \ref{thm: asympdist}, respectively.  Further, we discuss that the assumptions required for these theorems hold under very mild conditions for certain illustrative examples considered in Section \ref{subsec: example} which are employed often in practice. In Section \ref{sec: adap}, we propose a data based adaptive inference scheme  for obtaining the asymptotic distribution of the change point estimate in practice which does not require prior knowledge on the signal-to-noise regime. Performance evaluation results based on synthetic data are presented in Section \ref{sec: simulation}.  Finally, an application of the proposed methodology to financial data is discussed in Section \ref{sec: Dataanalysis}.  Additional technical details and all proofs are delegated to an Appendix - Sections \ref{sec: appA} and \ref{sec: proof}, respectively.

\subsection{Modeling Framework} \label{sec: model}

We observe $\{X_{t,p(n)}^{(n)}:\ 1 \leq t \leq n\}$ (dimension $p = p(n)$ depends on sample size $n$) from the following model:  
\begin{eqnarray}
X_{t,p(n)}^{(n)} &=& \mu_{1,p(n)}I(t \leq [n\tau_n]) +\mu_{2,p(n)}I(t > [n\tau_n])  + \varepsilon_{t,p(n)}^{(n)}\ \ \ \text{where} \label{eqn:model} \\
\varepsilon_{t,p(n)}^{(n)} &=& \sum_{j=0}^{\infty} A_{j,p(n)}^{(n)}  \eta_{t-j,p(n)}. \nonumber
\end{eqnarray}
Here, $\tau_n$ is the change point and 
$\mu_{1,p(n)} = (\mu_{11,p(n)},\mu_{12,p(n)},\ldots,\mu_{1p(n),p(n)})^\prime$, $\mu_{2,p(n)} = (\mu_{21,p(n)},\mu_{22,p(n)},$ $\ldots,\mu_{2p(n),p(n)})^\prime$ are $p(n)$-dimensional mean vectors before and after the change point.  All processes
$X_{t,p(n)}^{(n)} = (X_{1t,p(n)}^{(n)}, $ $X_{2t,p(n)}^{(n)},\ldots, X_{p(n)t,p(n)}^{(n)})^\prime$,  $\varepsilon_{t,p(n)}^{(n)} = (\varepsilon_{1t,p(n)}^{(n)},\varepsilon_{2t,p(n)}^{(n)},\ldots,\varepsilon_{p(n)t,p(n)}^{(n)})^\prime$, and  $\eta_{t,p(n)} = (\eta_{1t,p(n)},\eta_{2t,p(n)},$ $\ldots,\eta_{p(n)t,p(n)})^\prime$ correspond to $p(n)$-dimensional random vectors. Notation-wise we often write $p$ instead of $p(n)$, when there is no room for confusion.  Further, $\{\eta_{kt,p}:\ k,t,p \geq 1\}$ are \textit{i.i.d.}~ mean $0$, variance $1$ random variables  with finite 4th moments; i.e. $E|\eta_{kt,p} |^4 < \infty$. Finally, the coefficient matrices  $A_{j,p}^{(n)}$ correspond to $p \times p$ deterministic matrices. We assume for all $t \geq 1$ and $j \geq 0$,  $\{\eta_{t,p}:\ p \geq 1\}$ 
are nested. That means the first $p$ components of $\eta_{t,p+1}$ is $\eta_{t,p}$. 
Clearly, $X_{t,p(n)}^{(n)}$ and $\varepsilon_{t,p(n)}^{(n)}$ are not nested, but they form triangular sequences.  We consider $p = p(n) \to \infty$ as $n \to \infty$. For expository clarity, we shall use $\mu_1, \mu_2, X_t, X_{kt}, \epsilon_t, \varepsilon_{kt}, \eta_t, \eta_{kt}, \mu_{1k}, \mu_{2k},  A_j$ for $\mu_{1,p(n)}$, $\mu_{2,p(n)}$,  $X_{t,p(n)}^{(n)}$, $X_{kt,p(n)}^{(n)}$, $\varepsilon_{t,p(n)}^{(n)}$, $\varepsilon_{kt,p(n)}^{(n)}$, $\eta_{t,p(n)}$, $\eta_{kt,p(n)}$,  $\mu_{1k,p(n)}$, $\mu_{2k,p(n)}$, $A_{j,p}^{(n)}$, respectively. 
\vskip 2pt

The objective is to estimate the change point $\tau_n$, together with all other model parameters  $[\mu_1, \mu_2\ \text{and}\  \{A_j: j \geq 0\}]$. 

In the above model,  the data stream process $\{X_t\}$ exhibits dependence across both time and co-ordinates. Since 
$\{\varepsilon_t\}$ is a stationary process, the covariance between $\varepsilon_t$ and $\varepsilon_{t+h}$ depends only on lag $h$. Further, the population autocovariance matrix of order $u$ is given by
\begin{eqnarray}
\Gamma_h := E(\varepsilon_t \varepsilon_{t+h}^\prime) = \sum_{j=0}^{\infty} A_j A_{j+h}^{\prime}. \nonumber 
\end{eqnarray}
Note that $\Gamma_h = \Gamma_{-h}^{\prime}$ are all $p \times p$ matrices.  If $A_j =0\ \forall\ j \geq 1$ then $\Gamma_h =0\ \forall\ h \neq 0$ and we have independence across time. On the other hand, if $A_j$ are all diagonal matrices, then so are 
the $\Gamma_h$ ones, which in turn implies component-wise independence. In this paper, we assume that the coefficient matrices have a general form. 

This model for $\{\varepsilon_{t,p(n)}^{(n)}\}$ has attracted much attention in the literature and various aspects of it have been studied, including estimating the autocovariance matrices (\cite{BB2013a}), studying properties of the spectrum of the sample autocovariance matrices (\cite{LAP2015}, \cite{BB2016}, \cite{WAP2017}), testing for the presence of trends (\cite{CW2018}) and so forth.  
Further, \citet{bai2010common} and \citet{BBM2017} obtained a consistent estimator of the change point when all $A_j$ are diagonal matrices. In this work, we extend the analysis to general coefficient matrices.

\section{The change point estimator and its asymptotic properties} \label{sec: results}
Next, we propose an estimator of the change point and study its asymptotic properties. Since the change point in the model
is driven by changes in the mean structure, we employ a least squares criterion for the task at hand. Specifically,
the estimate $\tau_n$ is obtained by
\begin{eqnarray}
\hat{\tau}_n &=& \arg\min_{b \in (c^*, 1-c^*)} L(b)\ \ \ \text{where} \label{eqn: cpdefine} \\
L(b) &=& \frac{1}{n}\sum_{k=1}^{p}\bigg[\sum_{t=1}^{nb}(X_{kt}-\hat{\mu}_{1k}(b))^2  + \sum_{t=nb+1}^{n}(X_{kt}-\hat{\mu}_{2k}(b))^2\bigg],  \nonumber \\
\hat{\mu}_{1k}(b) &=& \frac{1}{nb}\sum_{t=1}^{nb} X_{kt},\ \ \hat{\mu}_{2k}(b) = \frac{1}{n(1-b)} \sum_{t=nb+1}^{n} X_{kt}. \nonumber
\end{eqnarray}

The first result established is that of the rate of convergence of $\hat{\tau}_n$ in Theorem \ref{thm: conrate}. To that end, let $||\cdot||_2$ denote the spectral norm of a matrix. Further, define $\gamma_p = \sum_{j=0}^{\infty} ||A_j||_2$.  Note that time dependence in $\{X_t\}$ is characterized by the variation of the coefficient matrices $\{A_j\}$ across $j$. Hence, the aggregate $\gamma_p$ provides a measure of cross-sectional dependence in $\{X_t\}$, but not of temporal dependence.  We need the following signal-to-noise (SNR) condition for establishing the convergence rate of $\hat{\tau}_n$. 
\vskip 2pt
\noindent \textbf{(SNR)} $\frac{n\gamma_p^{-2}}{p} ||\mu_1-\mu_2||_2^2 \to \infty$
\vskip 2pt
\noindent Note that $\frac{1}{p} ||\mu_1-\mu_2||_2^2$ is the average signal per model parameter, which drives the occurrence of the change point. The SNR condition intuitively states that the average signal needs to grow faster than  $\frac{\gamma_p^2}{n}$. If the coefficient matrices $\{A_j: j \geq 0\}$ satisfy   $\gamma_p = O(1)$, then the SNR condition reduces to 
\newline
\textbf{(SNR*)}  $\frac{n}{p} ||\mu_1-\mu_2||_2^2 \to \infty$; 
\newline
i.e., the average signal needs to grow faster than $\frac{1}{n}$,  
which is similar to the identifiability conditions in other change point problems that exhibit independence both across
time and across coordinates (see e.g. \cite{BBM2017}).  Examples of  coefficient matrices satisfying $\gamma_p = O(1)$ are discussed in Sections \ref{cor: panelind}, \ref{cor: paneltimeind} and Example \ref{example: 1}(ii). Often $\gamma_p = O(1)$ may not hold, as shown in Examples \ref{example: 1}(i),  \ref{example: band} and \ref{example: band}. In the latter case, the SNR condition is required, which is stronger than the SNR*.  Finally, note that the SNR condition depends only on the total number of time points observed, but not on the nature of the temporal dependence of $\{X_{t}\}$. 

\begin{theorem}  \label{thm: conrate}
Suppose (SNR) holds. Then
$n \gamma_p^{-2} ||\mu_1  - \mu_2||_2^2 (\hat{\tau}_n -\tau_n) = O_{\text{P}}(1)$.
\end{theorem}
If $\gamma_p = O(1)$, then the convergence rate of $\hat{\tau}_n$ is the same as that under independence across both time and panels. For details see Section \ref{cor: paneltimeind}.  However, if $\gamma_p$ grows with $p$, the the convergence rate in Theorem \ref{thm: conrate} is compromised: this is the price paid for growing cross-sectional dependence.  

\begin{remark} \label{rem: baigammamain}
It is easy to see from the proof of Theorem \ref{thm: conrate} and Remark \ref{rem: baigamma} that under cross-sectional independence -i.e. when $A_{j,p} = \text{Diag}\{a_{j,k}:\ 1\leq k \leq p\}\ \forall j\geq 0$-  $\gamma_p$ in SNR can be replaced by a smaller quantity $\tilde{\gamma}_p := \sup_{1\leq k \leq p} \sum_{j=0}^{\infty}|a_{j,k}|$.  As a consequence, when $A_{j,p} = \text{Diag}\{a_{j,k}:\ 1\leq k \leq p\}\ \forall j\geq 0$, the conclusion of Theorem \ref{thm: conrate} continues to hold under the weaker \textbf{(SNR$^\prime$)} condition $\frac{n\tilde{\gamma}_p^{-2}}{p} ||\mu_1-\mu_2||_2^2 \to \infty$.  Further, if $\tilde{\gamma}_p = O(1)$, then SNR$^\prime$ reduces to SNR*. 
\end{remark}

\begin{remark} \label{rem: cho}
As mentioned in the Introduction, \citet{cho2016change} considered the problem of single, as well as multiple change-point detection in high dimensional panel data under general cross-sectional dependence, but for processes that exhibit geometrically decaying $\alpha$-mixing across time. In general, a multivariate MA($\infty$) process is not a geometrically decaying $\alpha$-mixing process (e.g. consider the case of polynomially decaying coefficients). Thus, there are many linear  processes which are eligible under our setting, but can not be accommodated by \cite{cho2016change}. 

Focusing on the results pertaining to a single change point, note that if the coefficients of the MA($\infty$) process are geometrically decaying -e.g. autoregressive (AR) or autoregressive moving average (ARMA) processes-, then it becomes a geometrically decaying $\alpha$-mixing process.  Although AR and ARMA processes can be considered as special cases of the model in \citet{cho2016change}, our results do not follow from those in that paper. The reason is that we employ a least squares based estimator, while that paper considers an estimator derived from a CUSUM statistic. The latter estimator requires the following identifiability condition $\frac{\sqrt{n}}{p \log n}\sum_{k=1}^{p} |\mu_{1k}-\mu_{2k}| \to \infty$  which is stronger than the posited 
(SNR) above. In addition, the obtained convergence rate for the change point estimator is $(\log n)^2 \frac{p}{n} (\sum_{k=1}^{p} |\mu_{1k}-\mu_{2k}|)^{-2}$, which is also slower than the convergence rate of the least squares estimator for $\gamma_p=O(1)$. Moreover, \citet{cho2016change} assumed that all moments of $\{X_t\}$ are finite, whereas we only require finite $4$-th order moments of $\{X_t\}$. Finally, we allow the dimension of the data stream to grow as $\log p \sim n^\delta$ for $\delta \in (0,1)$, vis-a-vis the $p \sim n^\delta$ for some fixed $\delta \in [0,\infty)$ required in the aforementioned paper. 
\end{remark}

\subsection{Asymptotic distribution of the least squares estimator}

Unlike the convergence rate result that can be established under the SNR condition, the derivation of the asymptotic distribution for $\hat{\tau}_n$ is significantly more involved, as presented next. We start by noting that in the panel data setting, the asymptotic distribution of $\hat{\tau}_n$ differs, depending on the following regimes: (I) $\lim_{n \to \infty}  ||\mu_1 - \mu_2||_2 \to \infty$, (II) $\lim_{n\to\infty}  ||\mu_1 - \mu_2||_2 \to 0$ and (III) $\lim_{n\to\infty}  ||\mu_1 - \mu_2||_2 \to c >0$ (see \citet{bai2010common}). 
\newline \indent Recall that
in the presence of a single panel (univariate case), the following two results have been established in the literature:  (i) if $|\mu_1-\mu_2|\rightarrow 0$ (Regime (II)) at an appropriate rate as a function of the sample size $n$, then the asymptotic distribution of the change point is given by the maximizer of a Brownian motion with triangular drift  (for details see \cite{bhattacharya1994}); and (ii) if $|\mu_1-\mu_2|\rightarrow c$ (Regime (III)), then the
asymptotic distribution of the change point, in the random design setting, is given by the maximizer of a two-sided compound Poisson process (for details see Chapter 14 of the book by \citet{kosorok2008}).  As previously mentioned and will be established rigorously next, in the panel data setting analogous regimes emerge, with the modification that in the case of (ii) since we are dealing with a fixed design (equispaced time points), the limit process becomes a two-sided {\em generalized} random walk. In addition, there exists a third one (Regime (I)), where the asymptotic distribution of the change point becomes degenerate at the true value.
\newline
\indent Next, we introduce  assumptions needed to establish these results. As we deal with dependence across both time and panels, we have the following modified regimes --- (a) $\gamma_p^{-2}||\mu_1-\mu_2||_2^2 \to \infty$, (b) $\gamma_p^{-2}||\mu_1-\mu_2||_2^2 \to 0$ and (c) $\gamma_p^{-2}||\mu_1-\mu_2||_2^2 \to c>0$. If $\gamma_p = O(1)$, then (a)-(c) coincide with (I)-(III) above.  In Regime (a),  the asymptotic distribution of the change point can be derived under the same assumptions as in Theorem \ref{thm: conrate}. On the other hand, in the second and third regimes,  a non-degenerate limit distribution can be obtained under the following additional  assumptions.  Detailed comments on these assumptions will be provided after stating the results. 
\vskip 5pt
\noindent \textbf{Regime (b): $\gamma_p^{-2}||\mu_1 - \mu_2||_2 \to 0$, assumptions}.
For all $h_1,h_2 \in \mathbb{R}$,  define
\begin{eqnarray}
\Lambda_{h_1,h_2} &=& \sum_{\stackrel{t_i = \bigg[\frac{\gamma_p^2 h_i}{||\mu_1-\mu_2||_2^2}\bigg]\wedge 0 +1}{i=1,2}}^{\bigg[\frac{\gamma_p^2 h_i}{||\mu_1-\mu_2||_2^2}\bigg]\vee 0} \Gamma_{t_2-t_1},\ \ \  
\sigma_{h_1,h_2} = \lim_p \gamma_p^{-4} (\mu_1-\mu_2)^\prime \Lambda_{h_1,h_2} (\mu_1-\mu_2).\label{eqn: sigma}
\end{eqnarray}
We  require, \vskip 2pt
\noindent \textbf{(A1)} $\sigma_{h_1,h_2} $ exists  for all  $h_1,h_2 \in \mathbb{R}$. 
\vskip 5pt
\noindent \textbf{(A2)} $((\sigma_{h_i,h_j}  ))_{1 \leq i,j \leq r}$ is positive definite for all $h_1,h_2,\ldots,h_r \in \mathbb{R}$ and $r \geq 1$.
\vskip 5pt
\noindent \textbf{(A3)} $\frac{\sup_{k} |\mu_{1k}-\mu_{2k}|(1-I(A_j=0\ \forall j \geq 1)) }{||\mu_1-\mu_2||_2} = o(1)$. 
\vskip 2pt
\noindent 
If we have independence across time and panels -i.e. $A_j =0\ \forall j \geq 1$- condition (A3) is satisfied automatically.  
If we have independence across both time and panels, then further simplification of $\sigma_{h_1,h_2}$ is possible. For more details, see Sections \ref{cor: panelind} and \ref{cor: paneltimeind}. 
\vskip 5pt
\noindent \textbf{Regime (c): $\gamma_p^{-2}||\mu_1-\mu_2||_2 \to c>0$, assumptions:}.
\noindent 
Consider the following disjoint and exhaustive subsets of $\{1,2,3,\ldots, p(n)\}$:
\begin{align} \label{eqn: partition}
\mathcal{K}_0 &= \{k:\ 1 \leq k \leq p(n),\ \lim (\mu_{1k} - \mu_{2k}) \neq 0\} \ \ \text{and} \\
\mathcal{K}_n &= \mathcal{K}_0^c = \{k:\ 1 \leq k \leq p(n),\ \lim (\mu_{1k} - \mu_{2k}) = 0\},\ \forall n \geq 1. \nonumber
\end{align}
We make the following assumptions. 
\vskip 3pt
\noindent \textbf{(A4)}  $\mathcal{K}_0$ does not vary with $n$.  
\vskip 3pt
\noindent \textbf{(A5)} For some $\tau^{*} \in (c^{*},1-c^{*})$, $\tau_n \to \tau^{*}$. 

\vskip 3pt
\noindent Define, 
\begin{eqnarray}
c_1 &=& \lim \gamma_p^{-2} \sum_{k \in \mathcal{K}_n}  (\mu_{1k} - \mu_{2k})^2, \label{eqn: tildesigma} \\
\tilde{\sigma}_{(0,0),(t_1,t_2)} &=& \lim \gamma_p^{-4} \sum_{k_1,k_2 \in \mathcal{K}_n} (\mu_{1k_1} - \mu_{2k_1}) (\mu_{1k_2}-\mu_{2k_2}) \Gamma_{t_2-t_1}(k_1,k_2), \nonumber \\
\tilde{\sigma}_{(k,0),(t_1,t_2)} &=&  \lim \gamma_p^{-3} \sum_{k_1 \in \mathcal{K}_n} (\mu_{1k_1} - \mu_{2k_1}) \Gamma_{t_2-t_1}(k,k_1), \nonumber \\
\tilde{\sigma}_{(k_1,k_2),(t_1,t_2)} &=&  \lim \gamma_p^{-2}\Gamma_{t_2-t_1}(k_1,k_2). \nonumber
\end{eqnarray}
\noindent \textbf{(A6)} $c_1$, $\tilde{\sigma}_{(k_1,k_2),(t_1,t_2)}$ exists for all $t_1, t_2 \in \mathbb{Z}$ and $k_1,k_2 \in \mathcal{K}_0\cup\{0\}$. 
\vskip 5pt
\noindent \textbf{(A7)}  $((\tilde{\sigma}_{(k_i,k_j),(t_i,t_j)}))_{1\leq i,j \leq r}$ is positive definite  for all $k_1,k_2,\ldots,k_r \in \mathcal{K}_0\cup\{0\}$, $t_1,t_2,\ldots,t_r \in \mathbb{Z}$ and $r \geq 1$. 
\vskip 5pt
\noindent \textbf{(A8)} $\displaystyle \frac{\sup_{k \in \mathcal{K}_0} \bigg[\sup_{i} \sum_{j=0}^{\infty}|A_{j}(k,i)|^3 \bigg]^{1/3}}{\inf_{k \in \mathcal{K}_0} \bigg[ \inf_{i} \sum_{j=0}^{\infty}|A_{j}(k,i)|^2 \bigg]^{1/2}} = o(p^{1/6})$ and $\gamma_p^{-1}\mu_{ik} \to \mu_{ik}^{*}$ $\forall k \in \mathcal{K}_0$,  $i=1,2$.
\vskip 5pt
\noindent \textbf{(A9)}  
$\sup_{k \in \mathcal{K}_n} \gamma_p^{-1}|\mu_{1k}-\mu_{2k}| \to 0$.
\vskip 2pt
\vskip 2pt


\noindent Given the previously posited assumptions, we next state the following theorem that describes the limiting distribution of $\hat{\tau}_{n}$, whose proof is given in Section \ref{subsec: asympdist}. 

\begin{theorem} \label{thm: asympdist}
Suppose (SNR) holds. 
\vskip 2pt
\noindent (a)  If $\gamma_p^{-2}||\mu_1 -\mu_2||_2^2 \to \infty$, then $\stackrel{\lim}{n\to \infty} P(\hat{\tau}_n = \tau_n)=1$. 
\vskip 5pt
\noindent (b) If $\gamma_p^{-2}||\mu_1 -\mu_2||_2^2 \to 0$ and (A1), (A2) and (A3) hold, then
for all $h_1,h_2,\ldots,h_r \in \mathbb{R}$ and $r \geq 1$,
\begin{eqnarray}
n \gamma_p^{-2}||\mu_1-\mu_2||_2^2 (\hat{\tau}_n - \tau_n) \stackrel{\mathcal{D}}{\to} \arg \max_{h \in \mathbb{R}} (-0.5|h| + B_h^*)\ \ \text{where} \nonumber \\
(B_{h_1}^*,B_{h_2}^*,\ldots,B_{h_r}^*) \sim \mathcal{N}_{r}(0,\Sigma),\ \ \Sigma = ((\sigma_{h_ih_j}))_{1 \leq i,j \leq r} \,,\nonumber
 \end{eqnarray}
 { where $B_h^{\star}$ is a tight Gaussian process on $\mathbb{R}$ with continuous sample paths.} 
 \newline
\noindent (c) If $\gamma_p^{-2} ||\mu_1-\mu_2||_2^2 \to c>0$ and (A4)-(A9) hold, then
\begin{eqnarray}
n(\hat{\tau}_n-\tau_n) \stackrel{\mathcal{D}}{\to} \arg\max_{h \in \mathbb{Z}} (-0.5c_1 |h| +  \sum_{t=0\wedge h}^{0 \vee h} (W_t^* + A_t^*)) \nonumber
\end{eqnarray}
where for each $t_1,t_2, \ldots,t_r \in \mathbb{Z}$, $k_1,k_2,\ldots,k_r \in \mathcal{K}_0$  and $r \geq 1$
\begin{eqnarray}
&&(W_{t_1}^{*},W_{t_2}^{*},\ldots, W_{t_r}^{*}) \sim \mathcal{N}_r (0, ((\tilde{\sigma}_{(0,0),(t_i,t_j)}))_{1 \leq i,j \leq r} ),\ \ \
\nonumber \\
&& A_t^* = \frac{1}{2}\sum_{k \in \mathcal{K}_0} \bigg[ (X_{kt}^{*} +b_{kt}-\mu_{2k}^*)^2-(X_{kt}^* +b_{kt}-\mu_{1k}^*)^2 \bigg],\ \  \text{Cov}(X_{k_1 t_1}^*, W_{t_2}^*) =\tilde{\sigma}_{(k_1,0),(t_1,t_2)},  \nonumber \\
&& (X_{k_1t_1}^{*},X_{k_2t_2}^{*},\ldots, X_{k_r t_r}^{*})  \sim \mathcal{N}_r ((b_{k_1t_1},b_{k_2t_2},\ldots, b_{k_r t_r}),((\tilde{\sigma}_{(k_i,k_j),(t_i,t_j)}))_{1 \leq i,j \leq r} ),   \nonumber \\
&& \hspace{0 cm} b_{k_i,t_i} = \begin{cases} \mu_{1k_i}^{*}\ \ \text{if $t_i \leq n\tau^{*}$} \nonumber \\
\mu_{2k_i}^{*}\ \ \ \text{if $t_i > n\tau^{*}$}
\end{cases} \forall 1 \leq i \leq r.
\end{eqnarray}
\end{theorem}


\noindent \textbf{Discussion of Theorem \ref{thm: asympdist}}.  We now provide detailed comments on the assumptions and how they relate to the three regimes established in Theorem \ref{thm: asympdist}.
In the first regime, the signal for $\hat{\tau}_{n}$ is high and therefore the difference
$(\hat{\tau}_{n} - \tau_n)$ becomes a point
mass at $0$. On the other hand, in the second and third regimes,  the total signal is weak and moderate, respectively, and a non-degenerate limit distribution can be obtained under additional assumptions.
\newline
\indent
Under the last two regimes, the results are based on the weak convergence of the process 
\begin{eqnarray} 
M_{n} (h) := n(L_{}(\tau_n + n^{-1}\gamma_p^{-2}||\mu_1-\mu_2||_2^{-2}h)-L_{}(\tau_n)),  \nonumber \\
h \in ||\mu_1-\mu_2||_2^{2}\{-(n-1), -(n-2),\ldots, -1,0,1,\ldots, (n-2), (n-1)\}. \hspace{1 cm} \label{eqn: normalizationlse}
\end{eqnarray}
Under  appropriate conditions (as mentioned in Theorem \ref{thm: asympdist}),  $\arg \max_{h} M_{n}(h) = n \gamma_p^{-2}||\mu_1 - \mu_2||_2^2 (\hat{\tau}_{n}-\tau_n)$  converges weakly to the unique maximizer of the limiting process. For more details see Lemma \ref{lem: wvandis1}.  

\noindent \textbf{Regime (b): $\gamma_p^{-2}||\mu_1-\mu_2||_2 \to 0$}. In the second regime, the asymptotic covariance of $(M_{n} (h_1),M_{n} (h_2),\ldots,M_{n} (h_r))$ is proportional to $(( \sigma_{h_i h_j}))_{1 \leq i,j \leq r}$ for all $h_1,h_2,\ldots,h_r \in \mathbb{R}$, $r \geq 1$ and hence the need for assumption (A1). Discussion of  conditions under which (A1) is satisfied are given in Proposition \ref{prop: 1}.  
Assumption (A2) is required for establishing the non-degeneracy of the asymptotic distribution. Finally, the asymptotic normality of $M_n(h)$ requires (A3). If $A_j = 0\ \forall j \geq 1$ -i.e. if we have independence across time- then (A3) is satisfied automatically. 

\noindent \textbf{Regime (c): $\gamma_p^{-2}||\mu_1-\mu_2||_2 \to c>0$}. Under the third regime, the limiting process has two components based on the partition of  $\{1,2,\ldots, p(n)\}$ into sets $\mathcal{K}_0$ and $\mathcal{K}_n = \mathcal{K}_0^c$, defined in (\ref{eqn: partition}).   
Observe that $\mathcal{K}_n$ is the collection of all such indices whose corresponding variables eventually have the same mean before and after the change point. In the second regime, $\mathcal{K}_0$ is the empty set. On the other hand, under the third regime,  $\mathcal{K}_0$ may not be empty,  but can be at most a finite set. Hence, $\mathcal{K}_n$ is necessarily an infinite set. 

These two sets in the partition contribute differently to the limit. Let $L(b) = \frac{1}{n}\sum_{k=1}^{p} L_{k}(b)$,
\begin{eqnarray}
M_{k,n}^{*}(h) = (L_{k}(\tau_n + n^{-1}h)-L_{k}(\tau_n)),\ \ \   
h \in \{-(n-1), \ldots, -1,0,1,\ldots, (n-2), (n-1)\}.  \nonumber
\end{eqnarray}
  Note that
\begin{eqnarray}
M_{n}^{*}(h) = \sum_{k \in \mathcal{K}_n} M_{k,n}^{*}(h) + \sum_{k \in \mathcal{K}_0} M_{k,n}^{*}(h) =: M_n^{I}(h) + M_{n}^{II}(h),\ \text{say}.
\end{eqnarray}

\noindent \textbf{Limit of $\mathcal{K}_0$}: Conditions (A4) and (A5)  are required to establish the limit of the random part $M_{n}^{II}(h)$ involving  $\{X_{kt}: k \in \mathcal{K}_0, 1 \leq t \leq n\}$.   The asymptotic covariance between $X_{k_1t_1}$ and $X_{k_2t_2}$ is given by $\tilde{\sigma}_{(k_1,k_2),(t_1,t_2)}$ for all $k_1,k_2 \in \mathcal{K}_0$ and $t_1,t_2 \in \mathbb{Z}$. Existence of these covariances and non-degeneracy of the corresponding asymptotic distributions are guaranteed by (A6) and (A7) for $k_1,k_2 \in \mathcal{K}_0$ and $t_1,t_2 \in \mathbb{Z}$. Condition (A8)  is a technical condition and is required for establishing the asymptotic normality of $X_{kt}$ for all $t \in \mathbb{Z}$ and $k \in \mathcal{K}_0$. As $\mathcal{K}_0$ is a finite set, by (A4)-(A8),  $M_{n}^{II}(h)$ converges weakly to the process $\sum_{t=h\wedge 0}^{h \vee 0} A_t^*$, described in Theorem \ref{thm: asympdist}(c). 
\vskip 2pt
\noindent \textbf{Limit of $\mathcal{K}_n$}: The limit of the random part $M_{n}^{I}(h)$ involving  $\{X_{kt}: k \in \mathcal{K}_n, 1 \leq t \leq n\}$ is an appropriately scaled  Gaussian process  on $\mathbb{Z}$ with a triangular drift, as given by $-0.5c_1 |h| +  \sum_{t=0\wedge h}^{0 \vee h} W_t^*$ in Theorem \ref{thm: asympdist}(c)  and can be established using similar arguments given in Regime (b).   Analogous to Assumptions (A1), (A2) and (A3) in Regime (b) are conditions (A6),  (A7) for $k_1,k_2 = 0$ and $t_1,t_2 \in \mathbb{Z}$ and (A9). Discussion of  conditions under which Assumption (A6) is satisfied are given in Proposition \ref{prop: 2}.
\vskip 2pt
\noindent \textbf{Dependence between  $\mathcal{K}_0$ and $\mathcal{K}_n$}: Moreover, the limits coming from $\{X_{kt}: k \in \mathcal{K}_0, 1 \leq t \leq n\}$ and $\{X_{kt}: k \in \mathcal{K}_n, 1 \leq t \leq n\}$ are correlated and their covariances are given by $\tilde{\sigma}_{(k,0),(t_1,t_2)}$, where $k \in \mathcal{K}_0$ and $t_1,t_2\in \mathbb{Z}$. Thus, we require Assumptions (A6) and (A7) for all  $t_1, t_2 \in \mathbb{Z}$ and $k_1,k_2 \in \mathcal{K}_0\cup\{0\}$. 
\vskip 2pt
\indent (A9) is a technical assumption. Following the proof of Theorem \ref{thm: asympdist}(c),  at some point we need to establish  the  asymptotic normality of 
\begin{align}
\sum_{k \in \mathcal{K}_n} (\mu_{1k}-\mu_{2k})(X_{kt} - E(X_{kt}))\ \forall t \geq 1. \label{eqn: asymnorlse}
\end{align}
Note that for $t \geq 1$, $\{(\mu_{1k}-\mu_{2k})(X_{kt} - E(X_{kt})): k \in \mathcal{K}_n\}$ is a collection of infinitely many  centered random variables. To apply Lyapunov's central limit theorem to (\ref{eqn: asymnorlse}),  we require
\begin{align}
\frac{\gamma_p^{-6}\sum_{k_1,k_2,k_3 \in \mathcal{K}_n} \left(\prod_{i=1,2,3} |\mu_{1k_i}-\mu_{2k_i}|\right) E\left(\prod_{i=1,2,3} |X_{k_i t} - E(X_{k_i t})| \right)}{\gamma_p^{-4}\sum_{k_1,k_2 \in \mathcal{K}_n}(\mu_{1k_1}-\mu_{2k_1}) (\mu_{1k_2}-\mu_{2k_2}) \Gamma_{0}(k_1,k_2) }  \to 0.  \label{eqn: lyaplse}
\end{align}
By (A6) and (A7), the left side of (\ref{eqn: lyaplse}) is dominated by
\begin{eqnarray}
&& C\gamma_p^{-3} (\sup_{k \in \mathcal{K}_n} |\mu_{1k}-\mu_{2k}| )^3 \gamma_p^{-3} \sum_{k_1,k_2,k_3 \in \mathcal{K}_n}  E\left(\prod_{i=1,2,3} |X_{k_i t} - E(X_{k_i t})| \right) \nonumber \\
& \leq & C \gamma_p^{-3} (\sup_{k \in \mathcal{K}_n} |\mu_{1k}-\mu_{2k}| )^3 \gamma_p^{-3} \sum_{k_1,k_2,k_3 \in \mathcal{K}_n} \sum_{j_1,j_2,j_3=0}^{\infty} \sum_{i_1,i_2,i_3=1}^{p} |A_j(k_1,i_1)| \nonumber \\
& \leq & C \gamma_p^{-3} (\sup_{k \in \mathcal{K}_n} |\mu_{1k}-\mu_{2k}| )^3  \label{eqn: lyaplsedom}
\end{eqnarray}
for some $C>0$.  (A9) is a natural sufficient condition for (\ref{eqn: lyaplsedom}) to converge to $0$.  Similarly we need (A8) for the asymptotic normality of $X_{kt}-EX_{kt}$ for all $k \in \mathcal{K}_0$ and $t \in \mathbb{Z}$.

\begin{remark} \label{rem: baigammadist}
Similarly to Remark \ref{rem: baigammamain}, if we have cross-sectional independence, i.e. $A_{j,p}= \text{Daig}\{a_{j,k}:\ 1 \leq k \leq p\}$, then $\gamma_p$ can be replaced by $\tilde{\gamma}_p := \sup_{1\leq k \leq p} \sum_{j=0}^{\infty}|a_{j,k}|$ throughout Theorem \ref{thm: asympdist} and we require the weaker (SNR$^\prime$) $\frac{n\tilde{\gamma}_p^{-2}}{p} ||\mu_1-\mu_2||_2^2 \to \infty$ instead of (SNR). 
\end{remark}

\noindent The asymptotic distributions obtained in Theorem \ref{thm: asympdist} can be simplified further for some special cases such as independence across panels and independence across both panels and time. 

\vskip 2pt
\noindent \textbf{Independence across panels}: Then, for all $t_1,t_2 \in \mathbb{Z}$, $\tilde{\sigma}_{(k,0),(t_1,t_2)} =0\ \forall k \in \mathcal{K}_0$ and $\tilde{\sigma}_{(k_1,k_2),(t_1,t_2)} = 0\ \forall k_1\neq k_2,\  k_1,k_2 \in \mathcal{K}_0$. This implies that the limits of $\{X_{kt}:  1 \leq t \leq n\}$ are independent across $k \in \mathcal{K}_0$. Moreover, the limits coming from $\{X_{kt}: k \in \mathcal{K}_0, 1 \leq t \leq n\}$ and $\{X_{kt}: k \in \mathcal{K}_n, 1 \leq t \leq n\}$ are also independent. For more details, see Section \ref{cor: panelind}. 
\vskip 2pt
\noindent \textbf{Independence across time and panels}: Then, the population autocovariance matrix $\Gamma_u = 0\ \forall u \neq 0$ and consequently $\tilde{\sigma}_{(k_1,k_2),(t_1,t_2)} \neq 0$, only when $t_1=t_2 \in \mathbb{Z}$ and $k_1 = k_2 \in \mathcal{K}_0 \cup \{0\}$. More details are given in Section \ref{cor: paneltimeind}.

\vskip 10pt
\noindent \textbf{Sufficient conditions for (A1) and (A6)}. Next, we further elaborate on Assumptions (A1) and (A6) that ensure the existence of certain limits,  which are not trivially satisfied.  We need some restrictions on the means $\mu_1$ and $\mu_2$ and on the coefficient matrices $\{A_{j,p}:\ j \geq 0\}$ in the panel data model (\ref{eqn:model}) for satisfying (A1) and (A6). Propositions \ref{prop: 1} and \ref{prop: 2}, given below, provide sufficient conditions for (A1) and (A6) to hold.  
Their proofs are  given in Section \ref{subsec: prop}. 

Define for any  sequence of matrices $\{M_p = ((m_{ij}))_{p \times p}\}$,
\begin{eqnarray*}
||M||_{(1,1)}& =& \max_{1 \leq j \leq p} \sum_{i= 1}^{p} |m_{ij}|\ \ \text{and}\ \ \beta_p = \sum_{j=0}^{\infty} ||A_j||_{(1,1)}.
\end{eqnarray*}

\begin{prop} \label{prop: 1}
Suppose the following conditions hold.
\vskip 2pt
\noindent (a) $\frac{\gamma_{p(n+1)}}{\gamma_{p(n)}} \to 1$, (b) $\frac{\beta_{p(n+1)}}{\beta_{p(n)}} \to 1$,  (c) $\frac{p(n+1)}{p(n)} \to 1$,
\vskip 2pt
\noindent (d)  $\frac{||\mu_{1,p(n)}-\mu_{2,p(n)} ||_2^2}{||\mu_{1,p(n+1)}-\mu_{2,p(n+1)} ||_2^2 } \to 1$,
\vskip 2pt
\noindent (e) $\gamma_{p(n)}^{-2}\beta_{p(n)}^2 p(n) \frac{\displaystyle{\sup_{i=1,2}\sup_{1 \leq k \leq p(n)} |\mu_{ik,p(n+1)}-\mu_{ik,p(n)}|^2}}{||\mu_{1,p(n)}-\mu_{2,p(n)} ||_2^2} \to 0$,
\vskip 2pt
\noindent (f) $\gamma_{p(n)}^{-2}\beta_{p(n)}^2 p(n) \frac{\displaystyle{\sup_{1 \leq k \leq p(n)} |\mu_{1k,p(n)}-\mu_{2k,p(n)}|^2}}{||\mu_{1,p(n)}-\mu_{2,p(n)} ||_2^2} \to 0$.
\vskip 2pt
\noindent Then, under $\gamma_{p(n)}^{-2} ||\mu_{1,p(n)}-\mu_{2,p(n)} ||_2^2 \to 0$,  $\sigma_{h_1,h_2}$ exists for all $h_1,h_2 \in \mathbb{R}$. 
\end{prop}

\begin{prop} \label{prop: 2}
(I) Suppose (a), (b), (c) in Proposition \ref{prop: 1} hold and  
\vskip 2pt
\noindent (g) $\gamma_{p(n)}^{-4}\beta_{p(n)}^2 p(n) \displaystyle{\sup_{i=1,2}\sup_{k \in \mathcal{K}_n} |\mu_{ik,p(n+1)}-\mu_{ik,p(n)}|^2} \to 0$,
\vskip 2pt
\noindent (h) $\gamma_{p(n)}^{-4}\beta_{p(n)}^2 p(n) \displaystyle{\sup_{k \in \mathcal{K}_n} |\mu_{1k,p(n)}-\mu_{2k,p(n)}|^2} \to 0$.
\vskip 2pt
\noindent Then, under $\gamma_{p(n)}^{-2} ||\mu_{1,p(n)}-\mu_{2,p(n)} ||_2^2 \to c> 0$,  $\tilde{\sigma}_{(0,0),(t_1,t_2)}$ exists for all $t_1,t_2 \in \mathbb{Z}$. 
\vskip 2pt
\noindent (II) Suppose (a), (b), (c) in Proposition \ref{prop: 1} hold and  
\vskip 2pt
\noindent (i) $\gamma_{p(n)}^{-4}\beta_{p(n)}^2 p(n) \displaystyle{\sup_{i=1,2}\sup_{ k \in \mathcal{K}_n} |\mu_{ik,p(n+1)}-\mu_{ik,p(n)}|} \to 0$,
\vskip 2pt
\noindent (j) $\gamma_{p(n)}^{-4}\beta_{p(n)}^2 p(n) \displaystyle{\sup_{k \in \mathcal{K}_n} |\mu_{1k,p(n)}-\mu_{2k,p(n)}|} \to 0$.
\vskip 2pt
\noindent Then, under $\gamma_{p(n)}^{-2} ||\mu_{1,p(n)}-\mu_{2,p(n)} ||_2^2 \to c> 0$,  $\tilde{\sigma}_{(k,0),(t_1,t_2)}$ exists for all $t_1,t_2 \in \mathbb{Z}$  and  $k \in \mathcal{K}_0$. 
\vskip 2pt
\noindent (III) Suppose (a),  (c) in Proposition \ref{prop: 1} hold. Then, under $\gamma_{p(n)}^{-2} ||\mu_{1,p(n)}-\mu_{2,p(n)} ||_2^2 \to c> 0$,  $\tilde{\sigma}_{(k_1,k_2),(t_1,t_2)}$ exists for all $t_1,t_2 \in \mathbb{Z}$  and  $k_1,k_2 \in \mathcal{K}_0$.
\end{prop}

 
Proposition \ref{prop: 1} needs Conditions (a)-(f) for satisfying (A1).  Next we explain these conditions.  For a sequence $\{a_n\}$, the first order differences of $\{a_n\}$ grow slower than $\{a_n\}$ if $\frac{a_{n+1}}{a_n} \to 1$ holds. Conditions (a)-(d) in Proposition \ref{prop: 1} assume that the first order differences of sequences $\{\gamma_{p(n)}\}$, $\{\beta_{p(n)}\}$, $\{p(n)\}$ and $||\mu_{1,p(n)}-\mu_{2,p(n)}||_2$ grow  at a slower rate than these sequences respectively.  Conditions (e) and (f)  in Proposition \ref{prop: 1} respectively ensure that, uniformly over $i=1,2$ and $1 \leq k \leq p(n)$,  the first order differences of $\{\mu_{ik,p(n)}:\ n \geq 1\}$ and the sequence $\{\gamma_p^{-1}|\mu_{1k,p(n)}-\mu_{2k,p(n)}|:\ n \geq 1\}$ decay faster than $||\mu_1-\mu_2||_2 \gamma_p  (\beta_p\sqrt{p(n)})^{-1}$ and $||\mu_1-\mu_2||_2  (\beta_p\sqrt{p(n)})^{-1}$ respectively. 

In Proposition \ref{prop: 2}, we divide Assumption (A6) into three parts --- (I) $\tilde{\sigma}_{(0,0),(t_1,t_2)}$ exists for all $t_1,t_2 \in \mathbb{Z}$, (II) $\tilde{\sigma}_{(k,0),(t_1,t_2)}$ exists for all $t_1,t_2 \in \mathbb{Z}$  and  $k \in \mathcal{K}_0$ and (III) $\tilde{\sigma}_{(k_1,k_2),(t_1,t_2)}$ exists for all $t_1,t_2 \in \mathbb{Z}$  and  $k_1,k_2 \in \mathcal{K}_0$. For (I), Proposition \ref{prop: 2} needs  (a)-(c) given in Proposition \ref{prop: 1} and two additional Conditions (g) and (h) which are same as (e) and (f) respectively  except $1 \leq k \leq p(n)$ is replaced by $k \in \mathcal{K}_n$ and the rate of decay $||\mu_1-\mu_2||_2 \gamma_p (\beta_p\sqrt{p(n)})^{-1} \sim \gamma_p^2 (\beta_p\sqrt{p(n)})^{-1}$ for (g) and $||\mu_1-\mu_2||_2  (\beta_p\sqrt{p(n)})^{-1} \sim \gamma_p (\beta_p\sqrt{p(n)})^{-1}$ for (h) (as $\gamma_p^{-2}||\mu_1-\mu_2||_2^2 \to c>0$). Proposition \ref{prop: 2} also states that (II) holds if (a)-(c) in Proposition \ref{prop: 1} are satisfied along with Conditions (i) and (j) which are similar to (g) and (h) except  here the rates of decay are $\gamma_p^{4} (\beta_p^2 {p(n)})^{-1}$ and  $\gamma_p^{3} (\beta_p^2 {p(n)})^{-1}$  respectively. Moreover, (III) is satisfied if (a) and (c) in Proposition \ref{prop: 1} hold. 

Consider the following simplified conditions.
\vskip5pt
\noindent (k) $p(n) \frac{\displaystyle{\sup_{i=1,2}\sup_{1 \leq k \leq p(n)} |\mu_{ik,p(n+1)}-\mu_{ik,p(n)}|^2}}{||\mu_{1,p(n)}-\mu_{2,p(n)} ||_2^2} \to 0$ and  $p(n) \frac{\displaystyle{\sup_{1 \leq k \leq p(n)} |\mu_{1k,p(n)}-\mu_{2k,p(n)}|^2}}{||\mu_{1,p(n)}-\mu_{2,p(n)} ||_2^2} \to 0$.  That is uniformly over $i=1,2$ and $1 \leq k \leq p$,  the first order differences of $\{\mu_{ik,p}:\ n \geq 1\}$ and the sequence $\{|\mu_{1k,p}-\mu_{2k,p}|:\ n \geq 1\}$ decay faster than $||\mu_1-\mu_2||_2 p^{-1/2}$. 
\vskip 5pt
\noindent (l) $p(n) \frac{\displaystyle{\sup_{i=1,2}\sup_{k \in \mathcal{K}_n} |\mu_{ik,p(n+1)}-\mu_{ik,p(n)}|}}{||\mu_{1,p(n)}-\mu_{2,p(n)} ||_2^2} \to 0$ and  $p(n) \frac{\displaystyle{\sup_{k \in \mathcal{K}_n} |\mu_{1k,p(n)}-\mu_{2k,p(n)}|}}{||\mu_{1,p(n)}-\mu_{2,p(n)} ||_2^2} \to 0$.  That is uniformly over $i=1,2$ and $k \in \mathcal{K}_n$,  the first order differences of $\{\mu_{ik,p}:\ n \geq 1\}$ and the sequence $\{|\mu_{1k,p}-\mu_{2k,p}|:\ n \geq 1\}$ decay faster than $||\mu_1-\mu_2||_2^2 p^{-1}$.
\vskip 5pt
\noindent Suppose (a)-(d) in Proposition \ref{prop: 1} hold and $\gamma_p^{-1}\beta_p = O(1)$ (e.g. $\beta_p = O(1)$ and $\inf_{p} \gamma_p >0$).  Then it is easy to see that (A1) and (A6)  are satisfied under (k) and (l) respectively. 

\noindent If $\gamma_p^{-1}\beta_p = O(1)$ and $\gamma_p = O(1)$, then (g)-(j) in Proposition \ref{prop: 2} reduce to  (l1), given below.
\vskip5pt
\noindent (l1) $p(n) \displaystyle{\sup_{i=1,2}\sup_{k \in \mathcal{K}_n} |\mu_{ik,p(n+1)}-\mu_{ik,p(n)}|} \to 0$ and  $p(n) \displaystyle{\sup_{k \in \mathcal{K}_n} |\mu_{1k,p(n)}-\mu_{2k,p(n)}|} \to 0$.  That is uniformly over $i=1,2$ and $k \in \mathcal{K}_n$,  the first order differences of $\{\mu_{ik,p}:\ n \geq 1\}$ and the sequence $\{|\mu_{1k,p}-\mu_{2k,p}|:\ n \geq 1\}$ decay faster than $p^{-1}$.
\vskip 5pt

Following the proof of Propositions \ref{prop: 1} and \ref{prop: 2}, it is easy to see that, under cross-sectional independence i.e. when $A_{j,p} = \text{Diag}\{a_{j,k}:\ 1\leq k \leq p\}\ \forall j \geq 0$, $\gamma_p$ and $\beta_p$ both can be replaced by smaller quantity $\tilde{\gamma}_p := \sup_{1 \leq k \leq p} \sum_{j=0}^{\infty} |a_{j,k}|$ throughout the propositions.  In this case also, (e), (f) in Proposition \ref{prop: 1} are implied by (k) and (g)-(j) in Proposition \ref{prop: 2} are implied by (l).  Moreover  if $\tilde{\gamma}_p = O(1)$, then it can be replaced by $1$ and  (l) reduces to (l1).  

\noindent \textbf{Sufficient conditions for (A2) and (A7)}. It is easy to see that $((\sigma_{h_i,h_j}))_{1 \leq i,j \leq r}$ and  $((\tilde{\sigma}_{(k_i,k_j),(t_i,t_j)}))_{1 \leq i,j \leq r}$ are positive definite, 
if $((\Gamma_{i-j}))_{1 \leq i,j \leq r}$ is positive definite for all $r \geq 1$. 
Using similar arguments as in the univariate case in \citet{BD2009}, if $\Gamma_0$ is positive definite then so is  
$((\Gamma_{i-j}))_{1 \leq i,j \leq r}$ for all $r \geq 1$. Therefore, a sufficient condition for (A2) and (A7) is positive definiteness of $\Gamma_0$ uniformly over $p$ i.e., $\inf_{p} \lambda_{\min}(\Gamma_0) >0$ where $\lambda_{\min}$ denotes its smallest eigenvalue.

\subsection{Illustrative Examples} \label{subsec: example}
Next, we provide some examples of coefficient matrices which satisfy the assumptions required for Theorems \ref{thm: conrate} and \ref{thm: asympdist}. Note that Assumptions (A1), (A2), (A6), (A7) and the first condition of (A8) relate to the coefficient matrices and, as previously discussed, are rather technical in nature. The following examples and ensuing discussion provide additional insight into the nature of these assumptions.  

\noindent For any $p\times p$ matrix $M$, define $||M||_{\infty} = \max_{1 \leq i,j \leq p} |M(i,j)|$. 

\noindent For all examples given below,  suppose the first order differences of sequences  $\{p\}$ and $||\mu_{1,p}-\mu_{2,p}||_2$ grow  at a slower rate than these sequences respectively, i.e.,  suppose (c) and (d) in Proposition \ref{prop: 1} are satisfied.

\begin{Example} \label{example: band}
\noindent \textbf{Banded coefficient matrices}: A matrix $M = ((m_{ij}))$ is said to be $K$-banded, if $m_{ij} = 0$ for $|i-j|>K$. It is known that the $u$-th order population auocovariance matrix $\Gamma_u$ can be consistently estimated, if $\{A_j\}$ are $K$-banded (see e.g. \cite{BB2013a}).  The $K$-banded structure of $A_j$ implies that $\varepsilon_{kt}$ and $\eta_{k^\prime t-j}$ are uncorrelated for $|k -k^\prime|>K$. 
For all $j \geq 0$, suppose $A_{j,p}$ is symmetric  and $A_{j,p}(i,l) = 0$ if $|i-l|> K_{j,p}$.  Also assume $0<\sup_{p}\sum_{j=0}^{\infty} ||A_{j,p}||_{\infty} < C$,  $\sup_{j} K_{j,p} < K_p$ and the smallest eigenvalue of $A_{j,p}$ is bounded away from $0$ for at least one $j \geq 0$. Then, $\gamma_p \leq \beta_p = \sum_{j=0}^{\infty} ||A_{j,p}||_{(1,1)} \leq \sum_{j=0}^{\infty} K_{j,p}||A_{j,p}||_{\infty} \leq K_p \sum_{j=0}^{\infty} ||A_j||_{\infty} = O(K_p)$ and they are bounded away from $0$.  Hence, (A2), (A7) and the first condition of (A8) hold. The SNR condition becomes $\frac{n}{pK_p^2} ||\mu_1-\mu_2||_2^2 \to \infty$.  Finally, (A1) and (A6) are satisfied if (a)-(j) in Propositions \ref{prop: 1} and \ref{prop: 2} hold with  $\beta_p$ replaced by $K_p$.  
For more explanation, suppose the first order differences of sequences $\{\gamma_p\}$ and $\{K_p\}$  grow  at a slower rate than these sequences respectively.
Then (A1) holds  if uniformly over $i=1,2$ and $1 \leq k \leq p$,  the first order differences of $\{\mu_{ik,p}:\ n \geq 1\}$ and the sequence $\{\gamma_p^{-1}|\mu_{1k,p}-\mu_{2k,p}|:\ n \geq 1\}$ decay faster than $||\mu_1-\mu_2||_2 \gamma_p  (K_p\sqrt{p})^{-1}$ and $||\mu_1-\mu_2||_2  (K_p\sqrt{p})^{-1}$ respectively.  Moreover (A6) is satisfied if  uniformly over $i=1,2$ and $k \in \mathcal{K}_n$,  the first order differences of $\{\mu_{ik,p}:\ n \geq 1\}$ and the sequence $\{\gamma_p^{-1}|\mu_{1k,p}-\mu_{2k,p}|:\ n \geq 1\}$ decay faster than $||\mu_1-\mu_2||_2^2 \gamma_p^2  (K_p^2 {p})^{-1}$ and $||\mu_1-\mu_2||_2^2 \gamma_p  (K_p^2 {p})^{-1}$ respectively.
\end{Example}

\begin{Example} \label{example: VARMA}
\textbf{VARMA  process}: Suppose $\{\varepsilon_{t,p(n)}^{(n)}\}$ follows a VARMA$(p,q)$ process given by
\begin{eqnarray} \label{eqn: ARMA}
\sum_{j=0}^{q} \Psi_{j,p(n)} \varepsilon_{t-j,p(n)}^{(n)}  = \sum_{j=0}^{p} \Phi{j,p(n)} \eta_{t-j,p(n)}
\end{eqnarray}
where $\sup_{p} ||\Psi_{j,p}||_{(1,1)},  \sup_{p} ||\Phi_{j,p}||_{(1,1)} <\infty$, the smallest eigenvalue of $\Psi_{j,p}$ and $\Phi_{j,p}$ are bounded away from $0$ with respect to both $j$ and $p$,  and $\{\eta_{t,p}\}$ is as described after Model (\ref{eqn:model}). Then, (\ref{eqn: ARMA}) can be expressed as Model (\ref{eqn:model}) under the causality condition
$\sum_{j=0}^{q} \Psi_{j,p(n)}z^{j} \neq 0\ \forall\ z\in \mathbb{C},\  |z| < 1+\epsilon$, for some $\epsilon >0$. 
For details, see \cite{BD2009} and \cite{BB2013a}.   We assume these causality conditions hold. Then, $\sup_{p} ||A_{j,p}||_{(1,1)}  <\infty$ and $\{A_{j,p}\}$ decays exponentially with $j$. 
Hence, $\gamma_p =\beta_p = O(1)$ and they are also bounded away from $0$. Therefore, (A2), (A7) and the first condition of (A8) are met.  In addition, (SNR) becomes $\frac{n}{p} ||\mu_1-\mu_2||_2^2 \to \infty$.  Finally, (A1) and (A6) are satisfied if conditions (k) and (l1),  stated after Propositions \ref{prop: 1} and \ref{prop: 2},  hold respectively. 
\end{Example}

\begin{Example}  \label{example: 1}
\textbf{(Separable cross-sectional and time dependence)}
Often, dependence among panels and dependence across time can be separated. For example,
let $A_{j,p} = \theta^{j} B_p$ for all $j \geq 0$ and for some $|\theta| < 1$, $p \times p$ matrix $B_p$, where $\{B_p\}$ is a sequence of nested matrices. It is easily observed that the time dependence is controlled by $\theta$ whereas cross-sectional dependence is absorbed into $B_p$. As these two types of dependency are controlled by two  different parameters, this model is said to exhibit {\em separable} cross-sectional and time dependence.  Here,  $||A_{j,p}||_2 = |\theta|^{j}  ||B_p||_{2}$ and  $||A_{j,p}||_{(1,1)} = |\theta|^{j}  ||B_p||_{(1,1)}$. Therefore, $\gamma_p  = (1-|\theta|)^{-1} ||B_p||_{2}$ and  $\beta_p =  (1-|\theta|)^{-1} ||B_p||_{(1,1)}$. Thus,  $\gamma_p$ and $\beta_p$ are of orders $||B_p||_2$ and $||B_p||_{(1,1)}$ respectively and they characterize only cross-sectional dependence. The following statements hold: 
\vskip 2pt
\noindent (a) (SNR) is satisfied if $\frac{n}{p||B_p||_{2}^2} ||\mu_1-\mu_2||_2^2 \to \infty$. 
\vskip 2pt
\noindent (b) (A1) and (A6) hold  if (a)-(j) in Propositions \ref{prop: 1} and \ref{prop: 2} are satisfied after replacing $\gamma_p$ and $\beta_p$ by $||B_p||_{2}$ and $||B_p||_{(1,1)}$ respectively.  More simplification is given in (i) and (ii) below. 
\vskip 2pt
\noindent (c) (A2) and (A7) hold if the smallest  singular value of $B_p$ is bounded away from $0$ and $|\theta|<1$.
\vskip 2pt
\noindent (d) The first condition in (A8) is satisfied if $\frac{\sup_{k \in \mathcal{K}_0} \sup_{1 \leq i \leq p} |B_p(i,k)|}{\inf_{k \in \mathcal{K}_0} \inf_{1 \leq i \leq p} |B_p(i,k)|} = o(p^{1/6})$. 
\vskip 2pt
\noindent The above statements provide simplified assumptions for Theorems \ref{thm: conrate} and \ref{thm: asympdist} to hold. 
\vskip 2pt
\noindent Next, we discuss some special choices for $B_p$ which often arise in practice. 
\vskip 2pt
\noindent (i) \textbf{Coefficient matrices with equal off-diagonal elements}.  Coefficient matrices with equal off-diagonal entries arise whenever the covariances between  $\{\varepsilon_{kt}\}$ and $\{\eta_{k^\prime t^\prime}\}$, $k \neq k^\prime$, do not depend on $k$ and $k^\prime$; i.e., when these covariances are all equal.  Suppose $B_p = (1-\rho)I_p + \rho J_p$ where $0<\rho <1$,  $I_p$ is the identity matrix of order $p$ and $J_p$ is the $p \times p$ matrix with all elements $1$. Then $||B_p||_{2} = ||B_p||_{(1,1)} = (1-p)\rho + 1$. Thus, (SNR) becomes $\frac{n}{p^3} ||\mu_1-\mu_2||_2^2 \to \infty$. In Propositions \ref{prop: 1} and \ref{prop: 2},  condition (c) implies conditions (a) and (b). Also as $\gamma_p^{-1}\beta_p =O(1)$,  (A1) and (A6) hold if (k) and (l), stated after Propositions \ref{prop: 1} and \ref{prop: 2},  are satisfied respectively.
As $B_p$ is positive definite with smallest eigenvalue $(1-\rho)>0$, Assumptions (A2) and (A7) hold if $|\theta|<1$. Moreover, the  condition of (A8) is satisfied, since $\sup_{k \in \mathcal{K}_0} \sup_{1 \leq i \leq p} |B_p(i,k)| \leq 1$ and $\sup_{k \in \mathcal{K}_0} \sup_{1 \leq i \leq p} |B_p(i,k)|\geq \rho$.
\vskip 2pt
\noindent (ii) \textbf{Toeplitz coefficient matrices}.  Toeplitz  matrices have a wide range of applications in many fields, including engineering, economics and biology for modeling and analysis of stationary stochastic processes. A Toeplitz structure in coefficient matrices can arise naturally, when the covariances between $\{\varepsilon_{kt}\}$ and $\{\eta_{k^\prime t^\prime}\}$ depends only on $|k - k^\prime|$ and $|t-t^\prime|$. 
 Suppose $B_p = ((t_{|i-j|}))_{p \times p}$ where $\sum_{k \geq 0} |t_k| < \infty$ and the smallest eigenvalue of $B_p$ is bounded away from $0$. Then, $||B_p||_{2} \leq ||B_p||_{(1,1)} \leq \sum_{k \geq 0} |t_k| < \infty$. Thus, $\gamma_p = O(1)$, $\beta_p = O(1)$ and they are bounded away from $0$.  Hence, (SNR) becomes $\frac{n}{p} ||\mu_1-\mu_2||_2^2 \to \infty$, which is the same SNR required for independence across both time and panels. (A1) and (A6) holds if (k) and (l1), stated after Propositions \ref{prop: 1} and \ref{prop: 2}, are satisfied respectively.
  As the smallest eigenvalue of $B_p$ is bounded away from $0$, (A2) and (A7) also hold if $|\theta| <1$. Moreover, this choice of $B_p$ satisfies (A8). 
\end{Example}

\begin{Example} \label{example: polydep}
\noindent Time dependence in Example \ref{example: 1}  may not decay exponentially fast.  More generally, in Example \ref{example: 1}, $\{\theta^j: j 
\geq 0\}$ can be replaced by $\{a_j: j \geq 0\}$ where $0< \sum_{j=0}^{\infty}|a_j| <\infty$.  For example, polynomially decaying time dependence given by $A_j =a_j B_p$, $a_j = (j+1)^{-k},\ k > 1,\ j \geq 0$ is also allowed. In this case,  (a)-(d) in Example \ref{example: 1} continue to hold once we replace $|\theta|<1$ by $0< \sum_{j=0}^{\infty}|a_j| <\infty$. 
\end{Example}

\noindent   One may think that the structures in Examples \ref{example: 1}  and \ref{example: polydep}  are restrictive. In Example \ref{example: dominant}, we consider a significantly wider class of  coefficient matrices  which are dominated by separable cross-sectional and time dependence structure.

\begin{Example} \label{example: uncorpanel}
\textbf{Independence across panels and dependence across time}. All  previous examples deal with correlated panels. In this example, we consider the case where panels are independent, but dependence across time is present. As cross-sectional dependence is absorbed into the off-diagonal entries of the coefficient matrices, independence across panels can be modeled with diagonal coefficient matrices $A_{j,p} = \text{Diag}\{a_{j,k}:\ 1\leq k \leq p\}$. Suppose $\tilde{\gamma}_p :=  \sup_{1 \leq k \leq p} \sum_{j=0}^{\infty}  |a_{j,k}| = O(1)$. Then, by Remarks \ref{rem: baigammamain} and \ref{rem: baigammadist},  the conclusions of Theorems \ref{thm: conrate} and \ref{thm: asympdist} continue to hold for this  case if instead of (SNR), condition (SNR*) $\frac{n}{p} ||\mu_1-\mu_2||_2^2 \to \infty$ is satisfied. Moreover, (A1) and (A6) holds if (k) and (l1), stated after Propositions \ref{prop: 1} and \ref{prop: 2}, are satisfied respectively.
Further, suppose $\sup_{j \geq 0} \sup_{k \geq 1} |a_{j,k}| >0$. Then, (A2), (A7) and the first condition of (A8) hold.  In this model, a weaker assumption (A8*) (instead of  (A8)) stated in Section \ref{cor: panelind} also serves our purpose: the expressions for $\sigma_{h_1,h_2}$, $\tilde{\sigma}_{(k_1,k_2),(t_1,t_2)}$ defined in (\ref{eqn: sigma}) and (\ref{eqn: tildesigma})  can be simplified further and the processes $\{W_t\}$ and $\{A_t\}$ in Theorem \ref{thm: asympdist}(c) turn out to be independent. A detailed discussion of this model is given in Section \ref{cor: panelind}.
\end{Example}
\noindent Some more examples are deferred to Section \ref{subsec: otherexamples}.

\begin{cor}  \textbf{Connections to the results presented in \citet{bai2010common}}. \label{rem: bai}
\vskip 2pt \noindent
{\bf (A)} We next compare our results previously with those in the paper by \citet{bai2010common} that posited that
data streams $\{X_{kt}\}$ are generated according to the model in Example \ref{example: uncorpanel} (elaborated  in Section \ref{cor: panelind}) and considered the following assumptions: 

\begin{enumerate}
\item
$\sup_{k} \sum_{j=0}^{\infty} j |a_{j,k}| < \infty$ 
\item
$p^{-1/2} \sum_{k=1}^{p}(\mu_{1k} - \mu_{2k})^2 \to \infty$ 
\item
$||\mu_1 - \mu_2||_2 \to \infty\ \ \text{and}\ \ n^{-1}p\log( \log (n)) \to 0$
\end{enumerate}

The key result in that paper is that under assumptions (1)-(3)
$\lim_{n \to \infty} P(\hat{\tau}_{n}=\tau_n) =1$. Details are presented in Theorems $3.1$ and $3.2$ in \citet{bai2010common}.

In comparison,  we assume in Example \ref{example: uncorpanel} (elaborated in Section \ref{cor: panelind}) that $\sup_{k} \sum_{j=0}^{\infty}  |a_{k,j}| < \infty$, which is clearly weaker than (1).
Further, observe that assumptions (SNR*) and (2) above indicate two different regimes, since none of them implies the other. 
Moreover, note that assumption (3) above is {\em stronger} than the posited (SNR*). Therefore, Example \ref{example: uncorpanel} (elaborated in  Section \ref{cor: panelind}(a))  implies 
\citet{bai2010common}'s result under assumptions (1) and (3). 

\noindent
{\bf (B)} Suppose $\{\varepsilon_{kt}\}$ are uncorrelated. Then, $a_{j,k} = 0$ for all $j \neq 0$. Thus  $\tilde{\sigma}_{(0,0),(t_1,t_2)}$ in (\ref{eqn: tildesigma})
becomes $\sigma^{*2}   = \lim \sum_{k\in \mathcal{K}_n} (\mu_{1k}-\mu_{2k})^2 a_{0,k}^2$. 
Let $B_h$ denote the standard Brownian motion. 

Suppose that $\{\varepsilon_{kt}\}$  are uncorrelated,  $||\mu_1 - \mu_2||_2 \to c>0$, $n^{-1}p\log( \log (n)) \to 0$  and  that $\sigma^*$ exists. Then, \citet{bai2010common} also established in Theorem $4.2$ that
\begin{eqnarray} \label{eqn: 24}
n (\hat{\tau}_{n} -\tau) \stackrel{\mathcal{D}}{\to}   \arg \max_{h \in \mathbb{Z}} (-0.5 |h|c + \sigma^* B_{h}).
\end{eqnarray}

To derive (\ref{eqn: 24}),  one needs to establish the asymptotic normality of $\sum_{k=1}^{m}(\mu_{1k} - \mu_{2k})\varepsilon_{kt}$, presented at the end of the first column on page $90$ in \citet{bai2010common}. To that end, further assumptions need to be imposed on $\mu_1$ and $\mu_2$ in addition to $||\mu_1 - \mu_2||_2 \to c>0$, as has been previously discussed around (\ref{eqn: asymnorlse})-(\ref{eqn: lyaplsedom}). 
However, such assumptions appear to be missing in \citet{bai2010common}. 

Finally, consider all the assumptions stated  before (\ref{eqn: 24}) in (B). 
Further, assume the weaker condition $pn^{-1} \to 0$ instead of $n^{-1}p\log( \log (n)) \to 0$.  Recall the sets $\mathcal{K}_0$ and $\mathcal{K}_n$ from (\ref{eqn: partition}).  
We additionally need (A9) and $\mathcal{K}_0$ as the empty set  for (\ref{eqn: 24}) to hold.  
Section \ref{cor: panelind}(c) provides a more general result for the model in Example \ref{example: uncorpanel}  when $\{\varepsilon_{kt}\}$ are not necessarily uncorrelated. 
\end{cor}

\section{Data Driven Adaptive Inference} \label{sec: adap}

The results in Theorem \ref{thm: asympdist} identify three different limiting regimes for the least squares estimator
$\hat{\tau}_n$ that are determined by the norm difference of the model parameters before and after the change point. The latter norm difference is {a priori unknown}, thus posing a dilemma for the practitioner of which regime to use for construction of confidence intervals for the change point. Next, we present a
{\em data driven adaptive procedure} to determine  the quantiles of the asymptotic distribution of the change point, irrespective of the specific regime pertaining to the data at hand. 
\vskip 2pt
\noindent Let 
\begin{eqnarray}
\hat{\mu}_{1} = \frac{1}{n\hat{\tau}_n} \sum_{t=1}^{n\hat{\tau}_n} X_{t,p},\ \ \ \hat{\mu}_{2} = \frac{1}{n(1-\hat{\tau}_n)} \sum_{t=n\hat{\tau}_n + 1}^{n} X_{t,p}\ \ \text{and}\ \ \hat{b}_t = \begin{cases} \hat{\mu}_1\ \ \text{if $t \leq n\hat{\tau}_n$} \\
\hat{\mu}_2\ \ \text{if $t > n\hat{\tau}_n$}. 
\end{cases} \nonumber 
\end{eqnarray}
Define the sample autocovariance matrix of order $u$ by
\begin{eqnarray}
\hat{\Gamma}_u = \frac{1}{n} \sum_{t=1}^{n-u} (X_{t,p}- b_t)(X_{t+u,p}- b_{t+u})^\prime\ \ \forall u \geq 0. \nonumber 
\end{eqnarray}
For any matrix $M$ of order $p$ and $l>0$, the \textit{banded} version of $M$
is  $$B_l (M)=((m_{ij}I(|i-j|\leq l))).$$ 
It is known in the literature that $\hat{\Gamma}_u$ is not consistent for $\Gamma_u$ in $||\cdot||_2$. 
However, a suitable banded version of $\hat{\Gamma}_u$ is consistent when the coefficient matrices $\{A_j\}$ belong to an appropriate parameter space. See \cite{BB2013a} for a detailed discussion on the estimation of autocovariance matrices. 
A discussion of the parameter space for consistency is provided after stating the required SNR condition for adaptive inference (SNR-ADAP), later in this section. 

For $p$-dimensional vectors $Y_1,Y_2,\ldots,Y_n$, let $\text{Vec}(Y_1,Y_2,\ldots,Y_n)$ be the vector of dimension $pn$, built by stacking $Y_1,Y_2,\ldots,Y_n$; $((B_{l_n}(\hat{\Gamma}_{|t_1-t_2|}) ))_{1\leq i,j \leq n}$ is a $np \times np$ matrix, comprising of $n^2$-many $p \times p$ block matrices, with the $(i,j)$-th block being $B_{l_n}(\hat{\Gamma}_{|i-j|})$. 
\vskip 2pt
\noindent Generate data from the process $\{\varepsilon_{t,p,\text{ADAP}}:\ 1 \leq t \leq n\}$ that satisfies
\begin{eqnarray}
\text{Vec}(\varepsilon_{1,p,\text{ADAP}},\varepsilon_{2,p,\text{ADAP}},\ldots,\varepsilon_{n,p,\text{ADAP}}) \sim \mathcal{N}(0, ((B_{l_n}(\hat{\Gamma}_{|i-j|}) ))_{1\leq i,j \leq n} ),\ \ \ l_n = \left(\frac{\log p}{n} \right)^{-\frac{1}{(2+\alpha)}}.  \nonumber
\end{eqnarray}
Irrespective of the probability distribution of the originally observed data $\{X_t\}$, we always generate data from the above Gaussian process in the proposed adaptive inference procedure. Define $X_{t,p,\text{ADAP}} = \hat{b}_t + \varepsilon_{t,p,\text{ADAP}}$ and write 
$X_{t,p,\text{ADAP}} = (X_{1t,p,\text{ADAP}},X_{2t,p,\text{ADAP}},\ldots,X_{pt,p,\text{ADAP}})^\prime$. Obtain 
\begin{eqnarray}
\hat{h}_{\text{ADAP}} &=& \arg \min_{h \in (n(c^{*}-\tau_n),n(1-c^{*}-\tau_n)} \hat{L}(h)\ \ \text{where} \nonumber \\
\hat{L}(h) &=& \frac{1}{n}\sum_{k=1}^{p}\bigg[\sum_{t=1}^{n\hat{\tau}_n +h}(X_{kt,p,\text{ADAP}}-\hat{\mu}_{1k})^2  + \sum_{t=n\hat{\tau}_n+h+1}^{n}(X_{kt,p,\text{ADAP}}-\hat{\mu}_{2k})^2\bigg]. \nonumber
\end{eqnarray}
The following theorem states the asymptotic distribution of $\hat{h}_{\text{ADAP}}$ under a
stronger identifiability condition (SNR-ADAP) together with consistency of $B_{l_n}(\hat{\Gamma}_u)$.
\vskip 2pt
\noindent \textbf{(SNR-ADAP)} $\frac{n\gamma_p^{-2}}{p \log p} ||\mu_1-\mu_2||_2^2 \to \infty$
\vskip 2pt
\noindent We first define the appropriate parameter spaces for $\{A_j\}$ where consistency of $B_{l_n}(\hat{\Gamma}_u)$ can be achieved. To impose restrictions on the parameter space, define for
 any nested sequence of matrices $\{M_p = ((m_{ij}))_{p \times p}\}$,
\begin{eqnarray*}
||M||_{(1,1)}& =& \max_{j \geq 1} \sum_{i\geq 1} |m_{ij}|\ \ \text{and}\ \   
T(M,t)   = \max_{j\geq 1} \sum_{i:|i-j|>t}|m_{ij}|.
\end{eqnarray*}
$\{M_p\}$ is said to have polynomially decaying corner if $T(M,t) \leq Ct^{-\alpha}$ for some $C,\alpha>0$ and for all large $t$.
\noindent Let $\gamma_p^{-1}\max(||A_j||_{(1,1)}, ||A_j^{\prime}||_{(1,1)})=r_j, j \geq 0.$ 
We define the following class of $\{A_j\}_{j=0}^{\infty}$ for some $0<\beta<1$ and $\lambda\geq 0$,
$$\Im(\beta,\lambda)=\bigg\{\{A_j\}_{j=0}^{\infty}
:\sum_{j=0}^{\infty} r_j^{\beta} < \infty, \sum_{j=0}^{\infty}
r_j^{2(1-\beta)} j^{\lambda} <\infty\bigg\}$$
which ensures that the  dependence between $X_t$ and $X_{t+\tau}$ decreases with the lag $\tau$.
Note that summability implies that the decay rate of $r_j$ cannot be slower than a polynomial rate. In case of a finite order moving average process,  as we have a finite number of norm bounded parameter matrices, $\{A_j\}$ will automatically belong to $\Im(\beta,\lambda)$ for all $0<\beta<1$ and $\lambda\geq 0$.  

 For any $1 \leq i \leq p$, let $X_{i,t,p}$ be the $i$-th coordinate of the vector $X_{t,p}$. Next, we ensure that for any 
$t_1<t$ and $k>0$, the dependence between $X_{(i\pm k),t_1,p}$ and $X_{i,t,p}$ becomes weaker for larger lag $k$. We achieve this by putting restrictions over $\{T(\gamma_p^{-1}A_j,t):j=0,1,2,\dots\}$ for all $t>0$. 
Consider the following class for some $C,\alpha,\nu>0$ and $0<\eta<1$ as
\begin{eqnarray}
\mathcal{G}(C,\alpha,\eta,\nu) &=& \big\{\{A_{j}\} :
T\big(\gamma_p^{-1} A_j,t\sum_{u=0}^{j}\eta^{u}\big)<Ct^{-\alpha}r_j
j^\nu\sum_{u=0}^{j}{\eta^{-u\alpha}},\ \ \text{and}\nonumber \\
  && \hspace{6 cm} \sum_{j=k}^\infty
\frac{r_j r_{j-k} j^\nu}{\eta^{\alpha j}}<\infty\big\}. \label{eqn: Im13.1} \nonumber
\end{eqnarray}
Though the conditions in $\mathcal{G}(C,\alpha,\eta,\nu)$ are very technical, but they are satisfied if $\{r_j\}$'s are decaying exponentially fast and for all $j \geq 0$, $\gamma_p^{-1}A_j$ has polynomially decaying corner. 
 For VAR and VARMA processes, if all parameter matrices have polynomially decaying corner, then they satisfy the condition of $\mathcal{G}(C,\alpha,\eta,\nu)$. For details see \cite{BB2013a}. 

\noindent To establish the result, we introduce next the following assumptions.
\vskip 5pt
\noindent \textbf{(C1)} $\{A_j\} \in \Im(\beta,\lambda) \cap \mathcal{G}(C,\alpha,\eta,\nu)$ for some $0 <\beta, \eta <1$ and $C, \lambda, \alpha,\nu >0$. 
\vskip 5pt
\noindent Let $\eta_{kt,p}$ be the $k$-th coordinate of $\eta_{t,p}$.  
\vskip 5pt
\noindent \textbf{(C2)} $\sup_{j \geq 1} E(e^{\lambda
|\eta_{kt,p}|})<  C_1 e^{C_2\lambda^2} < \infty\ \  \text{for all }\ \lambda \in \mathbb{R}$ and for some $C_1,C_2>0$.
\vskip 5pt
\noindent \textbf{(C3)}  $\log p = o(n)$.
\vskip 2pt
\noindent By  Theorem $4.1$ of \cite{BB2013a}, if (C1), (C2) and (C3) hold, then  
\begin{eqnarray}
||B_{l_n}(\hat{\Gamma}_u)-\Gamma_u||_2= O_p(l_n^{-\alpha})\ \ \text{for all}\  u\ \text{and}\ \ l_n=\left(n^{-1}{\log p}\right)^{-\frac{1}{2(\alpha+1)}}.  \label{eqn: gammarate}
\end{eqnarray}
\vskip 5pt
\noindent The following theorem provides the limiting distribution of $\hat{h}_{\text{ADAP}}$.  Its proof is given in Section \ref{subsec: adap}. 
\begin{theorem} \label{thm: adap}
Suppose SNR-ADAP, (C1), (C2) and (C3)  hold. 
\vskip 2pt
\noindent (a)  If $\gamma_p^{-2}||\mu_1 -\mu_2||_2^2 \to \infty$, then $\lim_{n \rightarrow \infty}\,P(\hat{h}_{\text{ADAP}} = 0)=1$. 
\vskip 5pt
\noindent (b) If $\gamma_p^{-2}||\mu_1 -\mu_2||_2^2 \to 0$ and (A1), (A2) and (A3) hold, then $\gamma_p^{-2}||\mu_1-\mu_2||_2^2 \hat{h}_{\text{ADAP}}$ converges in distribution to the same limit as $n\gamma_p^{-2}||\mu_1-\mu_2||_2^2(\hat{\tau}_n-\tau_n)$ does in Theorem \ref{thm: asympdist}(b). 

\noindent (c) If $\gamma_p^{-2} ||\mu_1-\mu_2||_2^2 \to c>0$ and (A4)-(A9) hold, then $\hat{h}_{\text{ADAP}}$ converges in distribution to the same limit as $n(\hat{\tau}_n-\tau_n)$ does in Theorem \ref{thm: asympdist}(c). 
\end{theorem}

Note that the asymptotic distribution of $\hat{h}_{\text{ADAP}}$ is identical to the asymptotic distribution of $\hat{\tau}_n$. Therefore, in practice we can simulate $\hat{h}_{\text{ADAP}}$ for a large number of replicates and calculate the quantiles of the resulting empirical distribution. The above Theorem guarantees that the empirical quantiles will be accurate estimates of the quantiles of the limiting distribution under the true regime.  Note that the proposed procedure is trivially parallelizable, which controls its computational cost. However, the procedure requires a stronger (SNR-ADAP) condition, together with assumptions (C1), (C2) and (C3), which represents the price we pay for not knowing the exact regime.

\section{Performance Evaluation} \label{sec: simulation}

Next, we investigate the performance of the least squares estimator $\hat{\tau}_n$ on synthetic data, and also undertake a comparison with the one introduced in \citet{cho2016change}, henceforth denoted by DC$_n$.

\noindent \textbf{Models considered:} Model (\ref{eqn:model}) is examined with the following two choices for the error process
$\{\varepsilon_t\}$. 
\vskip 3pt
\noindent \textbf{Model 1: ARMA$(1,1)$ process}. Let $\varepsilon_t - B_1 \varepsilon_{t-1} = \eta_t + B_2 \eta_{t-1}$,
where $B_1 = 0.25 (((0.3)^{|i-j|}))_{p \times p}$ and $B_2 = (((0.5)^{|i-j|}))_{p \times p}$. It is easy to see that $||B_1||_{2} < 1$ and hence the process $\{\varepsilon_t\}$ is causal. Therefore, it can be represented as an MA$(\infty)$ process with exponentially decaying spectral norm of the coefficient matrices. Moreover, this process is also a geometrically decaying $\alpha$-mixing process. 
\vskip 3pt
\noindent \textbf{Model 2: MA$(\infty)$ process with polynomially decaying coefficient}. Let $\varepsilon_t = \sum_{j=0}^{1000} \frac{1}{j^2} B_3 \eta_{t-j}$. We consider two choices for $B_3$ --- Model 2a: $B_3 = (((0.5)^{|i-j|}))_{p \times p}$ and Model 2b: $B_3 = 0.5I_p + 0.5 J_p$, where $I_p$ denotes the identity matrix of order $p$ and $J$ is a $p\times p$ matrix with all entries equal to $1$. In both models, the coefficient matrices are polynomially decaying.  In addition, $\gamma_p$ is bounded away from $0$,  bounded above for Model 2a and of order $p$ for Model 2b. It is then easy to see that they are not geometrically decaying $\alpha$-mixing processes, and therefore, are not amenable to the procedure in \citet{cho2016change}.  Although no theoretical results are established for this model in the latter paper, the ensuing simulation results render support to the empirical fact that the change point estimator in \citet{cho2016change} behaves similarly to the least squares estimator introduced in this study. 
 \vskip 3pt
 \noindent Define
 \begin{eqnarray}
 F_{1n}  = \frac{\sqrt{n}}{p\log n} \sum_{k=1}^{p}|\mu_{1k}-\mu_{2k}|\ \ \text{and}\ \  F_{2n} =\begin{cases}
 \frac{n}{p} \sum_{k=1}^{p}|\mu_{1k}-\mu_{2k}|^2 \ \ \text{for Models 1 and 2a} \\
 \frac{n}{p^3} \sum_{k=1}^{p}|\mu_{1k}-\mu_{2k}|^2\ \ \ \text{for Model 2b}.
 \end{cases}
 \end{eqnarray}
 Note that \cite{cho2016change} requires $F_{1n} \to \infty$, while the least squares estimator requires $F_{2n} \to \infty$. 
 \vskip 3pt
\noindent Throughout this section we consider $\tau_n = 0.5$. The choices for $\mu_1, \mu_2$ and the probability distribution of $\{\eta_{kt}\}$ are specified in the ensuing subsections, depending on the objective under consideration. 
To better understand  the performance of the change point estimators, 
we also mention the value of SNR ($F_{1n}$ for DC$_n$ and $F_{2n}$ for $\tau_n$) and the signal $s = \gamma_{p}^{-2}||\mu_1-\mu_2||_F^2$.  Consistency of the change point estimators needs the SNR to go to infinity. Moreover, asymptotic distribution of the change point estimators depends on $s$.

\noindent \textbf{(A) Effect of the (SNR) condition:} \label{subsec: snreffect}
We simulate from $\eta_{kt} \stackrel{i.i.d.}{\sim} \mathcal{N}(0,1)$, with $p =\sqrt{n}$. Further, we consider the following two choices for $\mu_1$ and $\mu_2$. 
\vskip 2pt
\noindent (i) $\mu_{1k} = 0$ and  $\mu_{2k} = k^{-1}\ \forall k$. In this case, $F_{1n}$ is bounded and therefore the identifiablity condition of \citet{cho2016change} is not satisfied.  Moreover, $F_{2n} > \sqrt{n} \to \infty$, which implies (SNR) holds for Models 1 and 2a. However, since $F_{2n} < C$ for some $C>0$, (SNR) is not satisfied for Model 2b. 

\vskip 2pt
\noindent (ii) $\mu_{1k} = 0$ and  $\mu_{2k} = \frac{p}{n^{1/4}}\ \forall k$. In this case, $F_{1n} \geq \frac{pn^{1/4}}{\log n} \to \infty$ and  $F_{2n} \geq \sqrt{n} \to \infty$.  Therefore, both the identifiablity condition in \citet{cho2016change} and (SNR) hold.  

\noindent Tables \ref{table: 1} and \ref{table: 2} depict the performance of the change point estimators in Cases (i) and (ii), respectively. In Table \ref{table: 1}, $\hat{\tau}_n$ performs badly in Model 2b and DC$_n$ does not perform well either, since the required SNR condition is not satisfied in either case. On other cells in Table \ref{table: 1}, (SNR) is satisfied and consequently $\hat{\tau}_n$ exhibits good performance. In Table \ref{table: 2} for Models 1 and 2a, $\hat{\tau}_n$ and DC$_n$ estimate $\tau_n$ very well and fairly accurately in the presence of a very high signal. In Table \ref{table: 2}  for Model 2b, the magnitude of the signal is moderate and therefore $\hat{\tau}_n$ and DC$_n$ do not perform as accurately as for Models 1 and 2a, but their performance is overall satisfactory. \\\\
\begin{center}
\begin{table} [h!] 
\hspace{5 cm}\caption{Mean value of the change point estimators with $\tau_n = 1/2$,  $\mu_{1k} = 0$,  $\mu_{2k} = k^{-1}\ \forall k$  and $p = \sqrt{n}$ and $50$ replications. $\eta_{kt} \stackrel{i.i.d.}{\sim} \mathcal{N}(0,1)$. Figures in brackets are the standard deviation of the change point estimators over $50$ replicates. SNR: $F_{1n}$ for DC$_n$ and $F_{2n}$ for $\hat{\tau}_n$, 
$s = \gamma_{p}^{-2}||\mu_1-\mu_2||_2^2$.}  \label{table: 1}
\vskip 2pt
\begin{tabular}{|m{1.9cm}|m{1.3cm}|m{1.2cm}|m{1.3cm}|m{1.2cm}|m{1.3cm}|m{1.2cm}|m{1.3cm}|m{1.2cm}|}
\hline 
 &  \multicolumn{2}{|c|}{$n=500$,  $p=23$} & \multicolumn{2}{|c|}{$n=1000$,  $p=32$} & \multicolumn{2}{|c|}{$n=5000$, $p=71$} & \multicolumn{2}{|c|}{$n=10000$, $p=100$} \\ 
\hline 
 & $\hat{\tau}_n$ & $\text{DC}_n$ & $\hat{\tau}_n$ & $\text{DC}_n$ &  $\hat{\tau}_n$ & $\text{DC}_n$ & $\hat{\tau}_n$ & $\text{DC}_n$\\
\hline
Model 1 & 0.59 (0.074) & 0.78 (0.44) & 0.536 (0.039)  & 0.85 (0.36) & 0.4936 (0.0068)  & 0.67 (0.42)  & 0.5024 (0.0028) & 0.83 (0.44) \\ 
\hline 
SNR (Model 1)  & 22.36 & 0.6 & 31.62 & 0.588 & 70.71 & 0.569 & 100 & 0.563\\
\hline
$s$ $\text{(Model 1)}$ & 1.602 & 1.602 & 1.614 & 1.614 &  1.631 & 1.631 & 1.635  & 1.635 \\ 
\hline 
Model 2a & 0.584 (0.083) & 0.68 (0.37)  & 0.461 (0.037) & 0.82 (0.41) & 0.5068 (0.0072)  &  0.75 (0.33) & 0.4974 (0.0031)  & 0.79 (0.39)\\ 
\hline
SNR $\text{(Model 2a)}$  & 22.36 & 0.6 & 31.62 & 0.588 & 70.71 & 0.569 & 100 & 0.563 \\
\hline
$s$ $\text{(Model 2a)}$ & 1.602 & 1.602 & 1.614 & 1.614 &  1.631 & 1.631 & 1.635  & 1.635 \\ 
\hline 
Model 2b & 0.83 (0.46) & 0.69 (0.34) & 0.64 (0.38) & 0.78 (0.42)  & 0.72 (0.47) & 0.67 (0.32) & 0.81 (0.41) & 0.84 (0.48) \\ 
\hline
SNR $\text{(Model 2b)}$  & 0.0447 & 0.6 & 0.0316 & 0.588 & 0.01414 & 0.569 & 0.01 & 0.563\\
\hline
$s$ $\text{(Model 2b)}$ & $3.2 \times 10^{-3}$ & $3.2 \times 10^{-3}$ & $1.61 \times 10^{-3}$ & $1.61 \times 10^{-3}$ & $3.3 \times 10^{-4}$ & $3.3 \times 10^{-4}$ & $1.63 \times 10^{-4}$ & $1.63 \times 10^{-4}$\\
\hline
\end{tabular} 
\end{table}
\end{center}

\begin{center}
\begin{table} [h!] 
\hspace{5 cm}\caption{Mean value of the change point estimators with $\tau_n = 1/2$,  $\mu_{1k} = 0$,  $\mu_{2k} = \frac{p}{n^{1/4}}\ \forall k$  and $p = \sqrt{n}$ and $50$ replications. $\eta_{kt} \stackrel{i.i.d.}{\sim} \mathcal{N}(0,1)$.  Figures in brackets are the standard deviation  of the change point estimators over $50$ replicates. SNR: $F_{1n}$ for DC$_n$ and $F_{2n}$ for $\hat{\tau}_n$, 
$s = \gamma_{p}^{-2}||\mu_1-\mu_2||_2^2$.} \label{table: 2} 
\vskip 2pt
\begin{tabular}{|m{1.9cm}|m{1.2cm}|m{1.2cm}|m{1.2cm}|m{1.2cm}|m{1.2cm}|m{1.2cm}|m{1.2cm}|m{1.2cm}|}
\hline 
 &  \multicolumn{2}{|c|}{$n=500$,  $p=23$} & \multicolumn{2}{|c|}{$n=1000$,  $p=32$} & \multicolumn{2}{|c|}{$n=5000$, $p=71$} & \multicolumn{2}{|c|}{$n=10000$, $p=100$} \\ 
\hline 
 & $\hat{\tau}_n$ & $\text{DC}_n$ & $\hat{\tau}_n$ & $\text{DC}_n$ &  $\hat{\tau}_n$ & $\text{DC}_n$ & $\hat{\tau}_n$ & $\text{DC}_n$\\
 \hline
Model 1 & 0.4985 ($16 \times 10^{-3}$)   & 0.5018 ($22 \times 10^{-3}$) & 0.5012 ($4.9 \times 10^{-3}$)  & 0.5016 ($9 \times 10^{-3}$)  & 0.5008 ($13.6 \times 10^{-4}$) & 0.499 ($15.2\times 10^{-4}$)   & 0.5002 ($3 \times 10^{-4}$) & 0.5002 ($3.5 \times 10^{-4}$) \\ 
\hline 
SNR (Model 1)  & $11.18 \times 10^{3}$ & 17.014 & $31.622 \times 10^3$ & 25.743 & $35.355 \times 10^{4}$ &  69.812 & $10^{6}$ & 108.574\\
\hline
$s$ $\text{(Model 1)}$ & 500 & 500 & 1000 & 1000 & 5000 & 5000 & 10000  & 10000\\ 
\hline
Model 2a &  0.5011 ($12 \times 10^{-3}$)  & 0.5015 ($24 \times 10^{-3}$) & 0.501 ($7 \times 10^{-3}$)  & 0.5014 ($10 \times 10^{-3}$)  & 0.5006 ($12.4 \times 10^{-4}$) & 0.4992 ($13.6 \times 10^{-4}$)  & 0.5001 ($2.6 \times 10^{-4}$) & 0.5002 ($3.2 \times 10^{-4}$)   \\ 
\hline 
SNR (Model 2a)  & $11.18 \times 10^{3}$ & 17.014 & $31.622 \times 10^3$ & 25.743 & $35.355 \times 10^{4}$ &  69.812 & $10^{6}$ & 108.574\\
\hline
$s$ $\text{(Model 2a)}$ & 500 & 500 & 1000 & 1000 & 5000 & 5000 & 10000  & 10000 \\ 
\hline 
Model 2b & 0.405 (0.123)   & 0.62 (0.1) & 0.552 (0.058)  & 0.43 (0.072)  & 0.4911  (0.013)  & 0.5107 (0.0538)  & 0.5042 (0.007) & 0.5088 (0.0489)  \\ 
\hline 
SNR (Model 2b)  & 22.36 & 17.014 & 31.62 & 25.743 & 70.71 & 69.812  & 100  & 108.574 \\
\hline
$s$ $\text{(Model 2b)}$ &0.95 & 0.95 & 0.98 & 0.98 & 0.99 & 0.99  & 1  & 1\\ 
\hline
\end{tabular} 
\end{table}
\end{center}

\noindent \textbf{(B) Effect of dimension $p$ and sample size $n$:}
Next, we simulate from the same setup as in (A)(ii) and consider $p=e^{n^{7/32}}$. Clearly, both $F_{1n}$ and $F_{2n}$ go to $\infty$. However, the DC$n$ estimator additionally requires $p \sim n^\delta$ for some $\delta \in [0,\infty)$, which is not satisfied when $p=e^{n^{7/32}}$. Nevertheless, the results in Table \ref{table: 3} suggest that DC$_n$ performs as well as $\hat{\tau}_n$ when  $p=e^{n^{7/32}}$. 

\begin{center}
\begin{table} [h!] 
\hspace{5 cm}\caption{Mean value of the change point estimators with $\tau_n = 1/2$,  $\mu_{1k} = 0$,  $\mu_{2k} = \frac{p}{n^{1/4}}\ \forall k$  and $p=e^{n^{7/32}}$ and $50$ replications. $\eta_{kt} \stackrel{i.i.d.}{\sim} \mathcal{N}(0,1)$.  Figures in brackets are the standard deviation of the change point estimators over $50$ replicates. SNR: $F_{1n}$ for DC$_n$ and $F_{2n}$ for $\hat{\tau}_n$,  
$s = \gamma_{p}^{-2}||\mu_1-\mu_2||_2^2$.} \label{table: 3} 
\vskip 2pt
\begin{tabular}{|m{1.9cm}|m{1.2cm}|m{1.2cm}|m{1.2cm}|m{1.2cm}|m{1.2cm}|m{1.2cm}|m{1.2cm}|m{1.2cm}|}
\hline 
 &  \multicolumn{2}{|c|}{$n=500$,  $p=50$} & \multicolumn{2}{|c|}{$n=1000$,  $p=93$} & \multicolumn{2}{|c|}{$n=5000$, $p=629$} & \multicolumn{2}{|c|}{$n=10000$, $p=1807$} \\ 
\hline 
 & $\hat{\tau}_n$ & $\text{DC}_n$ & $\hat{\tau}_n$ & $\text{DC}_n$ &  $\hat{\tau}_n$ & $\text{DC}_n$ & $\hat{\tau}_n$ & $\text{DC}_n$\\
\hline
Model 1 & 0.5016 ($19 \times 10^{-3}$) &  0.5016 ($23 \times 10^{-2}$)  & 0.4988 ($9 \times 10^{-3}$) & 0.5014 ($12 \times 10^{-3}$) & 0.5006 ($9.6 \times 10^{-4}$) & 0.5006 ($13.2 \times 10^{-4}$) & 0.4999 ($2.1 \times 10^{-4}$) & 0.5002 ($1.9 \times 10^{-4}$)\\ 
\hline 
SNR (Model 1)  & $5.59 \times 10^{4}$ & 38.045  & $2.74 \times 10^5$ & 75.709 & $27.98 \times 10^{6}$ & 621.007  & $32.65 \times 10^{7}$ & 1961.925\\
\hline
$s$ $\text{(Model 1)}$ & 5590.17 & 5590.17 & 25436 & 25436 & $3519 \times 10^3$ & $3519 \times 10^3$  & $59 \times 10^6$ & $59 \times 10^6$\\ 
\hline 
Model 2a & 0.5014 ($12 \times 10^{-3}$) & 0.5012 ($19 \times 10^{-3}$) & 0.5014 ($9.4 \times 10^{-3}$) & 0.5015 ($9.1 \times 10^{-3}$) & 0.5007 ($9.9 \times 10^{-4}$) & 0.4994 ($11.7 \times 10^{-4}$) & 0.5001 ($2.9 \times 10^{-4}$) & 0.5001 ($3.6 \times 10^{-4}$)\\ 
\hline 
SNR (Model 2a)  & $5.59 \times 10^{4}$ & 38.045  & $2.74 \times 10^5$ & 75.709 & $27.98 \times 10^{6}$ & 621.007  & $32.65 \times 10^{7}$ & 1961.925 \\
\hline
$s$ $\text{(Model 2a)}$ & 5590.17 & 5590.17 & 25436 & 25436 & $3519 \times 10^3$ & $3519 \times 10^3$  & $59 \times 10^6$ & $59 \times 10^6$ \\  
\hline 
Model 2b & 0.578 (0.074) & 0.594 (0.084) & 0.468 (0.028) & 0.538 (0.038) & 0.5044 (0.0052) & 0.4952 (0.0048) & 0.5017 (0.0012) & 0.4985 (0.0018) \\ 
\hline 
SNR (Model 2b)  & 22.36 & 38.045 & 31.62 & 75.709 & 70.71 & 621.007 & 100 & 1961.925\\
\hline
$s$ $\text{(Model 2b)}$ & 2.24 & 2.24 & 2.94 & 2.94 & 8.895 & 8.895 & 18.07  & 18.07\\ 
\hline 
\end{tabular} 
\end{table}
\end{center}


\noindent \textbf{(C) Effect of moments:}
Next, we simulate from $\eta_{kt} \stackrel{i.i.d.}{\sim} (\textbf{Beta}(4,1))^{-1} - E(\textbf{Beta}(4,1))^{-1}$ and consider $n$, $p$, $\mu_1$ and $\mu_2$ to be as in settings (A)(ii) and (B). Recall that the DC$n$ estimator requires finite moments of all orders, whereas $\hat{\tau}_n$ only requires a $4$-th order moment. The above choice for $\eta_{kt}$ has finite moments up to order $4$. Tables \ref{table: 4} and \ref{table:five} illustrate the performance of both change point estimators. As expected, DC$_n$ exhibits an inferior performance, whereas $\hat{\tau}_n$ estimates the change point very well, as seen in both Tables \ref{table: 4} and \ref{table:five}.

\begin{center}
\begin{table} [h!] 
\hspace{5 cm}\caption{Mean value of the change point estimators with $\tau_n = 1/2$,  $\mu_{1k} = 0$,  $\mu_{2k} = \frac{p}{n^{1/4}}\ \forall k$  and $p = \sqrt{n}$ and $50$ replications. $1/\eta_{kt} \stackrel{i.i.d.}{\sim} (\textbf{Beta}(4,1))^{-1} - E(\textbf{Beta}(4,1))^{-1}$. Figures in brackets are the standard deviation of the change point estimators over $50$ replicates. SNR: $F_{1n}$ for DC$_n$ and $F_{2n}$ for $\hat{\tau}_n$, 
$s = \gamma_{p}^{-2}||\mu_1-\mu_2||_2^2$ are as in Table \ref{table: 2}} \label{table: 4}
\vskip 2pt
\begin{tabular}{|m{1.9cm}|m{1.2cm}|m{1.2cm}|m{1.2cm}|m{1.2cm}|m{1.2cm}|m{1.2cm}|m{1.2cm}|m{1.2cm}|}
\hline 
 &  \multicolumn{2}{|c|}{$n=500$,  $p=23$} & \multicolumn{2}{|c|}{$n=1000$,  $p=32$} & \multicolumn{2}{|c|}{$n=5000$, $p=71$} & \multicolumn{2}{|c|}{$n=10000$, $p=100$} \\ 
\hline 
 & $\hat{\tau}_n$ & $\text{DC}_n$ & $\hat{\tau}_n$ & $\text{DC}_n$ &  $\hat{\tau}_n$ & $\text{DC}_n$ & $\hat{\tau}_n$ & $\text{DC}_n$\\
\hline
Model 1 & 0.5017 ($18\times 10^{-3}$) & 0.59 (0.132) & 0.5011 ($9.6\times 10^{-3}$)  & 0.62 (0.072) & 0.5008 ($14.4\times 10^{-4}$) & 0.69 (0.0494) & 0.5001 ($3.5\times 10^{-4}$) & 0.39 (0.0457) \\ 
\hline 
Model 2a & 0.5014 ($15\times 10^{-3}$) & 0.63 (0.112) & 0.4987 ($10.1\times 10^{-3}$) & 0.42 (0.069) & 0.5007 ($15.2\times 10^{-4}$) & 0.42 (0.0516) & 0.5002 ($3.3\times 10^{-4}$)  & 0.64 (0.0439)\\ 
\hline 
Model 2b & 0.612 (0.126) & 0.58 (0.119) & 0.559 (0.046) & 0.41 (0.072) & 0.4932 (0.009)  & 0.67 (0.0513) & 0.4965 (0.0056) & 0.62 (0.0445)\\ 
\hline 
\end{tabular} 
\end{table}
\end{center}

\begin{center}
\begin{table} [h!] 
\hspace{5 cm}\caption{Mean value of the change point estimators with $\tau_n = 1/2$,  $\mu_{1k} = 0$,  $\mu_{2k} = \frac{p}{n^{1/4}}\ \forall k$  and $p=e^{n^{7/32}}$ and $50$ replications. $1/\eta_{kt} \stackrel{i.i.d.}{\sim} (\textbf{Beta}(4,1))^{-1} - E(\textbf{Beta}(4,1))^{-1}$. Figures in brackets are the standard deviation of the change point estimators over $50$ replicates. SNR: $F_{1n}$ for DC$_n$ and $F_{2n}$ for $\hat{\tau}_n$,  
$s = \gamma_{p}^{-2}||\mu_1-\mu_2||_2^2$ are as in Table \ref{table: 3}} \label{table:five}
\vskip 2pt
\begin{tabular}{|m{1.9cm}|m{1.2cm}|m{1.2cm}|m{1.2cm}|m{1.2cm}|m{1.2cm}|m{1.2cm}|m{1.2cm}|m{1.2cm}|}
\hline 
 &  \multicolumn{2}{|c|}{$n=500$,  $p=50$} & \multicolumn{2}{|c|}{$n=1000$,  $p=93$} & \multicolumn{2}{|c|}{$n=5000$, $p=629$} & \multicolumn{2}{|c|}{$n=10000$, $p=1807$} \\ 
\hline 
 & $\hat{\tau}_n$ & $\text{DC}_n$ & $\hat{\tau}_n$ & $\text{DC}_n$ &  $\hat{\tau}_n$ & $\text{DC}_n$ & $\hat{\tau}_n$ & $\text{DC}_n$\\
\hline
Model 1 & 0.5018 ($17\times 10^{-3}$) & 0.62 (0.115) & 0.5014 ($11 \times 10^{-3}$) & 0.66 (0.069)  &   0.5006 ($14.3 \times 10^{-4}$) & 0.61 (0.0491)  & 0.4998 ($1.7 \times 10^{-4}$) & 0.42 (0.0457)\\ 
\hline 
Model 2a & 0.5014 ($14 \times 10^{-3}$) & 0.58 (0.119) & 0.5011 ($9 \times 10^{-3}$) & 0.56 (0.063)  & 0.4992 ($13.7 \times 10^{-4}$)  & 0.38 (0.0501)  & 0.5001 ($1.5 \times 10^{-4}$) & 0.64 (0.0442)\\ 
\hline 
Model 2b & 0.474 (0.077) & 0.64 (0.121) & 0.557 (0.019) & 0.44 (0.071) & 0.4933 (0.0047) & 0.42 (0.0515) & 0.5012 (0.0018) & 0.38 (0.0437) \\ 
\hline 
\end{tabular} 
\end{table}
\end{center}

\noindent \textbf{(D) Performance of the Asymptotic distribution of $\hat{\tau}_n$:} We consider the same $n$, $p$, $\mu_{1k}$ and probability distribution of $\{\eta_{kt}\}$ as in setting (A). Consider the following two choices for $\mu_2$: 
\vskip 2pt
\noindent (i) $\mu_{2k} = \frac{1}{n^{3/8}}$ for Models 1 and 2a,  and $\mu_{2k} = \frac{p}{n^{3/8}}$ for Model 2b.  Therefore, $F_{2n} \to \infty$, but $\gamma_{p}^{-2} ||\mu_1-\mu_2||_2^2 \to 0$. 
\vskip 2pt
\noindent (ii)  $\mu_{2k} = \frac{1}{n^{1/4}}$ for Models 1 and 2a,  and $\mu_{2k} = \frac{p}{n^{1/4}}$ for Model 2b.  Therefore, $F_{2n} \to \infty$, but $\gamma_{p}^{-2} ||\mu_1-\mu_2||_2^2 \to 1$. 
\vskip 2pt
\noindent In Tables \ref{table: 6}-\ref{table: 9}, we present  95\% confidence intervals assuming knowledge of the true limiting regime, as well as their adaptive counterparts, along with the proportion of containing $\hat{\tau}_n$ based on $100$ replications. Tables \ref{table: 6} and \ref{table: 8} report theoretical confidence intervals (TCI) which are obtained by simulating observations from the asymptotic distribution given in Theorem \ref{thm: asympdist} assuming knowledge of the true limiting regime and then computing the sample quantiles. Tables \ref{table: 7} and \ref{table: 9} present adaptive confidence intervals (ACI) obtained by the method discussed in Section \ref{sec: adap}. The simulation results show tighter confidence intervals as the sample size increases. This is due to the fact $n\gamma_p^{-2}||\mu_1-\mu_2||_2^2 \to \infty$. Finally, the performance of ACI is as good as TCI, which demonstrates the utility of the data-driven adaptive procedure.

\begin{center}
\begin{table} [h!] 
\hspace{5 cm}\caption{ 95\% Theoretical confidence interval (TCI) and empirical \% (EP) of containing 
 $\tau_n = 1/2$ with $\mu_{1k} = 0$,  $\mu_{2k} = \frac{1}{n^{3/8}}\ \forall k$ for Models 1 and 2a,  $\mu_{2k} = \frac{p}{n^{3/8}}\ \forall k$ for Models 2b  and $p = \sqrt{n}$ based on $5000$ sample paths for each  $100$ replications. $\eta_{kt} \stackrel{i.i.d.}{\sim} \mathcal{N}(0,1)$. SNR: $\frac{n}{p\gamma_p^{2}}||\mu_1-\mu_2||_2^2$, $s = \gamma_{p}^{-2}||\mu_1-\mu_2||_2^2$.} \label{table: 6} 
\vskip 2pt
\begin{tabular}{|m{1.9cm}|m{1.3cm}|m{1.2cm}|m{1.2cm}|m{1.2cm}|m{1.2cm}|m{1.2cm}|m{1.2cm}|m{1.2cm}|}
\hline 
 &  \multicolumn{2}{|c|}{$n=500$,  $p=23$} & \multicolumn{2}{|c|}{$n=1000$,  $p=32$} & \multicolumn{2}{|c|}{$n=5000$, $p=71$} & \multicolumn{2}{|c|}{$n=10000$, $p=100$} \\ 
\hline 
 & 95\% TCI & EP & 95\% TCI& EP &  95\% TCI & EP & 95\% TCI & EP\\
\hline
Model 1 & 0.2601, 0.7426 
 & 95.2 & 0.3483, 0.6294 
 & 94.8 & 0.4552, 0.5381 
 & 94.8 & 0.4736, 0.5230 
 & 95\\ 
\hline 
Model 2a & 0.2667, 0.7273 
& 94.6 & 0.3592, 0.6444 
& 95.4 & 0.4568, 0.5406 
& 95.2 & 0.4730, 0.5243 
& 94.8\\ 
\hline 
Model 2b &  0.2689, 0.7411
& 94.6 & 0.3527, 0.6345
& 94.8 & 0.4558, 0.5385
& 95.2 & 0.4731, 0.5244
& 94.9 \\ 
\hline 
SNR  & 4.73 & &  5.62 & & 8.41 & & 10 &\\
\hline
$s$ & 0.211 & & 0.178 & & 0.119 & & 0.1 & \\
\hline
\end{tabular} 
\end{table}
\end{center}

\begin{center}
\begin{table} [h!] 
\hspace{5 cm}\caption{ 95\% Theoretical confidence interval (TCI) and empirical \% (EP) of containing 
 $\tau_n = 1/2$ with  $\mu_{1k} = 0$,  $\mu_{2k} = \frac{1}{n^{1/4}}\ \forall k$ for Models 1 and 2a, $\mu_{2k} = \frac{p}{n^{1/4}}\ \forall k$ for Models 2b   and $p = \sqrt{n}$ based on $5000$ sample paths for each  $100$ replications. $\eta_{kt} \stackrel{i.i.d.}{\sim} \mathcal{N}(0,1)$. SNR: $\frac{n}{p\gamma_p^{2}}||\mu_1-\mu_2||_2^2$, $s = \gamma_{p}^{-2}||\mu_1-\mu_2||_2^2$.} \label{table: 7}
\vskip 2pt
\begin{tabular}{|m{1.9cm}|m{1.2cm}|m{1.2cm}|m{1.2cm}|m{1.2cm}|m{1.2cm}|m{1.2cm}|m{1.2cm}|m{1.2cm}|}
\hline 
 &  \multicolumn{2}{|c|}{$n=500$,  $p=23$} & \multicolumn{2}{|c|}{$n=1000$,  $p=32$} & \multicolumn{2}{|c|}{$n=5000$, $p=71$} & \multicolumn{2}{|c|}{$n=10000$, $p=100$} \\ 
\hline 
 & 95\% TCI & EP & 95\% TCI& EP &  95\% TCI & EP & 95\% TCI & EP\\
\hline
Model 1 & 0.342, 0.650 
& 95.4 & 0.423, 0.576 
& 95.2 & 0.4844, 0.5152  
& 94.8  & 0.4921, 0.5077 
& 95.2\\ 
\hline 
Model 2a &  0.346, 0.652  
& 94.4 &   0.422, 0.577    
& 94.4 &   0.4844, 0.515    
& 94.8 &  0.4923, 0.5075  
& 95\\ 
\hline 
Model 2b & 0.344, 0.646  
& 95.4 & 0.422, 0.573  
& 94.8 & 0.4846, 0.5148  
& 95.2 & 0.4924, 0.5075  
& 94.8 \\ 
\hline 
SNR  & 22.36 & &  31.62 & & 70.71 & & 100 &\\
\hline
$s$ & 1 & & 1 & & 1 & & 1 & \\
\hline
\end{tabular} 
\end{table}
\end{center}

\begin{center}
\begin{table} [h!] 
\hspace{5 cm}\caption{ 95\% Adaptive confidence interval (ACI) and empirical \% (EP) of containing 
 $\tau_n = 1/2$ with  $\mu_{1k} = 0$,  $\mu_{2k} = \frac{1}{n^{3/8}}\ \forall k$ for Models 1 and 2a, $\mu_{2k} = \frac{p}{n^{3/8}}\ \forall k$ for Models 2b   and $p = \sqrt{n}$ based on $5000$ sample paths for each  $100$ replictions. $\eta_{kt} \stackrel{i.i.d.}{\sim} \mathcal{N}(0,1)$. SNR: $\frac{n}{p\gamma_p^{2}}||\mu_1-\mu_2||_2^2$, $s = \gamma_{p}^{-2}||\mu_1-\mu_2||_2^2$.} \label{table: 8} 
\vskip 2pt
\begin{tabular}{|m{1.9cm}|m{1.2cm}|m{1.2cm}|m{1.2cm}|m{1.2cm}|m{1.2cm}|m{1.2cm}|m{1.2cm}|m{1.2cm}|}
\hline 
 &  \multicolumn{2}{|c|}{$n=500$,  $p=23$} & \multicolumn{2}{|c|}{$n=1000$,  $p=32$} & \multicolumn{2}{|c|}{$n=5000$, $p=71$} & \multicolumn{2}{|c|}{$n=10000$, $p=100$} \\ 
\hline 
 & 95\% ACI & EP & 95\% ACI& EP &  95\% ACI & EP & 95\% ACI & EP\\
\hline
Model 1 & 0.2729, 0.7472 
& 93.2 &  0.3631, 0.6589   
& 93.8 &  0.4602, 0.5439   
& 94.4 &  0.4766, 0.5257   
& 94.8\\ 
\hline 
Model 2a & 0.2777,  0.7472  
& 94.2 &  0.3707,   0.6592     
& 95.8 &   0.4585, 0.5445       
& 94.6 &  0.4759, 0.5271       
& 95\\ 
\hline 
Model 2b &  0.2763,  0.7563  
& 93.8 &  0.3709,  0.6489     
& 94.6 & 0.4568,  0.5443      
& 95.6 &   0.4744, 0.5271      
& 95.2 \\ 
\hline 
SNR  & 4.73 & &  5.62 & & 8.41 & & 10 &\\
\hline
$s$ & 0.211 & & 0.178 & & 0.119 & & 0.1 & \\
\hline
\end{tabular} 
\end{table}
\end{center}

\begin{center}
\begin{table} [h!] 
\hspace{5 cm}\caption{ 95\% Adaptive confidence interval (ACI) and empirical \% (EP) of containing 
 $\tau_n = 1/2$ with  $\mu_{1k} = 0$,  $\mu_{2k} = \frac{1}{n^{1/4}}\ \forall k$ for Models 1 and 2a, $\mu_{2k} = \frac{p}{n^{1/4}}\ \forall k$ for Models 2b   and $p = \sqrt{n}$ based on $5000$ sample paths for each  $100$ replications. $\eta_{kt} \stackrel{i.i.d.}{\sim} \mathcal{N}(0,1)$. SNR: $\frac{n}{p\gamma_p^{2}}||\mu_1-\mu_2||_2^2$, $s = \gamma_{p}^{-2}||\mu_1-\mu_2||_2^2$.} \label{table: 9} 
\vskip 2pt
\begin{tabular}{|m{1.9cm}|m{1.2cm}|m{1.2cm}|m{1.2cm}|m{1.2cm}|m{1.2cm}|m{1.2cm}|m{1.2cm}|m{1.2cm}|}
\hline 
 &  \multicolumn{2}{|c|}{$n=500$,  $p=23$} & \multicolumn{2}{|c|}{$n=1000$,  $p=32$} & \multicolumn{2}{|c|}{$n=5000$, $p=71$} & \multicolumn{2}{|c|}{$n=10000$, $p=100$} \\ 
\hline 
 & 95\% ACI & EP & 95\% ACI& EP &  95\% ACI & EP & 95\% ACI & EP\\
\hline
Model 1 & 0.348, 0.648   
& 93.8 &  0.423, 0.573    
& 94.4 &   0.4844, 0.5152    
& 94.6 &   0.4924, 0.5076     
& 95.4\\ 
\hline 
Model 2a &   0.346, 0.652     
& 94.2 &    0.422, 0.573          
& 94.4 &    0.4848, 0.5152          
& 95.4 &   0.4921, 0.5073           
& 94.8\\ 
\hline 
Model 2b &  0.346, 0.648        
& 95.4 &   0.422, 0.574             
& 94.4 &   0.4846, 0.5150             
& 95.2 &   0.4924, 0.5075             
& 95.2 \\ 
\hline 
SNR  & 22.36 & &  31.62 & & 70.71 & & 100 & \\
\hline
$s$ & 1 & & 1 & & 1 & & 1 & \\
\hline
\end{tabular} 
\end{table}
\end{center}


 \section{Application to Financial Asset Prices} \label{sec: Dataanalysis}
 
The data set comprises of weakly stock prices for 75 financial US firms - banks, insurance companies and broker-dealers
covering the period  from 1/2/2001 to 12/27/2016, containing in total $n=825$ time points. The data were retrieved from
Wharton's Research Data Service (WRDS). First, log-returns $\log (p_{t,i}/p_{t-1,i})$ were calculated, where $p_{t,i}$ denotes the stock price at time $t$ for firm $i$. The final number of firms analyzed is 66 that had sufficient data throughout the time period considered. 

Since the developed methodology focuses on a single change point, we considered the following strategy, also employed
in \cite{billio2012econometric}. 
We consider 24-month-long rolling-windows, separated by 3 months, thus applying the procedure 60 times. For each window containing 104 time points, we compute the minimizer of the least squares criterion function given in (\ref{eqn: cpdefine}). The minimizers that appear at least twice among these 60 potential minimizers, are declared to be candidate change points and reported in Table \ref{Table: DA1}. There are 9 such change points.
To calculate confidence intervals, we partition the time axis into 9 windows, starting from the mid-point of the interval contain the $(i-1)$-th and $i$-th change points to the mid-point of the $i$-th and $(i+1)$-th change points. For each of these 9 windows, we employ the adaptive method given in Section \ref{sec: adap} for computing confidence intervals and present these intervals in Table \ref{Table: DA1}. 

The identified nine change points together with their calculated confidence intervals cover time periods where major shocks occurred, thus providing evidence of the validity of the proposed methodology. Specifically, the 2002 coincides with the popping of the Internet bubble, while the 2003 marks the turning point for the broader market;
the 2005 change point and its confidence interval cover a period of rapid consecutive interest increases from 2.5\% to 4\% by the Fed; the 2007 is a few weeks after the August market turmoil induced by global credit fears and a liquidity crunch that led the Fed to shore up substantially liquidity of the US Banking system;
the 2008 change point is a few weeks before the collapse of Lehman Brothers on 9/15/2008; the 2009 one is a couple of weeks
off the bottom of the stock market following the Great Recession; the 2011 change point is a the center of the time period
that led to significant downgrades of various European Union countries sovereign debt including Italy, Spain, Portugal and Greece; the 2012 one is related to the finalization of the second bailout package for Greece that provide a respite from
market turmoil; finally, the 2016 change point is associated with a crash in oil prices and concerns about the Chinese economy
that led to a sharp correction in the SP500 index of more than 5\% during January of that year. 

\begin{center}
\begin{table} [h!] 
\hspace{5 cm}\caption{Change points (CP) identified based on the method of this article and their $95\%$ adaptive confidence intervals (CI) using $5000$ replications. LCI: lower confidence interval, UCI: upper confidence interval. } \label{Table: DA1}
 \begin{tabular}{|c|c|c|c|c|c|c|c|c|}
\hline 
CP & 7/9/2002 & 2/25/2003 & 4/12/2005 & 9/25/2007 & 8/19/2008 & 2/17/2009  \\ 
\hline 
 $95\%$ LCI & 2/26/2002 & 12/3/2002  & 12/14/2004 & 6/12/2007 & 6/3/2008 & 12/30/2008  \\ 
\hline 
$95\%$ UCI & 11/5/2002  &  6/17/2003 & 10/18/2005  & 1/22/2008 & 11/4/2008 & 4/14/2009   \\ 
\hline 
\end{tabular}
\vskip 2pt
\hspace{4 cm}\begin{tabular}{|c|c|c|c|}
\hline 
CP & 9/20/2011 &  3/6/2012  & 1/26/2016 \\ 
\hline 
$95\%$ LCI & 3/29/2011  & 1/3/2012 & 9/22/2015 \\ 
\hline 
$95\%$ UCI & 12/20/2011 & 8/21/2012 & 6/21/2016 \\ 
\hline 
\end{tabular} 
\end{table}
\end{center}


Further, we used the methodology based on double CUSUM statistics introduced in \citet{cho2016change} and the resulting candidate change points are given in Table \ref{table: chocp}. Note that the there is a certain degree of concordance between the results in Table \ref{Table: DA1} and those in Table \ref{table: chocp}. For example, the 
2nd, 4th, 6th, 8th, 10th and 11th change points in Table \ref{table: chocp} are very close to first 6 change points in Table \ref{Table: DA1}. However, the latter method does not declare any change points after 2009, which is surprising given the turmoil in world financial markets induced by the sovereign debt crisis in Europe (2011-2012) and the significant market correction in January 2016 due to a crash in oil prices and concerns about China's economy that resulted in a 23\% decline in the major stock indices in China.
\begin{center}
\begin{table} [h!]
\caption{Change points detected by the method of double CUSUM statistics in \cite{cho2016change}.} \label{table: chocp}
\hspace{1.5cm}\begin{tabular}{|c|c|c|c|c|c|}
 \hline 
10/2/2001 & 10/2/2001 & 10/1/2002 & 3/4/2003 & 7/15/2003 & 10/11/2005 \\
\hline 
\end{tabular} 
\vskip 1pt
\hspace{1.4cm} \begin{tabular}{|c|c|c|c|c|c|}
 \hline 
 2/13/2007 & 7/10/2007 & 2/19/2008 & 9/23/2008 & 1/27/2009 & 8/18/2009 \\ 
\hline 
\end{tabular} 
\end{table}
\end{center}

\begin{center}
\begin{figure}[htp] 
\includegraphics[height=90mm, width=154mm]{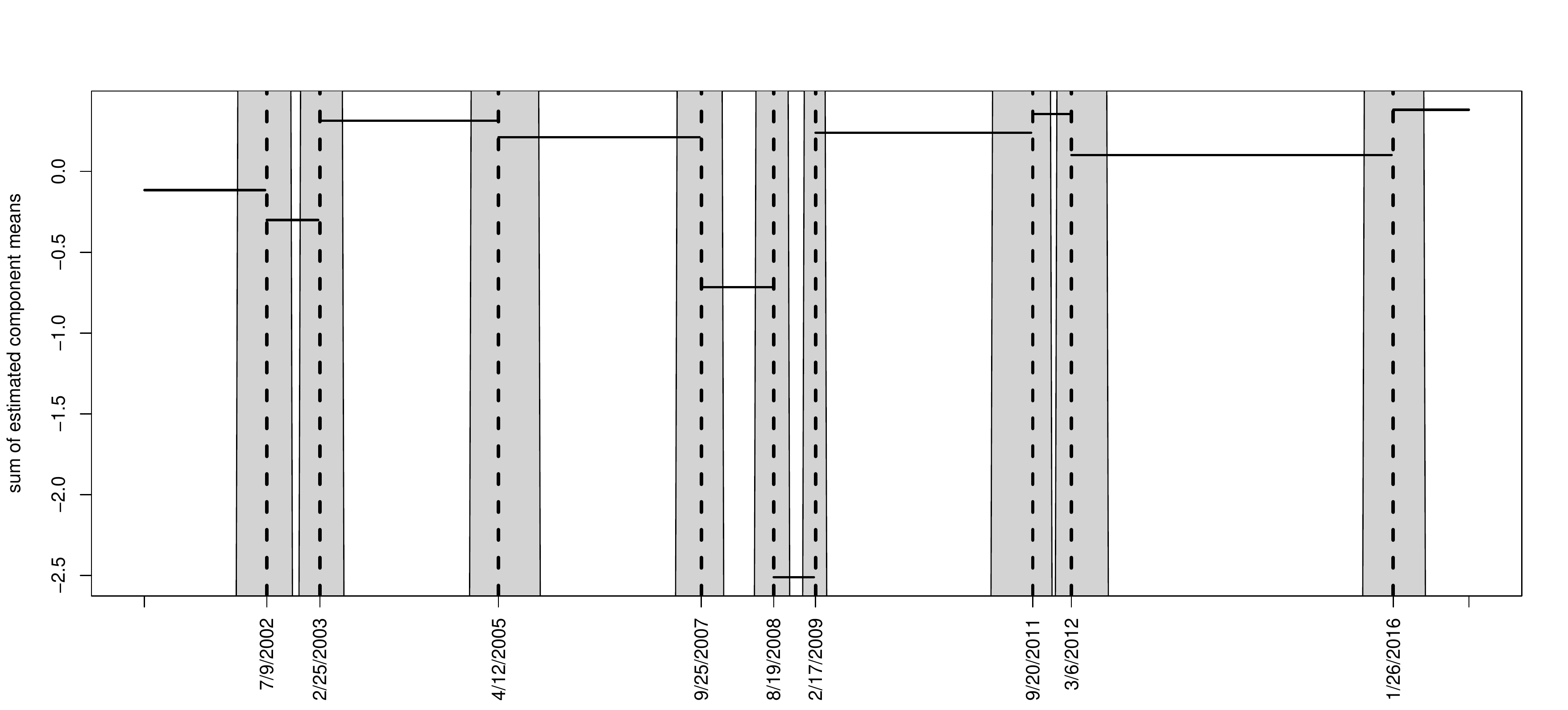} 
\caption{ \label{fig: 1} 
Change points are indicated by dashed vertical lines and their $95\%$ confidence intervals are presented by  grey vertical bands, mean between the change points are given by  horizontal lines.}
\end{figure} 
\end{center}

\bibliographystyle{abbrvnat}

\section{Appendix: Additional technical details} \label{sec: appA}

\subsection{Some consequences of Theorems \ref{thm: conrate} and \ref{thm: asympdist}} \label{subsec: cor}
In this section,  we provide some immediate consequences of Theorems \ref{thm: conrate} and \ref{thm: asympdist} in special cases such as independence across time and/or component-wise. 

\subsubsection{Independence across panels and dependence across time} \label{cor: panelind} 
Consider the model in (\ref{eqn:model}) with independence across $k$  and 
\begin{eqnarray}
A_{j,p} = \text{Diag}\{a_{j,k}:\ 1\leq k \leq p\}\ \ \text{and}\ \ \tilde{\gamma}_p := \sup_{1 \leq k \leq p} \sum_{j=0}^{\infty}  |a_{j,k}| = O(1).  \label{eqn: panelind}
\end{eqnarray}
Then $n  ||\mu_1  - \mu_2||_2^2 (\hat{\tau}_n -\tau_n) = O_{\text{P}}(1)$
holds under  \textbf{(SNR*)} $\frac{n}{p}  ||\mu_1-\mu_2||_2^2 \to \infty$.  Note that SNR* is weaker than  (SNR). \vskip 2pt
\noindent  \textbf{(a)} Moreover, if (SNR*) holds and $||\mu_1-\mu_2||_2 \to \infty$, then $P(\hat{\tau}_n =\tau_n) \to 1$. 
\vskip 2pt
\noindent Note that in Model (\ref{eqn: panelind}), population autocovariance of order $u$ is $$\Gamma_u = \sum_{j=0}^{\infty}A_jA_{j+u}  = \text{Diag}\left\{\sum_{j=0}^{\infty} a_{j,k}a_{j+u,k}:\ 1 \leq k \leq p \right\}.$$
\noindent Thus the quadratic form 
\begin{eqnarray}
(\mu_1-\mu_2)^\prime \Gamma_{t_1-t_2} (\mu_1-\mu_2) = \sum_{k=1}^{p} \bigg[ (\mu_{1k}-\mu_{2k})^2 \sum_{j=0}^{\infty} a_{j,k} a_{j+t_2-t_1,k}\bigg]. \nonumber 
\end{eqnarray}
Hence, $\sigma_{h_1,h_2}$ becomes
\begin{eqnarray}
\sigma_{h_1,h_2}^* &:=& \lim \sum_{\stackrel{t_i = \bigg[\frac{h_i}{||\mu_1-\mu_2||_2^2}\bigg]\wedge 0 +1}{i=1,2}}^{\bigg[\frac{ h_i}{||\mu_1-\mu_2||_2^2}\bigg]\vee 0} \sum_{k=1}^{p} \bigg[ (\mu_{1k}-\mu_{2k})^2 \sum_{j=0}^{\infty} a_{j,k} a_{j+t_2-t_1,k}\bigg].\nonumber 
\end{eqnarray}
In Model (\ref{eqn: panelind}), $((\sigma_{h_i,h_j}^*))_{1 \leq i,j \leq r}$ is asymptotic variance-covariance matrix of a finite dimensional distribution of the limiting Gaussian process when $||\mu_1-\mu_2||_2 \to 0$. Therefore the following statement is true
\vskip 2pt
\noindent \textbf{(b)} Suppose (SNR*) and (A3)  hold, $||\mu_1-\mu_2||_2 \to 0$, $((\sigma_{h_i,h_j}^*))_{1 \leq i,j \leq r}$ exists and is positive definite 
for all $h_i,h_j\in \mathbb{R}$ and $r \geq 1$.  Then the conclusion of Theorem \ref{thm: asympdist}(b) holds with $\sigma_{h_1,h_2}$ replaced by $\sigma_{h_1,h_2}^*$. 

\noindent Recall the partition of $\{1,2,\ldots, p(n)\}$ into $\mathcal{K}_0$ and $\mathcal{K}_n$ from (\ref{eqn: partition}). In Regime $||\mu_1-\mu_2||_2^2 \to c>0$,  as discussed after (\ref{eqn: partition}), we need to treat $\{X_{kt}: k \in \mathcal{K}_0, 1 \leq t \leq n\}$ and $\{X_{kt}: k \in \mathcal{K}_n, 1\leq t \leq n\}$ separately. Clearly, $\mathcal{K}_0$ is a finite set.  Moreover, in Model (\ref{eqn: panelind}), $X_{kt} = \mu_1 I(t \leq [n\tau_n]) + \mu_2 I(t > [n\tau_n]) + \varepsilon_{kt} = \mu_1 I(t \leq [n\tau_n]) + \mu_2 I(t > [n\tau_n]) +  \sum_{j=0}^{\infty} a_{j,k}\eta_{k(t-j)}\ \forall k,t$. As $\{\eta_{kt}\}$ are independent across both $k$ and $t$, $\{X_{kt}\}$ are also independent across $k \in \mathcal{K}_0$  and converge weakly if (A4), (A5) given before Theorem \ref{thm: asympdist} and (A8*) stated below hold. 
\vskip 5pt
\noindent \textbf{(A8*)} $(\varepsilon_{kt_1},\varepsilon_{kt_2},\ldots,\varepsilon_{kt_r}) \stackrel{\mathcal{D}}{\to} (\varepsilon_{kt_1}^{*},\varepsilon_{kt_2}^*,\ldots,\varepsilon_{kt_r}^*)\ \forall k \in \mathcal{K}_0$, $t_1,t_2,\ldots,t_r \in \mathbb{Z}$ and $r\geq 1$. 
$\mu_{ik} \to \mu_{ik}^{*}$ $\forall k \in \mathcal{K}_0$,  $i=1,2$.
\vskip 5pt
\noindent As we have independence across components, $\{\varepsilon_{kt}^*: k \in \mathcal{K}_0\}$ and $\{X_{kt}: k \in \mathcal{K}_n\}$ are independent.  In this case, $c_1$, $\tilde{\sigma}_{(k,0),(t_1,t_2)}$ and $\tilde{\sigma}_{(0,0),(t_1,t_2)}$ reduces to 
\begin{eqnarray}
c_1^* &:=& \lim \sum_{k \in \mathcal{K}_n} (\mu_{1k} - \mu_{2k})^2,\hspace{0.5 cm} \tilde{\sigma}_{(k,0),(t_1,t_2)}=0\ \forall k \in \mathcal{K}_0,\nonumber \\
\tilde{\sigma}_{(t_1,t_2)}^* &:=& \lim \sum_{k \in \mathcal{K}_n} (\mu_{1k} - \mu_{2k})^2 \sum_{j=0}^{\infty} a_{j,k}a_{j+t_2-t_1,k}\ \forall t_1,t_2\in \mathbb{Z}. \nonumber 
\end{eqnarray}
\noindent Limit of the random part involving $\{X_{kt}: k \in \mathcal{K}_n, t \geq 1\}$ is a Gaussian process on $\mathbb{Z}$ with covariances $\{\tilde{\sigma}_{(t_1,t_2)}^*: t_1,t_2\in \mathbb{Z}\}$. Existence of this limits and non-degeneracy of the Gaussian process are guaranteed by (A6*) and (A7*) stated below. 
\vskip 2pt
\noindent \textbf{(A6*)} $c_1^*$, $\tilde{\sigma}_{(t_1,t_2)}^*$ exists for all $t_1, t_2 \in \mathbb{Z}$. 
\vskip 5pt
\noindent \textbf{(A7*)}  $((\tilde{\sigma}_{(t_i,t_j)}^*))_{1\leq i,j \leq r}$ is positive definite  for all  $t_1,t_2,\ldots,t_r \in \mathbb{Z}$ and $r \geq 1$. 
\vskip 5pt
\noindent (A9*) is the analogue of (A9) and is required to establish asymptotic normality of the random part involving $\{X_{kt}: k \in \mathcal{K}_n, t \geq 1\}$. 
\vskip 5pt
\noindent \textbf{(A9*)}  
$\sup_{k \in \mathcal{K}_n} |\mu_{1k}-\mu_{2k}| \to 0$.
\vskip 5pt
\noindent Now the following statement is true. 
\vskip 10pt
\noindent  \textbf{(c)} Suppose (SNR*) and (A4), (A5), (A6*)-(A9*) hold and $||\mu_1-\mu_2||_2 \to c>0$, then
\begin{eqnarray}
n(\hat{\tau}_n-\tau_n) \stackrel{\mathcal{D}}{\to} \arg\max_{h \in \mathbb{Z}} (-0.5c_1^* |h| +  \sum_{t=0\wedge h}^{0 \vee h} (\tilde{W}_t^* + \tilde{A}_t^*)) \nonumber
\end{eqnarray}
where for each $t_1,t_2, \ldots,t_r \in \mathbb{Z}$  and $r \geq 1$,
\begin{eqnarray}
&&(\tilde{W}_{t_1}^{*},\tilde{W}_{t_2}^{*},\ldots, \tilde{W}_{t_r}^{*}) \sim \mathcal{N}_r (0, ((\tilde{\sigma}_{(t_i,t_j)}^*))_{1 \leq i,j \leq r} ),\ \ \
\nonumber \\
&& \tilde{A}_t^* = \frac{1}{2}\sum_{k \in \mathcal{K}_0} \bigg[ (\varepsilon_{kt}^* +b_{kt} -\mu_{2k}^*)^2-(\varepsilon_{kt}^*+b_{kt}-\mu_{1k}^*)^2 \bigg],  
\hspace{0.5 cm} b_{kt} = \begin{cases} \mu_{1k}^{*}\ \ \text{if $t \leq n\tau^{*}$} \nonumber \\
\mu_{2k}^{*}\ \ \ \text{if $t > n\tau^{*}$}.
\end{cases} 
\end{eqnarray}

\begin{remark}
Assumption (A8*) is weaker than (A8). If further $\frac{\sup_{k \in \mathcal{K}_0} \sum_{j=0}^{\infty}a_{j,k}}{\inf_{k \in \mathcal{K}_0} \sum_{j=0}^{\infty}a_{j,k}} = o(p^{1/6})$ holds in Model (\ref{eqn: panelind}),   then (A8) satisfies,  $\tilde{\sigma}_{(k_1,k_2),(t_1,t_2)} = (\sum_{j=0}^{\infty}a_{j,k}a_{j+t_i-t_j,k})I(k_1=k_2=k)$ and $\{\varepsilon_{kt}^{*}\}$ are Gaussian. This implies $\{\varepsilon_{kt}^*\}$ are independent across $k \in \mathcal{K}_0$ and  $(\varepsilon_{kt_1}^*,\varepsilon_{kt_2}^*,\ldots,\varepsilon_{kt_r}^*)  \sim \mathcal{N}_{r}(0, ((\sum_{j=0}^{\infty}a_{j,k}a_{j+t_i-t_j,k}))_{1\leq i,j \leq r})$ for all $t_1,t_2,\ldots,t_r \in \mathbb{Z}$, $r\geq 1$. 
\end{remark}

\subsubsection{Independence across panels and time} \label{cor: paneltimeind}
Consider the model in (\ref{eqn:model}) with independence across $k$  and 
\begin{eqnarray}
A_0 = \text{Diag}\{a_{0,k}:\ 1\leq k \leq p\},\ \ A_j =0\ \forall j \neq 0.   \nonumber
\end{eqnarray}
That is
\begin{eqnarray}
X_{kt} = \mu_1 I(t \leq [n\tau_n]) + \mu_2 I(t > [n\tau_n]) + a_{0,k}\eta_{kt}\ \ \forall k,t. \label{eqn: paneltimeindp}
\end{eqnarray}
Further suppose  there is $\epsilon, C>0$ such that $\epsilon < \inf_{k} |a_{0,k}| \leq \sup_{k}|a_{0,k}| < C$. Then $\gamma_p = O(1)$. 
\vskip 2pt
\noindent Recall SNR* in Section \ref{cor: panelind}.  Then $n  ||\mu_1  - \mu_2||_2^2 (\hat{\tau}_n -\tau_n) = O_{\text{P}}(1)$ holds under SNR*. 
\vskip 2pt
\noindent \textbf{(a)} Moreover, if SNR* holds and $||\mu_1-\mu_2||_2 \to \infty$, then $P(\hat{\tau}_n =\tau_n) \to 1$. 
\vskip 2pt
\noindent Now consider regime $||\mu_1-\mu_2||_2 \to 0$. In this case, population autocovariance $$\Gamma_0 = \text{Diag}\{a_{0,k}^2: 1\leq k \leq p\}\ \ \text{and}\ \ \ \Gamma_u =0\ \forall u \neq 0.$$
\noindent The quadratic form becomes
\begin{eqnarray}
(\mu_1-\mu_2)^\prime \Gamma_{t_1-t_2} (\mu_1-\mu_2) = \bigg[\sum_{k=1}^{p}  (\mu_{1k}-\mu_{2k})^2  a_{0,k}^2\bigg]I(t_1=t_2)\ \ \forall t_1,t_2. \nonumber 
\end{eqnarray}
\noindent Define
\begin{eqnarray}
\bar{\sigma}^2  &=& \lim \frac{ \sum_{k=1}^{p} (\mu_{1k}-\mu_{2k})^2 a_{0,k}^2 }{||\mu_1-\mu_2||_2^2}.\nonumber 
\end{eqnarray}
Thus $\sigma_{h_1,h_2} = \min(|h_1|,|h_2|) \bar{\sigma}^2$ for all $h_1,h_2 \in \mathbb{R}$ and the following statement is true. 
\vskip 2pt
\noindent \textbf{(b)} If (SNR*)  holds, $||\mu_1-\mu_2||_2 \to 0$, $\bar{\sigma}$ exists, then 
\begin{eqnarray}
n||\mu_1-\mu_2||_2^2 (\hat{\tau}_n-\tau_n) \stackrel{\mathcal{D}}{\to} \arg\max_{h \in \mathbb{R}} (-0.5|h| + \bar{\sigma} B_h) \nonumber 
\end{eqnarray}
where $B_h$ is the standard Brownian motion. 

\noindent Now consider the regime $||\mu_1-\mu_2||_2^2 \to c>0$.  Recall that $c_1^* = \lim \sum_{k \in \mathcal{K}_n} (\mu_{1k} - \mu_{2k})^2$ and Assumptions (A4), (A5) and (A9*). Note that Model (\ref{eqn: paneltimeindp}) is a special case of Model (\ref{eqn: panelind}).   Thus reason of considering these assumptions are same as discussed in Section \ref{cor: panelind}.  
As $\{X_{kt}\}$ are independent across both $k$ and $t$, (A8*) in Section \ref{cor: panelind} reduces to (A8**) stated below. 
\vskip 5pt
\noindent \textbf{(A8**)} $\eta_{kt}\stackrel{\mathcal{D}}{\to} \eta_{kt}^{*}$, $a_{0,k} \to a_{0,k}^{*}$ for all $k \in \mathcal{K}_0$ and $t \in \mathbb{Z}$. 
$\mu_{ik} \to \mu_{ik}^{*}$ $\forall k \in \mathcal{K}_0$,  $i=1,2$.
\vskip 5pt
\noindent Define, 
\begin{eqnarray}
\tilde{\sigma}^2 &=& \lim \sum_{k \in \mathcal{K}_n} (\mu_{1k} - \mu_{2k})^2 a_{0,k}^2. \nonumber 
\end{eqnarray}
Thus for Model (\ref{eqn: panelind}), $\tilde{\sigma}_{(0,0),(t_1,t_2)} = \tilde{\sigma}^2 I(t_1=t_2)$, $\tilde{\sigma}_{(k,0),(t_1,t_2)} = 0$ for all $k \in \mathcal{K}_0$ and $t_1,t_2\in \mathbb{Z}$. Hence (A7*) holds automatically and (A6*) reduces to (A6**) as stated below. 
\vskip2pt
\noindent \textbf{(A6**)} $c_1$ and $\tilde{\sigma}^2$ exists. 
\vskip 5pt
\noindent Thus the following statement is true.
\vskip 10pt
\noindent  \textbf{(c)} Suppose (SNR*) and (A4), (A5), (A6**), (A8**) and (A9*) hold and $||\mu_1-\mu_2||_2 \to c>0$, then
\begin{eqnarray}
n(\hat{\tau}_n-\tau_n) \stackrel{\mathcal{D}}{\to} \arg\max_{h \in \mathbb{Z}} (-0.5c_1^{*} |h| +  \sum_{t=0\wedge h}^{0 \vee h} ({W}_t + {A}_t)) \nonumber
\end{eqnarray}
where $W_t \stackrel{\text{i.i.d.}}{\sim} \mathcal{N}(0,1)$ and 
\begin{eqnarray}
&& {A}_t = \sum_{k \in \mathcal{K}_0} \bigg[ (\varepsilon_{kt} +b_{kt} -\mu_{2k}^*)^2-(\varepsilon_{kt}+b_{kt}-\mu_{1k}^*)^2 \bigg],  
\hspace{0.5 cm} b_{kt} = \begin{cases} \mu_{1k}^{*}\ \ \text{if $t \leq n\tau^{*}$} \nonumber \\
\mu_{2k}^{*}\ \ \ \text{if $t > n\tau^{*}$}.
\end{cases} 
\end{eqnarray}

\subsection{Other examples} \label{subsec: otherexamples}
\begin{Example} \label{example: DD}
\noindent \textbf{Diagonally dominant coefficient matrices}. A symmetric matrix $M = ((m_{ij}))$ is said to be diagonally dominant if $\sum_{i: i \neq j} |m_{ij}| < |m_{jj}|\ \forall j \geq 1$.  If the coefficient matrices are sparse, diagonally dominance can often be a reasonable assumption. In this case, for all $j \geq 0$,  the covariances between $\varepsilon_{kt}$ and $\eta_{k^\prime t+j}$ are much smaller for $k \neq k^\prime$ compared to $k = k^\prime$. 
Suppose $\{A_{j,p}\}$ are all symmetric and diagonally dominant matrices. We know that $||A_{j,p}||_2 \leq ||A_{j,p}||_{(1,1)}$. Now $||A_{j,p}||_{(1,1)} = \max_{1 \leq i \leq p} \sum_{l=1}^{p}|A_{j,p}(i,l)| \leq \max_{1\leq i\leq p} 2|A_j(i,i)| \leq  2||A_j||_{2}$. Thus, $\gamma_p$ and $\beta_p$ are of the same order. Suppose (a)-(d) in Proposition \ref{prop: 1} holds.  For such choices of coefficient matrices, (A1) and (A6) are satisfied  if (k) and (l), stated after Propositions \ref{prop: 1} and \ref{prop: 2}, are satisfied respectively.
For this example, no other simplification is possible for (SNR),  (A2),  (A7) and for the first condition of (A8).  
\end{Example}

\begin{Example} \label{example: VAR}
\textbf{VAR process}:  A Vector Autoregressive process of order $r$ (VAR$(r)$) is given by
\begin{eqnarray} \label{eqn: AR}
\varepsilon_{t,p(n)}^{(n)} + \sum_{j=1}^{r} \Psi_{j,p(n)} \varepsilon_{t-j,p(n)}^{(n)}  = \eta_{t,p(n)}. \nonumber 
\end{eqnarray}
where the $p \times p$ nested matrices $\{\Psi_{i,p(n)}: 0 \leq i \leq r\}$ are the model parameter matrices and $\{\eta_t\}$ are as described after (\ref{eqn:model}). Moreover, if
for some $\epsilon >0$, 
\begin{eqnarray} \label{eqn: causal}
I_{p(n)} + \sum_{j=1}^{r} \Psi_{j,p(n)} z^{j} \neq 0\ \forall\ z \in \mathbb{C}\ \text{such that}\ |z| < 1+\epsilon, 
\end{eqnarray}
then (\ref{eqn: AR}) can be represented as (\ref{eqn:model}) with $A_{0,p(n)} = I_{p(n)}$, $A_{j,p(n)} = \sum_{i=1}^{j} \Psi_{i,p(n)} A_{j-i,p(n)}$.  Further, suppose $\sup_{p} ||\Psi_{j,p}||_{(1,1)}  <\infty$. Then, one can easily show that $\sup_{p} ||A_j||_{(1,1)} \leq C\theta^j$ for some $C>0$ and $0 <\theta <1$. This implies $\gamma_p, \beta_p = O(1)$. Also, as $A_{0,p(n)} = I_{p(n)}$, $\gamma_p$ is bounded away from $0$. Hence, for Theorems \ref{thm: conrate} and \ref{thm: asympdist}, we need (SNR*) $\frac{n}{p} ||\mu_1-\mu_2||_2^2 \to \infty$, instead of (SNR). Also, (A1) and (A6) hold if (k) and (l1), stated after  Propositions \ref{prop: 1} and \ref{prop: 2},  are satisfied respectively. Again, suppose the smallest eigenvalues of $\Psi_{j,p}$, $1 \leq j \leq r$,   are bounded away from $0$ with respect to both $j$ and $p$. This implies that the smallest eigenvalue of $\Gamma_0$ is bounded away from $0$ and (A2), (A7) and the first condition of (A8) hold. 
\end{Example}

\begin{Example} \label{example: dominant}
\noindent \textbf{Coefficient matrices dominated by separable cross-sectional and time dependence structure}.  One may think that the structures in Examples \ref{example: 1}  and \ref{example: polydep}  are restrictive. In this example, we consider a significantly wider class of  coefficient matrices  which are dominated by separable cross-sectional and time dependence structure. In other words, $\sup_{j \geq 0} |A_j(k,l)|  \leq a_j b(k,l)\ \forall k,l$ where $A_j(k,l)$ be the $(k,l)$-th element of $A_j$. Define a sequence of nested matrices $\{B_p\}$ such that $B_p = ((b(k,l)))_{1\leq k,l \leq p}$. Then (a) and (d) in Example \ref{example: 1} hold. Also suppose that $\inf_{p}||A_j||_2 >0$ for at least one $j \geq 0$. This implies that $\gamma_p$ is bounded away from $0$. Therefore, in Propositions \ref{prop: 1} and \ref{prop: 2}, we  need to replace $\gamma_p$ and  $\beta_p$ by $1$ and $||B_p||_{(1,1)}$, respectively. Consequently,  (A1) and (A6) hold if (b)-(j) in  Propositions \ref{prop: 1} and \ref{prop: 2} are satisfied after replacing $\gamma_p$ and $\beta_p$ by $1$ and $||B_p||_{(1,1)}$, respectively.  Finally,  (A2) and (A7) are satisfied if the smallest singular value of $B_p$ is bounded away from $0$ and $a_j>0$ for atleast one $j \geq 0$.
\end{Example}



\pagebreak

\section{Supplementary material: Proofs} \label{sec: proof}

\subsection{Useful lemmas}
Following two lemmas  quoted from \cite{Wellner1996empirical} are needed to prove Theorems \ref{thm: conrate} and \ref{thm: asympdist}.
\begin{lemma} \label{lem: wvan1}
For each $n$, let $\mathbb{M}_n$ and $\tilde{\mathbb{M}}_n$ be stochastic processes indexed by a set $\mathcal{T}$. Let $\tau_n\ \text{(possibly random)} \in \mathcal{T}_n \subset \mathcal{T}$ and 
$d_n(b,\tau_n)$ be a map (possibly random) from $\mathcal{T}$ to $[0,\infty)$. Suppose that for every large $n$ and $\delta \in (0,\infty)$
\begin{eqnarray}
&& \sup_{\delta/2 < d_n(b,\tau_n) < \delta,\  b \in \mathcal{T}} (\tilde{\mathbb{M}}_n(b) - \tilde{\mathbb{M}}_n(\tau_n)) \leq -C\delta^2, \label{eqn: lemcon1} \\
&& E\sup_{\delta/2 < d_n(b,\tau_n) < \delta,\  b \in \mathcal{T}}  \sqrt{n} |\mathbb{M}_n(b) - \mathbb{M}_n(\tau_n) - (\tilde{\mathbb{M}}_n(b) - \tilde{\mathbb{M}}_n(\tau_n))| \leq  C\phi_{n}(\delta), \label{eqn: lemcon2}
\end{eqnarray}
for some $C>0$ and for function $\phi_n$ such that $\delta^{-\alpha}\phi_n(\delta)$ is decreasing in $\delta$ on $(0,\infty)$ for some $\alpha <2$. Let $r_n$ satisfy
\begin{eqnarray} \label{eqn: lemrn}
r_n^2  \phi(r_n^{-1}) \leq \sqrt{n}\ \ \text{for every  $n$}.
\end{eqnarray}
Further, suppose that the sequence $\{\hat{\tau}_n\}$ takes its values in $\mathcal{T}_n$ and 
satisfies $\mathbb{M}_n(\hat{\tau}_n) \geq \mathbb{M}_n(\tau_n) - O_P (r_n^{-2})$ for large enough $n$. Then,
$r_n d_{n}(\hat{\tau}_n,\tau_n) = O_P (1)$.
\end{lemma}

\begin{lemma}  \label{lem: wvandis1}
Let $\mathbb{M}_n$ and $\mathbb{M}$ be two stochastic processes indexed by a metric space $\mathcal{T}$,
such that $\mathbb{M}_n \Rightarrow \mathbb{M}$ in $l^{\infty}(\mathcal{C})$ for every compact set $\mathcal{C} \subset \mathcal{T}$ i.e.,
\begin{align}
\sup_{h \in \mathcal{C}} |\mathbb{M}_n(h) - \mathbb{M}(h)| \stackrel{P}{\to} 0.
\end{align} 
Suppose that almost all sample paths $h \to \mathbb{M}(h)$ are upper semi-continuous and possess a unique maximum at a (random) point $\hat{h}$, which as a random map in $\mathcal{T}$ is tight. If the sequence $\hat{h}_n$ is uniformly tight and satisfies $\mathbb{M}_n(\hat{h}_n) \geq \sup_{n} \mathbb{M}_n(h) - o_{P}(1)$, then $\hat{h}_n \stackrel{\mathcal{D}}{\to} \hat{h}$ in $\mathcal{T}$.
\end{lemma}

Following lemma is useful to proof Theorem \ref{thm: adap}.  Define,
\begin{eqnarray}
\hat{\Lambda}_{h_1,h_2} &=& \sum_{\stackrel{t_i = \bigg[\frac{\gamma_p^2 h_i}{||\hat{\mu}_1-\hat{\mu}_2||_2^2}\bigg]\wedge 0 +1}{i=1,2}}^{\bigg[\frac{\gamma_p^2 h_i}{||\hat{\mu}_1-\hat{\mu}_2||_2^2}\bigg]\vee 0} B_{l_n}(\hat{\Gamma}_{t_2-t_1}),\ \ \  
\hat{\sigma}_{h_1,h_2} = \lim \gamma_p^{-4} (\hat{\mu}_1-\hat{\mu}_2)^\prime \hat{\Lambda}_{h_1,h_2} (\hat{\mu}_1-\hat{\mu}_2), \nonumber  \\
\hat{\tilde{\sigma}}_{(0,0),(t_1,t_2)} &=& \lim \gamma_p^{-4} \sum_{k_1,k_2 \in \mathcal{K}_n} (\hat{\mu}_{1k_1} - \hat{\mu}_{2k_1}) (\hat{\mu}_{1k_2}-\hat{\mu}_{2k_2}) B_{l_n}(\hat{\Gamma}_{t_2-t_1})(k_1,k_2), \nonumber \\
\hat{\tilde{\sigma}}_{(k,0),(t_1,t_2)} &=&  \lim \gamma_p^{-3} \sum_{k_1 \in \mathcal{K}_n} (\hat{\mu}_{1k_1} - \hat{\mu}_{2k_1}) B_{l_n}(\hat{\Gamma}_{t_2-t_1})(k,k_1), \nonumber \\
\hat{\tilde{\sigma}}_{(k_1,k_2),(t_1,t_2)} &=&  \lim \gamma_p^{-2}B_{l_n}(\hat{\Gamma}_{t_2-t_1})(k_1,k_2). \nonumber
\end{eqnarray}

\begin{lemma} \label{lem: adap}
Suppose SNR-ADAP, (C1), (C2) and (C3) hold. Then  the following statements  are true. 
\vskip 2pt
\noindent (a) $E(e^{\lambda\gamma_{p}^{-1}|\frac{1}{nb}\sum_{t=1}^{nb}\varepsilon_{kt,p}|}), E(e^{\lambda \gamma_{p}^{-1}|\frac{1}{n(1-b)}\sum_{t=nb+1}^{n}\varepsilon_{kt,p}|}) \leq C_1e^{C_2 \lambda^2}$  for all $\lambda \in \mathbb{R}$, $b\in (c^*,1-c^*)$ and for some $C_1,C_2>0$. 
\vskip 2pt
\noindent (b) $\gamma_p^{-1} ||\hat{\mu}_1 - \mu_1||_2, \gamma_p^{-1} ||\hat{\mu}_2 - \mu_2||_2 = O_{\text{P}}\left ( \sqrt{\frac{p\log p}{n}} \right)$ \\
\noindent (c) $\bigg|\frac{||\hat{\mu}_1-\hat{\mu}_2||_2^2}{||\mu_1-\mu_2||_2^2} -1\bigg| = O_{\text{P}}\left(\frac{p\log p}{n \gamma_p^{-2}||\mu_1-\mu_2||_2^2}\right)$ \\
\noindent (d) $|\sup_{k} |\hat{\mu}_{1k} - \hat{\mu}_{2k}| - \sup_{k} |{\mu}_{1k} - {\mu}_{2k}|| = O_{\text{P}}\left( \sqrt{\frac{\log p}{n}}\right)$ \\
\noindent (e) $|\sup_{k \in \mathcal{K}_n} |\hat{\mu}_{1k}-\hat{\mu}_{2k}| - \sup_{k \in \mathcal{K}_n} |\mu_{1k}-\mu_{2k}| | = O_{\text{P}}\left( \sqrt{\frac{\log p}{n}}\right)$ \\
\noindent (f) $\bigg|\frac{\sum_{k \in \mathcal{K}_n}(\hat{\mu}_{1k}-\hat{\mu}_{2k})^2}{\sum_{k \in \mathcal{K}_n} (\mu_{1k}-\mu_{2k})^2} -1\bigg| = O_{\text{P}}\left(\frac{p\log p}{n \gamma_p^{-2}||\mu_1-\mu_2||_2^2}\right)$ when $\gamma_p^{-2}\sum_{k \in \mathcal{K}_n} (\mu_{1k}-\mu_{2k})^2 \to c_1$. \\
\noindent (g) $|\hat{\sigma}_{h_1,h_2} - \sigma_{h_1,h_2}| = O_{\text{P}} \left(\sqrt{\frac{p\log p}{n}}||\mu_1-\mu_2||_2\right)$ for all $h_1,h_2 \in \mathbb{R}$ and $\gamma_p^{-2}||\mu_1-\mu_2||_2^2 \to 0$\\
\noindent (h) $|\hat{\tilde{\sigma}}_{(k_1,k_2),(t_1,t_2)} - \tilde{\sigma}_{(k_1,k_2),(t_1,t_2)} | = O_{\text{P}} \left(\sqrt{\frac{p\log p}{n}}||\mu_1-\mu_2||_2\right)$ for all $k_1,k_2 \in \mathcal{K}_0 \cup \{0\}$,  $t_1,t_2 \in \mathbb{Z}$  and $\gamma_p^{-2}||\mu_1-\mu_2||_2^2 \to c>0$. 
\end{lemma}

\begin{proof}
(a) It is easy to see that $$\frac{1}{nb} \sum_{t=1}^{nb} \varepsilon_{kt,p} = \frac{1}{nb} \sum_{t=1}^{nb} \sum_{j=0}^{\infty} \sum_{i=1}^{p} A_{j}(k,i)\eta_{i,t-j} = \sum_{i=1}^{p} \sum_{t=-\infty}^{nb} d_{t,i} \eta_{i,t}$$ where $\sum_{i=1}^{p}\sum_{t=-\infty}^{nb} |d_{t,i}|^2 \leq  \sum_{i=1}^{p}\sum_{j=0}^{\infty} |A_{j}(k,i)|^2  \leq \gamma_{p}^2 r_{k}$.

\noindent Hence, by independence of $\{\eta_{kt,p}\}$ and (C1),  (C2), for all $\lambda \in \mathbb{R}$, 
$$E(e^{\lambda\gamma_{p}^{-1}|\frac{1}{nb}\sum_{t=1}^{nb}\varepsilon_{kt,p}|}) \leq C_1e^{C_2 \lambda^2 \gamma_{p}^{-2}\sum_{i=1}^{p}\sum_{t=-\infty}^{nb} |d_{t,i}|^2} \leq C_1e^{C_2 \lambda^2 r_k} <\infty.$$
Similar arguments hold for $\frac{1}{n(1-b)}\sum_{t=nb+1}^{n}\varepsilon_{kt,p}$.  This completes the proof of Lemma \ref{lem: adap}(a). 
\vskip 5pt
\noindent (b) With out loss of generality, assume $\hat{\tau}_n > \tau_n$. 
Note that
\begin{eqnarray}
\hat{\mu}_1 - \mu_1 &=& \frac{\tau_n}{\hat{\tau}_n} \left(\frac{1}{n\tau_n}\sum_{t=1}^{n\tau_n} X_{t,p} - \mu_1 \right) + \frac{\hat{\tau}_n-\tau_n}{\hat{\tau}_n} \left(\frac{1}{n(\hat{\tau}_n-\tau_n)}\sum_{t=n{\tau}_n+1}^{n\hat{\tau}_n} X_{t,p} - \mu_2 \right) + \frac{\hat{\tau}_n-\tau_n}{\hat{\tau}_n} (\mu_2-\mu_1)  \nonumber  \\
&=& A_1+A_2 +A_3,\ \ \ \text{say}. \nonumber 
\end{eqnarray}
Now By Theorem \ref{thm: conrate}, Assumption (C1)-(C3) and Theorem $1$ of \cite{hsu2012tail}, we have
\begin{eqnarray}
\gamma_p^{-1} ||A_1||_2, \gamma_p^{-1} ||A_1||_2 = O_{\text{P}}\left ( \sqrt{\frac{p\log p}{n}} \right), \gamma_p^{-1} ||A_3||_2 = o(1). \nonumber 
\end{eqnarray}
This completes the proof of Lemma \ref{lem: adap}(b) for $\mu_1$. Similar arguments also work for $\mu_2$.  
\vskip 5pt
\noindent Proof of Lemma \ref{lem: adap}(c)-(f) is similar to the proof of Lemma \ref{lem: adap}(b). Hence we omit it. 
\vskip 3pt
\noindent (g) Note that
\begin{eqnarray}
&& |\hat{\sigma}_{h_1,h_2} - \sigma_{h_1,h_2}| \nonumber \\
&\leq & C\bigg[||\hat{\Lambda}_{h_1,h_2}-\Lambda_{h_1,h_2}||_2||\hat{\mu}_1-\hat{\mu}_2||_2^2 + ||\hat{\mu}_1-\hat{\mu}_2-\mu_1-\mu_2||_2||\Lambda_{h_1,h_2}||_2||\hat{\mu}_1-\hat{\mu}_2||_2\bigg] \nonumber \\
& = & O_{\text{P}}\left[\left(\frac{\log p}{n} \right)^{\frac{\alpha}{2+\alpha}} ||\mu_1-\mu_2||_2^{2}\right] O_{\text{P}} \left(\frac{p\log p}{n}||\mu_1-\mu_2||_2^{2} \right) + O_{\text{P}} \left(\sqrt{\frac{p\log p}{n}}||\mu_1-\mu_2||_2\right) \nonumber \\
& = & O_{\text{P}} \left(\sqrt{\frac{p\log p}{n}}||\mu_1-\mu_2||_2\right). \nonumber
\end{eqnarray}
This completes the proof of Lemma \ref{lem: adap}(g).
\vskip 3pt
\noindent Proof of Lemma \ref{lem: adap}(h) is similar to the proof of Lemma \ref{lem: adap}(g). 
\vskip 3pt
\noindent This completes the proof of Lemma \ref{lem: adap}.
\end{proof}

\subsection{Proof of Theorem \ref{thm: conrate}} \label{subsec: conrate}
Recall that
\begin{eqnarray}
\hat{\tau}_n &=& \arg\min_{b \in (c^*, 1-c^*)} L(b)\ \ \ \text{where} \label{eqn: cpdefine1} \\
L(b) &=& \frac{1}{n}\sum_{k=1}^{b}\bigg[\sum_{t=1}^{nb}(X_{kt}-\hat{\mu}_{1k}(b))^2  + \sum_{t=nb+1}^{n}(X_{kt}-\hat{\mu}_{2k}(b))^2\bigg],  \nonumber \\
\hat{\mu}_{1k}(b) &=& \frac{1}{nb}\sum_{t=1}^{nb} X_{kt},\ \ \hat{\mu}_{2k}(b) = \frac{1}{n(1-b)} \sum_{t=nb+1}^{n} X_{kt}. \nonumber
\end{eqnarray}
Here we prove $n\gamma_p^{-2}||\mu_1-\mu_2||_2^2(\hat{\tau}_n-\tau_n) = O_{\text{P}}(1)$. 
\vskip 5pt
\noindent To prove Theorem \ref{thm: conrate}, we need  Lemma \ref{lem: wvan1}  quoted from \cite{Wellner1996empirical}.  For our purpose, we make use of the above lemma with $\mathbb{M}_n (\cdot) = {L} (\cdot)$, $\tilde{\mathbb{M}}_n(\cdot) = E {L} (\cdot)$,
$\mathcal{T} = [0,1]$, $\mathcal{T}_n = \{1/n,2/n,\ldots, (n-1)/n,1\} \cap [c^{*},1-c^{*}]$,  $d_n(b,\tau_n) = ||\mu_1-\mu_2||_2 \sqrt{|b - \tau_n|}$, $\phi_n(\delta) = \delta \gamma_p$, $\alpha = 1.5$, $r_n = \sqrt{n}\gamma_p^{-1}$.  Thus, to prove Theorem \ref{thm: conrate}, it is enough to establish that for some $C>0$,
\begin{eqnarray}
&&  E(\mathbb{M}_n (b) - \mathbb{M}_n (\tau_n)) \leq -C||\mu_1-\mu_2||_2^2 |b - \tau_n|\ \ \text{and} \label{eqn: lsecon1} \\
&& E\sup_{\delta/2 < d_n(b,\tau_n) < \delta,\  b \in \mathcal{T}}   |\mathbb{M}_n (b) - \mathbb{M}_n (\tau_n) - E(\mathbb{M}_n (b) - \mathbb{M}_n (\tau_n))| \leq  C\frac{\delta \gamma_p}{\sqrt{n}}. 
\label{eqn: lsecon2f}
\end{eqnarray}
Note that the left hand side of (\ref{eqn: lsecon2f}) is dominated by
\begin{eqnarray}
\left(E^{*}\sup_{\delta/2 < d_n(b,\tau_n) < \delta,\  b \in \mathcal{T}}   (\mathbb{M}_n (b) - \mathbb{M}_n (\tau_n) - E(\mathbb{M}_n (b) - \mathbb{M}_n (\tau_n)))^2\right)^{1/2}. \label{eqn: doobprelse}
\end{eqnarray}
By Doob's martingale inequality, (\ref{eqn: doobprelse}) is further dominated by
\begin{eqnarray}
(\text{V}^{*}(\mathbb{M}_n (b) - \mathbb{M}_n (\tau_n)))^{1/2}\ \ \text{where $d_n(b,\tau_n) = \delta$}.
\end{eqnarray}
Thus, to prove Theorem \ref{thm: conrate}, it is enough to show that for some $C>0$,
\begin{eqnarray}
\text{V}^{*}(\mathbb{M}_n (b) - \mathbb{M}_n (\tau_n)) \leq Cn^{-1} d_n^2(b,\tau_n) \gamma_p^2. \label{eqn: lsecon2final}
\end{eqnarray}
Hence, it is enough to prove (\ref{eqn: lsecon1}) and (\ref{eqn: lsecon2final}) to establish Theorem \ref{thm: conrate}. We shall prove these for $b<\tau_n$. Similar arguments work when $b \geq \tau_n$.

Let $L(b) = \frac{1}{n} \sum_{k=1}^{p} L_k(b)$, say. We write $\tau$ for $\tau_n$. 
\begin{eqnarray}
L_{k}(b) - L_{k}(\tau) &=& A_k(b) + B_k(b) + D_k(b)\ \ \text{where} \label{eqn: abd1} \\
A_k(b) &=& \sum_{t=1}^{nb} \bigg[ (X_{kt}-\hat{\mu}_{1k}(b))^2 - (X_{kt}-\hat{\mu}_{1k}(\tau))^2\bigg], \nonumber \\
B_k(b) &=& \sum_{t=n\tau +1}^{n} \bigg[(X_{kt}-\hat{\mu}_{2k}(b))^2 - (X_{kt}-\hat{\mu}_{2k}(\tau))^2 \bigg], \nonumber \\
D_k(b) &=& \sum_{t=nb+1}^{n\tau} \bigg[(X_{kt}-\hat{\mu}_{2k}(b))^2 - (X_{kt}-\hat{\mu}_{1k}(\tau))^2\bigg]. \nonumber 
\end{eqnarray}
\noindent First we calculate expectation and variance of $A_k$. Let $\tilde{X}_{kt} = X_{kt}-EX_{kt}$.  Now, 
\begin{eqnarray}
A_{k}(b) &=& \sum_{t=1}^{nb} \bigg[(X_{kt}-\hat{\mu}_{1k}(b))^2 - (X_{kt} -\hat{\mu}_{1k}(\tau))^2 \bigg] \nonumber \\
&=& \sum_{t=1}^{nb} \bigg[ \hat{\mu}_{1k}^2 (b) - \hat{\mu}_{1k}^{2}(\tau) - 2X_{kt}\hat{\mu}_{1k}(b) + 2X_{kt}\hat{\mu}_{1k}(\tau) \bigg] \nonumber \\
&=& nb \bigg[ \hat{\mu}_{1k}^2(b) - \hat{\mu}_{1k}^2 (\tau) - 2\hat{\mu}_{1k}^2 (b) + 2\hat{\mu}_{1k}(b)\hat{\mu}_{1k}(\tau) \bigg] \nonumber \\
&=& -nb (\hat{\mu}_{1k}(b) - \hat{\mu}_{1k}(\tau))^2 \nonumber \\
&=& -nb\bigg[ \frac{1}{n}\left(\frac{1}{b}-\frac{1}{\tau}\right) \sum_{t=1}^{nb} X_{kt}-\frac{1}{n\tau}\sum_{t=nb+1}^{n\tau} X_{kt} \bigg]^2 \nonumber \\
&=& -nb \bigg[ \left(\frac{1}{n}\frac{\tau-b}{\tau b} \right)^2 \left(\sum_{t=1}^{nb} \tilde{X}_{kt} \right)^2  + \left( \frac{1}{n\tau}\right)^2 \left( \sum_{t=nb+1}^{n\tau} \tilde{X}_{kt}\right)^2 \nonumber\\
&& \hspace{4 cm} - \frac{2}{n} \frac{\tau-b}{\tau b} \frac{1}{n\tau} \left( \sum_{t=1}^{nb} \tilde{X}_{kt}\right) \left( \sum_{t=nb+1}^{n\tau} \tilde{X}_{kt}\right) \bigg] \nonumber \\
&=& A_{1k}(b) + A_{2k}(b) + A_{3k}(b),\ \ \text{say}. 
\end{eqnarray}
\noindent Note that
\begin{eqnarray}
&& E\left( \frac{1}{n} \sum_{k=1}^{p} A_{1k}(b)\right) = -b \left(\frac{1}{n} \frac{\tau-b}{\tau b}\right)^2 \sum_{k=1}^{p} E \left(\sum_{t=1}^{nb}\tilde{X}_{kt} \right)^2 \nonumber \\
&=& -b \left(\frac{1}{n} \frac{\tau-b}{\tau b} \right)^2 \sum_{k=1}^{p} \sum_{t_1=1}^{nb} \sum_{t_2=1}^{nb} \text{Cov}(\tilde{X}_{kt_1},\tilde{X}_{kt_2}) \nonumber \\
&=& - b \left(\frac{1}{n} \frac{\tau-b}{\tau b} \right)^2 \sum_{t_1=1}^{nb} \sum_{t_2=1}^{nb} \sum_{k=1}^{p} \Gamma_{t_2-t_1}(k,k)\nonumber \\
& \geq & -b \frac{1}{nb} \frac{(\tau-b)^2}{\tau^2} \sum_{u=-\infty}^{\infty} |\text{Tr}(\Gamma_{|u|})| \geq -C \frac{(\tau-b)p}{n} \sum_{u=-\infty}^{\infty} |\frac{1}{p} \text{Tr}(\Gamma_{|u|})| \nonumber \\
& \geq & -C(\tau-b)\frac{p}{n} \sum_{u=-\infty}^{\infty} \sum_{j=0}^{\infty} ||A_j||_2 ||A_{j+u}||_2 \geq -C(\tau-b)\frac{p}{n}\gamma_p^2. \nonumber 
\end{eqnarray}
\noindent Next,
\begin{eqnarray}
&& E\left( \frac{1}{n}\sum_{k=1}^{p} A_{2k}(b) \right) = \frac{-nb}{n} \left(\frac{1}{n\tau} \right)^2 \sum_{k=1}^{p} E \left( \sum_{t=nb+1}^{n\tau} \tilde{X}_{kt} \right)^2 \nonumber \\
&=& -b \left( \frac{1}{n\tau}\right)^2 \sum_{t_1 =  nb+1}^{n\tau} \sum_{t_2=nb+1}^{n\tau} \text{Tr}(\Gamma_{t_2-t_1}) \geq -C\frac{(\tau-b)}{n} \sum_{u=-\infty}^{\infty} |\text{Tr}(\Gamma_{|u|})| \nonumber \\
& \geq & -C\frac{(\tau-b)p}{n} \sum_{u=-\infty}^{\infty} |\frac{1}{p} \text{Tr}(\Gamma_{|u|})| \geq -C \frac{(\tau-b)p}{n} \gamma_p^2.  \nonumber
\end{eqnarray}
\noindent Also, 
\begin{eqnarray}
&& E\left(\frac{1}{n}\sum_{k=1}^{p} A_{3k} \right) = \frac{2b}{\tau} \frac{1}{n^2} \frac{\tau-b}{\tau b} \sum_{t_1 = 1}^{nb} \sum_{t_2=nb+1}^{n\tau} \sum_{k=1}^{p} \text{Cov}(\tilde{X}_{kt_1},\tilde{X}_{kt_2}) \nonumber \\
&=& \frac{2b}{\tau} \frac{1}{n^2} \frac{\tau-b}{\tau b} \sum_{t_1 = 1}^{nb} \sum_{t_2=nb+1}^{n\tau} \text{Tr}(\Gamma_{t_2-t_1})  \geq - C \frac{(\tau-b)p}{n} \sum_{u=-\infty}^{\infty} |\frac{1}{p}\text{Tr}(\Gamma_{|u|})| \geq -C \frac{(\tau-b)p}{n} \gamma_p^2. \nonumber
\end{eqnarray}
Thus 
\begin{eqnarray} \label{eqn: expA}
E\left( \frac{1}{n}\sum_{k=1}^{p}A_k(b)\right) \geq -C \frac{(\tau-b)p}{n} \gamma_p^2. 
\end{eqnarray}
\noindent Now we compute variance of $A_k$. 
\begin{eqnarray}
&& \text{V} \left( \frac{1}{n}\sum_{k=1}^{p} A_{1k}(b) \right) = \text{V} \bigg[ \frac{nb}{n} \left( \frac{1}{n} \frac{\tau-b}{\tau b}\right)^2 \sum_{k=1}^{p} \left(\sum_{t=1}^{nb} \tilde{X}_{kt} \right)^2\bigg] \nonumber \\
& \leq & \frac{C(\tau-b)}{n^4}\sum_{k_1,k_2=1}^{p} \sum_{t_1,t_2=1}^{nb} \sum_{t_3,t_4=1}^{nb} \bigg[ E(\tilde{X}_{k_1t_1}\tilde{X}_{k_1t_2}\tilde{X}_{k_2t_3}\tilde{X}_{k_2t_4}) - \Gamma_{t_2-t_1}(k_1,k_1)\Gamma_{t_4-t_3}(k_2,k_2) \bigg] \nonumber \\
& \leq & \frac{C(\tau-b)}{n^4}\sum_{k_1,k_2=1}^{p} \sum_{t_1,t_2=1}^{nb} \sum_{t_3,t_4=1}^{nb} \bigg[ \text{Cum}(\tilde{X}_{k_1t_1}, \tilde{X}_{k_1t_2}, \tilde{X}_{k_2t_3}, \tilde{X}_{k_2t_4}) \nonumber \\
&& \hspace{4 cm} + \Gamma_{t_3-t_1}(k_1,k_2)\Gamma_{t_4-t_2}(k_1,k_2) + \Gamma_{t_4-t_1}(k_1,k_2)\Gamma_{t_3-t_2}(k_1,k_2) \bigg]  \nonumber \\
& \leq & \frac{C(\tau-b)}{n^4}\sum_{t_1,t_2=1}^{nb} \sum_{t_3,t_4=1}^{nb} \bigg| \sum_{k_1,k_2=1}^{p} \text{Cum}(\tilde{X}_{k_1t_1}, \tilde{X}_{k_1t_2}, \tilde{X}_{k_2t_3}, \tilde{X}_{k_2t_4}) \bigg| \nonumber \\
&&  + \frac{C(\tau-b)}{n^4}\sum_{t_1,t_2=1}^{nb} \sum_{t_3,t_4=1}^{nb}  |\text{Tr}(\Gamma_{t_3-t_1}\Gamma_{t_2-t_4}) + \text{Tr}(\Gamma_{t_4-t_1}\Gamma_{t_2-t-3}) | \nonumber \\
& \leq & \frac{C(\tau-b)}{n^3}\sum_{u,v,w=-\infty}^{\infty} \bigg|\sum_{k_1,k_2=1}^{p} \text{Cum}(\tilde{X}_{k_1 0}, \tilde{X}_{k_1 u}, \tilde{X}_{k_2 v}, \tilde{X}_{k_2 w})  \bigg|  + \frac{C(\tau-b)}{n^2} \sum_{u,v=-\infty}^{\infty} |\text{Tr}(\Gamma_u \Gamma_v)|. \nonumber
\end{eqnarray}
Note that,
\begin{eqnarray}
\sum_{u,v=-\infty}^{\infty} |\text{Tr}(\Gamma_u \Gamma_v)| & \leq & p\sum_{u,v=-\infty}^{\infty} ||\Gamma_u||_2 ||\Gamma_v||_2   \leq  p \left( \sum_{u=-\infty}^{\infty} ||\Gamma_u||_2\right)^2 \leq p\gamma_p^4. \nonumber 
\end{eqnarray}
Let $\Delta = E(\eta_{kt}^4)-3$.  Therefore,
\begin{eqnarray}
&& \text{Cum}(\tilde{X}_{k_1 0}, \tilde{X}_{k_1 u}, \tilde{X}_{k_2 v}, \tilde{X}_{k_2 w}) \nonumber \\
&=& \sum_{k_1,k_2=1}^{p} \text{Cum} \bigg[\sum_{j=0}^{\infty} \sum_{k_3=1}^{p} A_j (k_1,k_3) \eta_{-j,k_3}, \sum_{j=0}^{\infty} \sum_{k_4=1}^{p} A_j(k_1,k_4)\eta_{u-j,k_4}, \nonumber \\
&& \hspace{2 cm} \sum_{j=0}^{\infty} \sum_{k_5=1}^{p} A_j(k_2,k_5)\eta_{v-j,k_5}, \sum_{j=0}^{\infty} \sum_{k_6=1}^{p} A_{j}(k_2,k_6) \eta_{w-j,k_6} \bigg] \nonumber \\
&=& \sum_{j=0}^{\infty} \sum_{k_1,k_2=1}^{p} \sum_{k=1}^{p} (A_j(k_1,k)A_{j+u}(k_1,k)A_{j+v}(k_2,k)A_{j+w}(k_2,k)) \Delta \nonumber \\
&=& \Delta \sum_{j=0}^{\infty} \sum_{k=1}^{p} \bigg[(A_j^{\prime} A_{u+j})(k,k) (A_{j+v}^{\prime} A_{j+w})(k,k) \bigg]. \nonumber 
\end{eqnarray}
Hence, 
\begin{eqnarray}
&& \bigg| \sum_{k_1,k_2=1}^{p} \text{Cum}(\tilde{X}_{k_1 0}, \tilde{X}_{k_1 u}, \tilde{X}_{k_2 v}, \tilde{X}_{k_2 w}) \bigg| \leq Cp \sum_{j=0}^{\infty} ||A_j||_2 ||A_{u+j}||_2 ||A_{v+j}||_2 ||A_{w+j}||_2. \nonumber 
\end{eqnarray}
Finally, 
\begin{eqnarray}
\text{V} \left(\frac{1}{n}\sum_{k=1}^{p} A_{1k}(b) \right) &\leq & \frac{C(\tau-b)p}{n^2} \bigg(\frac{\gamma_p^4}{n} \bigg) + \frac{C(\tau-b)p}{n^2} \gamma_p^4 
\leq  \frac{C(\tau-b)p}{n^2} \gamma_p^4. \nonumber
\end{eqnarray}
Similarly, 
\begin{eqnarray}
\text{V} \left( \frac{1}{n} \sum_{k=1}^{p} A_{2k}(b) \right), \text{V} \left( \frac{1}{n} \sum_{k=1}^{p} A_{3k}(b) \right) \leq \frac{C}{n} \frac{(\tau-b)p}{n} \gamma_p^4. \nonumber 
\end{eqnarray}
\noindent Therefore, 
\begin{eqnarray}
\text{V} \left( \frac{1}{n} \sum_{k=1}^{p} A_{k}(b) \right) \leq \frac{C}{n} \frac{(\tau-b)p}{n} \gamma_p^4. \label{eqn: VARA}
\end{eqnarray}
\noindent Next consider $D_k$. For $nb+1 \leq t \leq n\tau$, 
\begin{eqnarray}
X_{kt} - \hat{\mu}_{2k}(b) &=& (X_{kt}-\mu_{1k}) - (\hat{\mu}_{2k}(b) - \frac{\tau -b}{1-b} \mu_{1k} - \frac{1-\tau}{1-b} \mu_{2k}) +  \frac{1-\tau}{1-b}(\mu_{1k}-\mu_{2k}) \nonumber \\
&=& \tilde{X}_{kt} - \frac{1}{n(1-b)} \sum_{t=nb+1}^{n} \tilde{X}_{kt} + \frac{1-\tau}{1-b} (\mu_{1k}-\mu_{2k})\ \ \ \text{and} \nonumber \\
X_{kt} - \hat{\mu}_{1k}(\tau) &=& (X_{kt}-\mu_{1k}) - (\hat{\mu}_{1k}(\tau) - \mu_{1k})  = \tilde{X}_{kt} - \frac{1}{n\tau} \sum_{t=1}^{n\tau} \tilde{X}_{kt}. \nonumber 
\end{eqnarray}
Let,
\begin{eqnarray}
\tilde{\mu}_{1k}(\tau) &=& \frac{1}{n\tau} \sum_{t=1}^{n\tau} (X_{kt}-E(X_{kt}),\ \ \ \tilde{\mu}_{1k}(b) = \frac{1}{nb} \sum_{t=1}^{nb} (X_{kt}-E(X_{kt}), \nonumber \\
\tilde{\mu}_{2k}(b) &=& \frac{1}{n(1-b)} \sum_{t=nb+1}^{n} (X_{kt}-E(X_{kt}),\ \  \  \tilde{\mu}_{2k}(\tau) = \frac{1}{n(1-\tau)} \sum_{t=n\tau+1}^{n} (X_{kt}-E(X_{kt}),\nonumber \\
\tilde{\mu}_{k}(\tau, b) &=& \frac{1}{n(\tau-b)}\sum_{t=nb+1}^{n\tau} (X_{kt}-E(X_{kt})). \nonumber 
\end{eqnarray}
Thus
\begin{eqnarray}
D_k(b) &=& \sum_{t=nb+1}^{n\tau} \bigg[(X_{kt}-\hat{\mu}_{2k}(b)^2 - (X_{kt} - \hat{\mu}_{1k}(\tau))^2 \bigg] \nonumber \\
&=& \sum_{t=nb+1}^{n\tau} \bigg[ \left(\frac{1}{n(1-b)}\sum_{t=nb+1}^{n}\tilde{X}_{kt} \right)^2 + \left( \frac{1-\tau}{1-b} (\mu_{1k}-\mu_{2k})\right)^2 - \left( \frac{1}{n\tau} \sum_{t=1}^{n\tau} \tilde{X}_{kt}\right)^2 \nonumber \\
&& - 2\tilde{X}_{kt}\left(\frac{1}{n(1-b)}\sum_{t=nb+1}^{n}\tilde{X}_{kt} \right) + 2\tilde{X}_{kt} \left( \frac{1-\tau}{1-b} (\mu_{1k}-\mu_{2k})\right) \nonumber \\
&& - \frac{2}{n(1-b)} \frac{1-\tau}{1-b} \left( \sum_{t=nb+1}^{n}\tilde{X}_{kt}\right) (\mu_{1k}-\mu_{2k}) + 2\tilde{X}_{kt} \left(\frac{1}{n\tau} \sum_{t=1}^{n\tau}\tilde{X}_{kt} \right) \bigg] \nonumber \\
&=& n(\tau-b) \bigg[(\tilde{\mu}_{2k}(b))^2 + \left( \frac{1-\tau}{1-b} (\mu_{1k}-\mu_{2k})\right)^2 - (\tilde{\mu}_{1k}(\tau))^2 \nonumber \\
&& -2\tilde{\mu}_k(b,\tau)\tilde{\mu}_{2k}(b) + 2\tilde{\mu}_k (b,\tau)\frac{1-\tau}{1-b} (\mu_{1k}-\mu_{2k}) \nonumber \\
&& -\frac{2}{n(1-b)} \frac{1-\tau}{1-b} \tilde{\mu}_{2k}(b) (\mu_{1k}-\mu_{2k}) + 2\tilde{\mu}_{k}(\tau,b)\tilde{\mu}_{1k}(\tau) \bigg] \nonumber \\
&=& \sum_{i=1}^{7} D_{ik}(b),\ \ \text{say}. 
\end{eqnarray}
Now, 
\begin{eqnarray}
&& E\left(\frac{1}{n}\sum_{k=1}^{p} D_{1k}(b)\right) = \frac{1}{n} \sum_{k=1}^{p} n(\tau-b) E\left( \frac{1}{n(1-b)} \sum_{t=nb+1}^{n} \tilde{X}_{kt} \right)^2  \nonumber \\
& \geq & C \frac{(\tau-b)}{n^2} \sum_{k=1}^{p} \sum_{t_1,t_2=nb+1}^{n} \text{Cov}(\tilde{X}_{kt},\tilde{X}_{kt_2}) = C\frac{\tau-b}{n^2} \sum_{k=1}^{p} \sum_{t_1,t_2=nb+1}^{n} \Gamma_{t_2-t_1}(k,k) \nonumber \\
& \geq & -C\frac{\tau-b}{n} \sum_{u=-\infty}^{\infty} |\text{Tr}(\Gamma_{|u|})| \geq -C\frac{(\tau-b)p}{n} \gamma_p^2.
\end{eqnarray}
Similarly, 
\begin{eqnarray}
E\left( \frac{1}{n}\sum_{k=1}^{p} D_{3k}(b)\right), E\left(\frac{1}{n}\sum_{k=1}^{p}D_{4k}(b) \right),  E\left( \frac{1}{n} \sum_{k=1}^{p} D_{7k}(b) \right) \geq - C\frac{(\tau-b)p}{n} \gamma_p^2. 
\end{eqnarray}
\noindent Also, $E\left(\frac{1}{n}\sum_{k=1}^{p} D_{5k}(b) \right) = E\left(\frac{1}{n}\sum_{k=1}^{p} D_{6k}(b) \right) =0$,  $E\left(\frac{1}{n}\sum_{k=1}^{p} D_{2k}(b) \right) \geq C(\tau-b)||\mu_1-\mu_2||_2^2$. 
\vskip 2pt
\noindent Thus, 
\begin{eqnarray} \label{eqn: expD}
E\left(\frac{1}{n}\sum_{k=1}^{p} D_{k}(b) \right) \geq  -C(\tau-b)\frac{p}{n} \gamma_p^2 + c(\tau-b) ||\mu_1-\mu_2||_2^2.
\end{eqnarray}

\noindent Now, using similar calculations as in  $\text{V}\left(\frac{1}{n}\sum_{k=1}^{p}A_{1k}(b) \right)$, we have 
\begin{eqnarray}
\text{V}\left( \frac{1}{n}\sum_{k=1}^{p} (D_{1k}(b)+D_{3k}(b)+D_{4k}(b)+D_{7k}(b))\right) \leq \frac{C}{n} \frac{(\tau-b)p}{n}\gamma_p^4. \nonumber 
\end{eqnarray}
\noindent Moreover, $\text{V}\left(\frac{1}{n}\sum_{k=1}^{p} D_{2k}(b)\right) =0$ and  
\begin{eqnarray}
&& \text{V}\left( \frac{1}{n}\sum_{k=1}^{p} D_{5k}(b)\right) = \frac{4(\tau-b)^2 (1-\tau)^2}{(1-b)^2} \text{V} \left( \sum_{k=1}^{p} (\mu_{1k}-\mu_{2k})\tilde{\mu}_{k}(b,\tau)\right)  \nonumber \\
&=& \frac{4(1-\tau)^2}{n^2 (1-b)^2} \sum_{k_1,k_2=1}^{p} \sum_{t_1,t_2=nb+1}^{n\tau}(\mu_{1k_1}-\mu_{2k_1})(\mu_{1k_2}-\mu_{2k_2})\Gamma_{t_2-t_1}(k_1,k_2) \nonumber \\
&=& \frac{4(1-\tau)^2}{n^2(1-b)^2} \sum_{u=-n(\tau-b)+1}^{n(\tau-b)-1} (n(\tau-b)-u)(\mu_1-\mu_2)^\prime \Gamma_u (\mu_1-\mu_2)) \nonumber \\
& \leq & \frac{C(\tau-b)}{n} \sum_{u=-\infty}^{\infty} ||\mu_1-\mu_2||_2^2 ||\Gamma_u||_2 \nonumber \\
& \leq & \frac{C}{n}(\tau-b)||\mu_1 - \mu_2||_2^2 \gamma_p^2. \nonumber 
\end{eqnarray}
Similarly, $\text{V}\left( \frac{1}{n}\sum_{k=1}^{p} D_{6k}(b)\right) \leq \frac{C}{n}(\tau-b)||\mu_1-\mu_2||_2^2\gamma_p^2$.

\noindent Thus
\begin{eqnarray} \label{eqn: VARD}
\text{V}\left( \frac{1}{n}\sum_{k=1}^{p}D_{k}(b)\right) &\leq & \frac{C}{n}\frac{(\tau-b)p}{n}\gamma_p^4 + \frac{C}{n}(\tau-b)||\mu_1-\mu_2||_2^2\gamma_p^2 \nonumber \\
& \leq &\frac{C}{n} (\tau-b) ||\mu_1-\mu_2||_2^2 \gamma_p^2. 
\end{eqnarray}

\noindent Next, for $n\tau+1 \leq t \leq n$, we have
\begin{eqnarray}
X_{kt} -\hat{\mu}_{2k}(b) &=& (X_{kt} - \mu_{2k}) -  \left( \hat{\mu}_{2k}(b) - \frac{\tau-b}{1-b} \mu_{1k} -\frac{1-\tau}{1-b}\mu_{2k}\right) +\frac{\tau-b}{1-b}(\mu_{2k}-\mu_{1k}), \nonumber \\
X_{kt} -\hat{\mu}_{2k}(b) &=&  (X_{kt}-\mu_{2k}) - (\hat{\mu}_{2k}(\tau) - \mu_{2k}). \nonumber
\end{eqnarray}

\noindent Thus
\begin{eqnarray}
B_k &=& \sum_{t=n\tau +1}^{n} \bigg[(\tilde{X}_{kt} -\tilde{\mu}_{2k}(b) +\frac{\tau-b}{1-b}(\mu_{2k}-\mu_{1k}))^2 - (\tilde{X}_{kt}-\tilde{\mu}_{2k}(\tau))^2  \bigg] \nonumber \\
&=& \sum_{t=n\tau +1}^{n} \bigg[ (\tilde{\mu}_{2k}(b))^2 + \left( \frac{\tau-b}{1-b}\right)^2 (\mu_{2k}-\mu_{1k})^2 - (\tilde{\mu}_{2k}(\tau))^2 \nonumber \\
&& - 2\tilde{X}_{kt}\tilde{\mu}_{2k}(b)+ 2\tilde{X}_{kt} \left( \frac{\tau-b}{1-b}\right) (\mu_{2k}-\mu_{1k}) - 2\tilde{\mu}_{2k}(b) \frac{\tau-b}{1-b} (\mu_{2k}-\mu_{1k}) + 2\tilde{X}_{kt} \tilde{\mu}_{2k}(\tau)\bigg] \nonumber \\
&=& n(1-\tau) \bigg[ (\tilde{\mu}_{2k}(b))^2 + \left(\frac{\tau-b}{1-b} \right)^2 (\mu_{2k}-\mu_{1k})^2 -(\tilde{\mu}_{2k}(\tau))^2 - 2\tilde{\mu}_{2k}(\tau)\tilde{\mu}_{2k}(b) \nonumber \\
&& + 2 \left(\frac{\tau-b}{1-b} \right)\tilde{\mu}_{2k}(\tau)(\mu_{2k}-\mu_{1k}) -  2 \left(\frac{\tau-b}{1-b} \right)\tilde{\mu}_{2k}(b)(\mu_{2k}-\mu_{1k}) + 2(\tilde{\mu}_{2k}(\tau))^2\bigg] \nonumber \\
&=& n(1-\tau) \bigg[(\tilde{\mu}_{2k}(b)-\tilde{\mu}_{2k}(\tau))^2 + \left(\frac{\tau-b}{1-b}\right)^2 (\mu_{2k}-\mu_{1k})^2 \nonumber \\
&& + 2\left( \frac{\tau-b}{1-b}\right) \tilde{\mu}_{2k}(\tau) (\mu_{2k}-\mu_{1k}) - 2\tilde{\mu}_{2k}(b) \left( \frac{\tau-b}{1-b}\right) (\mu_{2k}-\mu_{1k}) \bigg]. \nonumber 
\end{eqnarray}
\noindent Note that
\begin{eqnarray}
\tilde{\mu}_{2k}(b) - \tilde{\mu}_{2k}(\tau) &=& \frac{1}{n(1-b)} \sum_{t=nb+1}^{n} \tilde{X}_{kt} - \frac{1}{n(1-\tau)} \sum_{t=n\tau +1}^{n} \tilde{X}_{kt} \nonumber \\
&=& \frac{1}{n}\left(\frac{1}{1-b}-\frac{1}{1-\tau} \right) \sum_{t=n\tau +1}^{n} \tilde{X}_{kt} + \frac{1}{n(1-b)} \sum_{t=nb+1}^{n\tau} \tilde{X}_{kt}  \nonumber \\
&=& - \frac{\tau-b}{1-b} \tilde{\mu}_{2k}(\tau) + \frac{\tau-b}{1-b} \tilde{\mu}_k (b,\tau). \nonumber 
\end{eqnarray}
\noindent Therefore, 
\begin{eqnarray}
B_k &=& n(1-\tau) \bigg[\left( \frac{\tau-b}{1-b}\right)^2 \tilde{\mu}_{2k}^2(\tau) + \left( \frac{\tau-b}{1-b}\right)^2 (\tilde{\mu}_k(b,\tau))^2 + \left( \frac{\tau-b}{1-b}\right)^2 (\mu_{2k}-\mu_{1k})^2  \nonumber \\
&& - 2\left( \frac{\tau-b}{1-b}\right) ^2\tilde{\mu}_{2k}(\tau) \tilde{\mu}_{k}(b,\tau) + 2\left( \frac{\tau-b}{1-b}\right) \tilde{\mu}_{2k}(\tau)(\mu_{2k}-\mu_{1k}) - 2\tilde{\mu}_{2k}(b) \left( \frac{\tau-b}{1-b}\right) (\mu_{2k}-\mu_{1k})   \bigg] \nonumber \\
&=& \sum_{i=1}^{6} B_{ik}(b),\ \ \ \text{say}. \nonumber 
\end{eqnarray}

Now, using similar calculations as in $E\left(\frac{1}{n} \sum_{k=1}^{p} A_k(b)\right)$, we have
\begin{eqnarray}
&& E \left( \frac{1}{n}\sum_{k=1}^{p} (B_{1k}(b)+B_{2k}(b)+B_{4k}(b)) \right)  \geq - \frac{C(\tau-b)p}{n} \gamma_p^2,  \nonumber \\
&& E  \left( \frac{1}{n}\sum_{k=1}^{p} (B_{5k}(b) + B_{6k}(b))\right) = 0, \nonumber \\
&& E\left(\frac{1}{n}\sum_{k=1}^{p} B_{3k}(b) \right) = (1-\tau) \left( \frac{\tau-b}{1-b}\right)^2 \sum_{k=1}^{p} (\mu_{1k}-\mu_{2k})^2   \geq   C(\tau-b)||\mu_1-\mu_2||_2^2. \nonumber
\end{eqnarray}

Thus   
\begin{eqnarray} \label{eqn: expB}
E \left(\frac{1}{n}\sum_{k=1}^{p}B_{k}(b) \right) \geq C(\tau-b)||\mu_1-\mu_2||_2^2 - \frac{C(\tau-b)p}{n} \gamma_p^2. 
\end{eqnarray}

Moreover, using similarly calculations as in $\text{V}\left(\frac{1}{n}\sum_{k=1}^{p}D_k(b) \right)$,
\begin{eqnarray} \label{eqn: VARB}
\text{V}\left( \frac{1}{n}\sum_{k=1}^{p}B_k(b)\right) \leq \frac{C}{n}\frac{(\tau-b)p}{n} \gamma_p^4 + \frac{C}{n} (\tau-b) ||\mu_1-\mu_2||_2^2 \gamma_p^2 
 \leq \frac{C}{n} (\tau-b) ||\mu_1-\mu_2||_2^2 \gamma_p^2.
\end{eqnarray}

Hence, by (\ref{eqn: expA}), (\ref{eqn: expD}) and (\ref{eqn: expB}), we have
\begin{eqnarray}
E(L(b) - L(\tau)) &=& E\left(\frac{1}{n} \sum_{k=1}^{p} (A_{k}(b)+B_{k}(b)+D_{k}(b)) \right) \nonumber \\
&\geq & C(\tau-b) ||\mu_1-\mu_2||_2^2 - \frac{C(\tau-b)p}{n}\gamma_p^2 \nonumber \\
& \geq & C(\tau-b)||\mu_1-\mu_2||_2^2. \nonumber 
\end{eqnarray}

\noindent Also, by (\ref{eqn: VARA}), (\ref{eqn: VARB}) and (\ref{eqn: VARD}), we have $$\text{V}\left(\frac{1}{n} \sum_{k=1}^{p} (A_{k}(b)+B_{k}(b)+D_{k}(b))  \right) \leq \frac{C}{n}(\tau-b) ||\mu_1-\mu_2||_2^2 \gamma_p^2.$$ 
This proves (\ref{eqn: lsecon1}) and (\ref{eqn: lsecon2final}) and hence Theorem \ref{thm: conrate}.

\begin{remark} \label{rem: baigamma}
It is easy to observe from the proof of Theorem \ref{thm: conrate} that $\gamma_p^2$ serves as an upper bound of  $\sum_{u=-\infty}^{\infty} |\frac{1}{p} \text{Tr}(\Gamma_{|u|})|$ or of similar quantities. Now under cross-sectional independence i.e. when $A_{j,p} = \text{Diag}\{a_{j,k}:\ 1\leq k\leq p\}\ \forall j \geq 0$, then $\tilde{\gamma}_p^2  := (\sup_{1\leq k \leq p} \sum_{j=0}^{\infty} |a_{j,k}|)^2$ is turned out to be an smaller upper bound and $\gamma_p$ can be replaced by $\tilde{\gamma}_p$ throughout the proof. As a consequence the conclusion of Theorem \ref{thm: conrate} continues to hold under weaker (SNR$^\prime$) $\frac{n \tilde{\gamma}^{-2}_p}{p} ||\mu_1 - \mu_2||_2^2 \to \infty$ compared to SNR. 
\end{remark}

\subsection{Proof of Theorem \ref{thm: asympdist}} \label{subsec: asympdist}
\noindent \textbf{Proof of (a)}. Note that $P(\hat{\tau}_{n} \neq \tau_n) = P(|\hat{\tau}_{n} -\tau_n| \geq n^{-1}) \to 0$ 
since $\gamma_p^{-2} ||\mu_1 - \mu_2||_2 \to \infty$ and by Theorem \ref{thm: conrate}, $n \gamma_p^{-2} ||\mu_1 - \mu_2||_2^2(\hat{\tau}_{n}-\tau_n) = O_{P}(1)$. 
\vskip 10pt
\noindent For Theorem \ref{thm: asympdist}(b) and (c),  we use Lemma \ref{lem: wvandis1}. 
To employ Lemma \ref{lem: wvandis1}, we consider 
 $\mathbb{M}_n(h) = n\gamma_p^{-2}(L(b) - L(\tau))$ where $b = \tau +n^{-1}\gamma_p^{2} ||\mu_1 - \mu_2||_2^{-2} h$ and $h \in \mathbb{R}$. 
We shall find weak limit of $n\gamma_p^{-2} (L(b)-L(\tau))$ when $\gamma_p^{-2} ||\mu_1 - \mu_2||_2^2 \to 0$ or $c>0$.  We assume $b<\tau_n$. Similar arguments work when $b \geq \tau_n$. 

\noindent From the previous calculations, it is easy to see that
\begin{eqnarray}
n\gamma_p^{-2} (L(b)-L(\tau))  &=& \gamma_p^{-2} \sum_{k=1}^{p} \bigg[A_{k}(b) + B_{k}(b) +\sum_{i=1}^{7}D_{ik}(b) \bigg] \nonumber
\end{eqnarray}
\noindent where
\begin{eqnarray}
\bigg|E\left( \gamma_p^{-2} \sum_{k=1}^{p} A_{k}(b) \right) \bigg| & \leq & C|\tau - b|p, \nonumber \\
\bigg| E\left( \gamma_p^{-2} \sum_{k=1}^{p} B_{k}(b) \right) \bigg| & \leq & C(\tau-b)^2 ||\mu_1-\mu_2||_2^2 n \gamma_p^{-2} + C|\tau-b|p, \nonumber \\
\bigg| E\left( \gamma_p^{-2} \sum_{k=1}^{p}\sum_{\stackrel{i=1}{i \neq 2,5}}^{7} D_{ik}(b)\right)\bigg| & \leq & C|\tau-b|p, \nonumber \\
\text{V}\left(\gamma_p^{-2} \sum_{k=1}^{p} A_{k}(b) \right)  & \leq & C|\tau-b|p, \nonumber \\
\text{V}\left(\gamma_p^{-2} \sum_{k=1}^{p} B_{k}(b) \right)  & \leq & C(\tau-b)^2||\mu_1-\mu_2||_2^2 n \gamma_p^{-2} + C|\tau-b|p, \nonumber \\
\text{V} \left( \gamma_p^{-2} \sum_{k=1}^{p}\sum_{\stackrel{i=1}{i \neq 2,5}}^{7} D_{ik}(b) \right) &\leq & C(\tau-b)^2||\mu_1-\mu_2||_2^2 n \gamma_p^{-2} + C|\tau-b|p, \nonumber \\
\gamma_p^{-2} \sum_{k=1}^{p} D_{2k}(b) &=& n\gamma_p^{-2} (\tau-b)\left(\frac{1-\tau}{1-b} \right)^2 ||\mu_1-\mu_2||_2^2, \nonumber \\
\gamma_p^{-2} \sum_{k=1}^{p} D_{5k}(b) &=& \frac{2n\gamma_p^{-2}(\tau-b)}{n(\tau-b)}\frac{1-\tau}{1-b}\sum_{k=1}^{p} \sum_{t=nb+1}^{n\tau} (X_{kt}-EX_{kt})(\mu_{1k}-\mu_{2k}). \nonumber 
\end{eqnarray}
Thus taking $n\gamma_p^{-2}||\mu_1-\mu_2||_2^2 (\tau-b) = h$, where $\gamma_p^{-2}||\mu_1-\mu_2||_2^2 \to 0$ or $c>0$, we have
\begin{eqnarray}
\sup_{h \in \mathcal{C}} |n\gamma_p^{-2}(L(b)-L(\tau)) - \gamma_p^{-2} \sum_{k=1}^{p} (D_{2k}(b) + D_{5k}(b))| \stackrel{\text{P}}{\to} 0 \nonumber
\end{eqnarray}
for some compact set $\mathcal{C} \subset \mathbb{R}$. 
\vskip 5pt
\noindent \textbf{Proof of (b)}.
 To prove Theorem \ref{thm: asympdist}(b),  by Lemma \ref{lem: wvandis1}, it is enough to establish
\begin{align} \label{eqn: enoughlse2b}
\sup_{h \in \mathcal{C}} |\gamma_p^{-2} \sum_{k=1}^{p} (D_{2k}(b) + D_{5k}(b)) - |h| -  2B_h^*| \stackrel{\text{P}}{\to} 0, \nonumber
\end{align} 
as $\gamma_p^{-2}||\mu_1 - \mu_2||_2 \to 0$, and for all compact subsets $\mathcal{C}$ of $\mathbb{R}$.  

\noindent It is easy to see that
\begin{eqnarray}
\sup_{h \in \mathcal{C}} |\gamma_p^{-2} \sum_{k=1}^{p} D_{2k}(b)  - |h| | \stackrel{\text{P}}{\to} 0, \label{eqn: asymph}
\nonumber
\end{eqnarray}
as $\gamma_p^{-2}||\mu_1 - \mu_2||_2 \stackrel{\text{P}}{\to} 0$, and for all compact subsets $\mathcal{C}$ of $\mathbb{R}$.  
 
 \noindent Now note that
 \begin{eqnarray}
 && E (\gamma_p^{-2} \sum_{k=1}^{p} D_{5k}(b))^4  \leq  C \gamma_p^{-8} \sum_{k_1,k_2,k_3,k_4=1}^{p} \sum_{t_1,t_2,t_3,t_4=nb+1}^{n\tau} \left( \prod_{l=1}^{4} (\mu_{1k_l}-\mu_{2k_l}) \right) \nonumber \\
 & &  \sum_{j_1,j_2,j_3,j_4=0}^{\infty} \sum_{i_1,i_2,i_3,i_4=1}^{p} A_{j_1}(k_1,i_1) A_{j_2}(k_2,i_2) A_{j_3}(k_3,i_3)  A_{j_4}(k_4,i_4) E(\eta_{t_1-j_1,i_1} \eta_{t_2-j_2,i_2}\eta_{t_3-j_3,i_3} \eta_{t_4-j_4,i_4})  \nonumber \\
 &\leq & C \frac{(\sup_{k} |\mu_{1k}-\mu_{2k}| )^4 \gamma_p^{-4}}{||\mu_1-\mu_2||_2^4} \left(\sum_{j=0}^{\infty} \sum_{k=1}^{p}\sum_{i=1}^{p} A_{j}(k,i)\right)^4 \nonumber \\
 & \leq &  C \bigg[\frac{\sup_{k} |\mu_{1k}-\mu_{2k}| }{||\mu_1-\mu_2||_2} \bigg]^{4}. \nonumber 
 \end{eqnarray}
Therefore, by (A3) and Lyapunov's Central Limit Theorem, 
\begin{eqnarray}
\sup_{h \in \mathcal{C}} |\gamma_p^{-2} \sum_{k=1}^{p} D_{5k}(b)   -  2B_h^*| \stackrel{\text{P}}{\to} 0
\end{eqnarray}
where for all $h_1,h_2,\ldots,h_r \in \mathbb{R}$ and $r \geq 1$,
\begin{eqnarray}
(B_{h_1}^*,B_{h_2}^*,\ldots,B_{h_r}^*) \sim \mathcal{N}_{r}(0,\Sigma),\ \ \Sigma = ((\sigma_{h_ih_j}))_{1 \leq i,j \leq r}. \nonumber 
\end{eqnarray}
This completes the proof of Theorem \ref{thm: asympdist}(b). 
\vskip 10pt
\noindent \textbf{Proof of (c)}. It is easy to see that
\begin{eqnarray}
\gamma_p^{-2} \sum_{k=1}^{p} (D_{2k}(b) + D_{5k}(b)) &=& \gamma_{p}^{-2}\sum_{t=nb+1}^{n\tau} \sum_{k\in \mathcal{K}_0} [(X_{tk}- {\mu}_{2k})^2 - (X_{tk}-\mu_{1k})^2] \nonumber \\
&& \hspace{-2 cm}+ 2\gamma_p^{-2}\sum_{k \in \mathcal{K}_n}  \sum_{t=nb+1}^{n\tau} (\mu_{1k}-\mu_{2k}) \varepsilon_{kt} 
+ n\gamma_p^{-2} (\tau-b)\sum_{k \in \mathcal{K}_n}(\mu_{1k}-\mu_{2k})^2 + o_{\text{P}}(1). \nonumber \\
&=& M_{n}^{II}(h) + M_{n}^{Ia}(h) + M_{n}^{Ib}(h), \ \ \text{say}. \nonumber 
\end{eqnarray}
Note that $$\sup_{h \in \mathcal{C}} |M_{n}^{Ib}(h) - |h|c_1c^{-1}| \stackrel{\text{P}}{\to} 0.$$ Moreover, by (A4)-(A8),  $$\sup_{h \in \mathcal{C}} |M_{n}^{II}(h) -\sum_{t=0\wedge (h/c)}^{0\vee (h/c)} A_t^*| \stackrel{\text{P}}{\to} 0.$$
Also by (A6), (A7) and (A9)
$$\sup_{h \in \mathcal{C}} |M_{n}^{Ia}(h) - 2\sum_{t=0\wedge (h/c)}^{0\vee (h/c)} W_t^*| \stackrel{\text{P}}{\to} 0.$$
This completes the proof of Theorem \ref{thm: asympdist}(c). 
\vskip 5pt
\noindent Hence, Theorem \ref{thm: asympdist} is proved. 

\subsection{Proof of Propositions \ref{prop: 1} and \ref{prop: 2}} \label{subsec: prop}
Here we prove only Proposition \ref{prop: 2}. Similar proof works for Proposition \ref{prop: 1}. 
\vskip 2pt
\noindent (I) Note that
\begin{eqnarray}
&& \gamma_{p(n+1)}^{-4}(\mu_{1,p(n+1)} - \mu_{2,p(n+1)})^\prime \Gamma_{0,p(n+1)} (\mu_{1,p(n+1)} - \mu_{2,p(n+1)}) \nonumber \\
&=& \gamma_{p(n+1)}^{-4} \sum_{i,j,l=1}^{p(n+1)} \sum_{k=0}^{\infty} (\mu_{1i,p(n+1)} - \mu_{2i,p(n+1)})(\mu_{1j,p(n+1)} - \mu_{2j,p(n+1)}) A_{k}(i,l)A_k(l,j) \nonumber \\
&=& \gamma_{p(n+1)}^{-4} \sum_{l=1}^{p(n+1)} \sum_{i,j=1}^{p(n)} \sum_{k=0}^{\infty} (\mu_{1i,p(n+1)} - \mu_{2i,p(n+1)})(\mu_{1j,p(n+1)} - \mu_{2j,p(n+1)}) A_{k}(i,l)A_k(l,j) \nonumber \\
&& +\gamma_{p(n+1)}^{-4} \sum_{l=1}^{p(n+1)} \sum_{\stackrel{i>p(n)}{\text{or}\ j > p(n)}} \sum_{k=0}^{\infty} (\mu_{1i,p(n+1)} - \mu_{2i,p(n+1)})(\mu_{1j,p(n+1)} - \mu_{2j,p(n+1)}) A_{k}(i,l)A_k(l,j) \nonumber \\
&=& \gamma_{p(n+1)}^{-4} \sum_{l=1}^{p(n+1)} \sum_{i,j=1}^{p(n)} \sum_{k=0}^{\infty} \left((\mu_{1i}(n+1)-\mu_{1i}(n))+(\mu_{1i}(n)-\mu_{2i}(n))+(\mu_{2i}(n)-\mu_{2i}(n+1)) \right)\nonumber \\
&& \hspace{1 cm}\left((\mu_{1j}(n+1)-\mu_{1j}(n))+(\mu_{1j}(n)-\mu_{2j}(n))+(\mu_{2j}(n)-\mu_{2j}(n+1)) \right) A_{k}(i,l)A_k(l,j) \nonumber \\
&& +\gamma_{p(n+1)}^{-4} \sum_{l=1}^{p(n+1)} \sum_{\stackrel{i>p(n)}{\text{or}\ j > p(n)}} \sum_{k=0}^{\infty} (\mu_{1i,p(n+1)} - \mu_{2i,p(n+1)})(\mu_{1j,p(n+1)} - \mu_{2j,p(n+1)}) A_{k}(i,l)A_k(l,j). \nonumber 
\end{eqnarray} 
Thus
\begin{eqnarray}
&& | \gamma_{p(n+1)}^{-4} (\mu_{1,p(n+1)}-\mu_{2,p(n+1)})^\prime \Gamma_{0,p(n+1)} (\mu_{1,p(n+1)}-\mu_{2,p(n+1)}) \nonumber \\
&& \hspace{5 cm}- \gamma_{p(n)}^{-4}(\mu_{1,p(n)}-\mu_{2,p(n)})^\prime  \Gamma_{0,p(n)}(\mu_{1,p(n)}-\mu_{2,p(n)})| \nonumber \\
& \leq & \gamma_{p(n+1)}^{-4} \beta_{p(n+1)}^{2} p(n+1) \bigg[\sup_{j=1,2}\sup_{1 \leq i \leq p(n+1)} |\mu_{ji,p(n+1)}-\mu_{ji,p(n)}| \nonumber \\
&& \hspace{5 cm}+ \displaystyle{\sup_{1 \leq i \leq p(n+1)} |\mu_{1i,p(n)}-\mu_{2i,p(n)}|} \bigg]^2 \nonumber \\
&&+ (\gamma_{p(n+1)}^{-4}-\gamma_{p(n)}^{-4})\beta_{p(n+1)}^{2}p(n+1) \bigg[\sup_{j=1,2}\sup_{1 \leq i \leq p(n+1)} |\mu_{ji,p(n+1)}-\mu_{ji,p(n)}| \nonumber \\
&& \hspace{5 cm}+ \displaystyle{\sup_{1 \leq i \leq p(n+1)} |\mu_{1i,p(n)}-\mu_{2i,p(n)}|} \bigg] \nonumber \\
&& +\gamma_{p(n)}^{-4}\beta_{p(n)}^{2}(p(n+1)-p(n)) \bigg[ \sup_{1 \leq i \leq p(n)} |\mu_{1i}(n)-\mu_{2i}(n) |\bigg]^2 \nonumber \\
&& + \gamma_{p(n+1)}^{-4}\beta_{p(n+1)}^2 p(n+1) \bigg[ \sup_{1 \leq i \leq p(n+1)} |\mu_{1i}(n+1)-\mu_{2i}(n+1) | \bigg]^2 \nonumber \\
&& \to 0,\ \ \text{by (a),(b),(c),(g) and (h)}. \nonumber
\end{eqnarray}
This completes the proof of Proposition \ref{prop: 2}(I). 
\vskip 2pt
\noindent Proof of Proposition \ref{prop: 2}(II) follows from similar arguments in the proof of Proposition \ref{prop: 2}(I). 
\vskip 2pt
\noindent (III) It is easy to see that
\begin{eqnarray}
&&\gamma_{p(n+1)}^{-2}\Gamma_{0,p(n+1)}(k_1,k_2) = \gamma_{p(n+1)}^{-2}\sum_{k=0}^{\infty}\sum_{l=1}^{p(n+1)} A_{k}(k_1,l) A_{k}(l,k_2) \nonumber \\
&=& (\gamma_{p(n+1)}^{-2} - \gamma_{p(n)}^{-2}) \sum_{k=0}^{\infty}\sum_{l=1}^{p(n+1)} A_{k}(k_1,l) A_{k}(l,k_2) 
+ \gamma_{p(n)}^{-2} \sum_{k=0}^{\infty}\sum_{l=p(n)+1}^{p(n+1)} A_{k}(k_1,l) A_{k}(l,k_2)\nonumber \\
&& \hspace{7 cm}+ \gamma_{p(n)}^{-2}\Gamma_{0,p(n)}(k_1,k_2). \nonumber 
\end{eqnarray}
Hence, by (a) and (c), we have
\begin{eqnarray}
|\gamma_{p(n+1)}^{-2}\Gamma_{0,p(n+1)}(k_1,k_2) - \gamma_{p(n)}^{-2}\Gamma_{0,p(n)}(k_1,k_2)| 
& \leq & |\gamma_{p(n+1)}^{-2} - \gamma_{p(n)}^{-2}| \gamma_{p(n+1)}^2 + |p(n+1)-p(n)| \nonumber\\
&& \hspace{5 cm} \to 0. \nonumber 
\end{eqnarray}
This proves Proposition \ref{prop: 2}(III). 
\vskip 2pt
\noindent This completes the proof of Proposition \ref{prop: 2}.

\subsection{Proof of Theorem \ref{thm: adap}} \label{subsec: adap}
Without loss of generality, assume $h>0$. Note that
\begin{eqnarray}
\hat{L}(h) - \hat{L}(0) &=& \frac{1}{n}\sum_{k=1}^{p}\sum_{t=n\hat{\tau}_n+1}^{n\hat{\tau}_n +nh} (2X_{kt,p,\text{ADAP}}-\hat{\mu}_{1k}-\hat{\mu}_{2k})(\hat{\mu}_{2k}-\hat{\mu}_{1k}). \nonumber
\end{eqnarray}
Let $E^*(\cdot) = E(\cdot|\{X_{tp}\})$ and $\text{V}^*(\cdot) = \text{V}(\cdot|\{X_{tp}\})$. Therefore,
\begin{eqnarray}
E^{*}(\hat{L}(h) - \hat{L}(0)) &=& h||\hat{\mu}_1-\hat{\mu}_2||_2^2, \nonumber \\
\text{V}^{*} (\hat{L}(h) - \hat{L}(0)) &\leq & C\frac{1}{n^2} \sum_{k_1,k_2}\sum_{t_1,t_2= n\hat{\tau}_n+1}^{n\hat{\tau}_n +h} (\hat{\mu}_{1k_1}-\hat{\mu}_{2k_1})(\hat{\mu}_{1k_1}-\hat{\mu}_{2k_1}) B_{l_n}(\hat{\Gamma}_{t_1-t_2})(k_1,k_2)  \nonumber \\
& \leq & \frac{h}{n} \gamma_p^2 ||\hat{\mu}_1-\hat{\mu}_2||_2^2. \nonumber
\end{eqnarray}
Hence by Lemma \ref{lem: wvan1} and similar arguments at the beginning of Section \ref{subsec: conrate}, we have
\begin{eqnarray}
n\gamma_p^{-2}||\hat{\mu}_1-\hat{\mu}_2||_2^2 \hat{h}_{\text{ADAP}} = O_{\text{P}}(1). \nonumber
\end{eqnarray}
Now by Lemma \ref{lem: adap}(c),
\begin{eqnarray}
n\gamma_p^{-2}||{\mu}_1-{\mu}_2||_2^2 \hat{h}_{\text{ADAP}} = O_{\text{P}}(1). \nonumber
\end{eqnarray}
This implies Theorem \ref{thm: adap}(a). 
\vskip 2pt
\noindent Now,
\begin{eqnarray}
&& n\gamma_p^{-2}(\hat{L}(h\gamma_p^{2}||\hat{\mu}_1-\hat{\mu}_2||_2^{-2}/n)-\hat{L}(0))\nonumber \\
& = & |h| + 2\gamma_p^{-2}\sum_{k=1}^{p}\sum_{t=n\hat{\tau}_n+1}^{n\hat{\tau}_n +h\gamma_p^{2}||\hat{\mu}_1-\hat{\mu}_2||_2^{-2}} (X_{kt,p,\text{ADAP}}-\hat{\mu}_{2k})(\hat{\mu}_{2k}-\hat{\mu}_{1k}) +o_{\text{P}}(1). \nonumber 
\end{eqnarray}
Note that $\{X_{kt,p,\text{ADAP}}\}$ are Gaussian. Hence Theorem \ref{thm: adap}(b) follows from Lemma \ref{lem: wvandis1} and Lemma \ref{lem: adap}(a)-(d) and (g). 
\vskip 3pt
\noindent A similar argument as in the proof of Theorem \ref{thm: asympdist}(c) and similar approximations as in the proof of Theorem \ref{thm: adap}(b) also work for Theorem \ref{thm: adap}(c) and hence we omit them. Hence, Theorem \ref{thm: adap} is established.

\end{document}